\documentclass[1 [leqno,11pt]{amsart}
\usepackage{amssymb, amsmath,amsmath,latexsym,amssymb,amsfonts,amsbsy, amsthm,mathtools,graphicx,CJKutf8,CJKnumb,CJKulem,color}
\usepackage[english]{babel}
\usepackage{breqn}
\usepackage{tikz}
\usepackage{bigints}
\usepackage{yfonts}
\usepackage{comment}
\usetikzlibrary{calc,hobby}
\usepackage{bbm}
\usepackage[shortlabels]{enumitem}
\usepackage{graphicx}
\usepackage{theoremref}
\newtheorem{theorem}{Theorem}[section]
\numberwithin{equation}{section}
\newtheorem{remark}[theorem]{Remark}
\newtheorem{prop}[theorem]{Proposition}
\newtheorem{corollary}[theorem]{Corollary}
\newtheorem{lemma}[theorem]{Lemma}

\usepackage{amssymb}
\usepackage{scalerel,stackengine}
\usepackage{blindtext}
\usepackage{xcolor}
\usepackage{mathrsfs}
\usepackage{ulem}
\usepackage{stmaryrd}
\usepackage{hyperref}
\usepackage{dsfont}
\hypersetup{
    colorlinks=true,
    linkcolor=blue,
    filecolor=magenta,      
    urlcolor=cyan,
}
\usepackage{scalerel,stackengine}
\allowdisplaybreaks[4]
\stackMath
\newcommand\sometext

\newcommand{\Mu}{M}

\newcommand{\By}{\mathrm{\textbf{y}}}
\newcommand{\JB}[1]{\langle #1 \rangle}
\newcommand{\A}{\mathcal{A}}
\newcommand{\J}{\mathcal{J}}
\newcommand{\M}{\mathcal{M}}

\newcommand{\CK}{\textbf{CK}}
\newcommand{\NR}{\textup{NR}}
\newcommand{\R}{\textup{R}}

\newcommand{\fr}{\frac}
\newcommand{\nn}{\nonumber}
\newcommand{\les}{\lesssim}

\newcommand{\norm}[1]{\left\lVert#1\right\rVert}
\newcommand*\circled[1]{\tikz[baseline=(char.base)]{
            \node[shape=circle,draw,inner sep=.0001pt] (char) {#1};}}

\stackMath
\newcommand\reallywidehat[1]{%
\savestack{\tmpbox}{\stretchto{%
  \scaleto{%
    \scalerel*[\widthof{\ensuremath{#1}}]{\kern-.6pt\bigwedge\kern-.6pt}%
    {\rule[-\textheight/2]{1ex}{\textheight}}
  }{\textheight}%
}{0.5ex}}%
\stackon[1pt]{#1}{\tmpbox}%
}

\title[stability of the Couette flow for the Stokes-Transport equation] {Asymptotic Stability of the two-dimensional Couette flow for the Stokes-transport equation in a finite channel}
\author{Daniel Sinambela}
\address[D. S.]{Department of Mathematics, New York University Abu Dhabi, Saadiyat Island, P.O. Box 129188, Abu Dhabi, United Arab Emirates.}
\email{dos2346@nyu.edu}

\author{Weiren Zhao}
\address[W. Z.]{Department of Mathematics, New York University Abu Dhabi, Saadiyat Island, P.O. Box 129188, Abu Dhabi, United Arab Emirates.}
\email{zjzjzwr@126.com, wz19@nyu.edu}

\author{Ruizhao Zi}

\address[R. Z.]{School of Mathematics and Statistics, and Key Laboratory of Nonlinear Analysis \& Applications (Ministry of Education), Central China Normal University, Wuhan,  430079,  P. R. China.}
\email{rzz@ccnu.edu.cn}

\begin{document}
\begin{abstract}
We study the Stokes-transport system in a two-dimensional channel with horizontally moving boundaries, which serves as a reduced model for oceanography and sedimentation. The density is transported by the velocity field, satisfying the momentum balance between viscosity, pressure, and gravity effects, described by the Stokes equation at any given time. Due to the presence of moving boundaries, stratified densities with the Couette flow constitute one class of steady states. In this paper, we investigate the asymptotic stability of these steady states. We prove that if the stratified density is close to a constant density and the perturbation belongs to the Gevrey-3 class with compact support away from the boundary, then the velocity will converge to the Couette flow as time approaches infinity. More precisely, we prove that the horizontal perturbed velocity decays as $\frac{1}{\langle t\rangle^3}$ and the vertical perturbed velocity decays as $\frac{1}{\langle t\rangle^4}$.
\end{abstract}
\maketitle

\section{Introduction}\label{Introduction}
We consider the two-dimensional Stokes-transport equation posed in the periodic channel $\Omega:=\mathbb{T}\times[0,1]$, 
\begin{equation}\label{governing eq}
 \begin{cases}
    \partial_t \varrho +u \cdot \nabla \varrho=0 ,\\
   -\Delta u +\nabla p= -\varrho e_2 ,\quad
    \nabla \cdot u=0,\\
    \varrho|_{t=0}=\varrho_{\rm in},
\end{cases} 
\end{equation}
which describes the evolution of an incompressible viscous fluid with inhomogeneous density. Here $\varrho$ is the density, $e_2=(0,1)\in \mathbb{R}^2$ is the vertical direction, and $u=(u^1,u^2)$ is the velocity field. It was shown in \cite{Grayer2023, hofer2018sedimentation, mecherbet2021sedimentation} that the Stokes-transport equation can be obtained as the homogenization limit of inertialess particles in a fluid satisfying Stokes equation or as a formal limit where the Prandtl number is infinite.
We refer to \cite{inversi2023lagrangian, Leblond2022, mecherbet2022few} and the references therein for the well-posedness results of \eqref{governing eq}. In this paper, we consider a horizontally moving boundary and impose the boundary condition that the fluid moves together with it, namely, 
\begin{align*}
    u^1|_{y=0}=0,\quad u^1|_{y=1}=1,\quad \text{and}\quad u^2|_{y=0,1}=0
\end{align*} 
With such a boundary condition, it is easy to see that for any stratified density profile $\underline{\varrho}(y)$ and the Couette flow $(y,0)$, 
\begin{equation}\label{steady states}
\underline{\varrho}:=\underline{\varrho}(y), \qquad \underline{u}:=(y,0),\qquad \underline{p}:=\int_{0}^y (-\underline{\varrho})(z)\;dz+\mathrm{constant},
\end{equation}
is a steady solution to \eqref{governing eq}.
In the present work, our main objective is to investigate the stability of the steady state \eqref{steady states}  in a perturbative way. Let us introduce the governing equations in terms of the perturbations $(U,\rho, P)$ away from the steady states $(\underline{u},\underline{\varrho},\underline{p})$. More precisely, we let $u=U+\underline{u},\; p=P+\underline{p}, \; \varrho=\rho+\underline{\varrho}$ with $U=(U^1,U^2)$. Thus, we obtain
\begin{equation}\label{governing eq under perturbation}
 \begin{cases}
    \partial_t \rho +U^2 \partial_y \underline{\varrho} + y \partial_x \rho +U \cdot \nabla \rho=0 ,\\
   -\Delta U +\nabla P= -\rho e_2 ,\\
    \nabla \cdot U=0,\quad U|_{y=0,1}=0,\\
    \rho|_{t=0}=\rho_{\rm in}.
\end{cases}
\end{equation}
Moreover, due to the incompressibility, there exists stream function $\psi$ such that $U=\nabla^\bot\psi=(-\partial_y\psi,\partial_x\psi)$ which is given by
\[
\Delta \psi=\omega:=\partial_xU^2-\partial_yU^1 \quad{\rm in}\ \ \mathbb{T}\times[0,1],\qquad \psi|_{y=0,1}=0\ \text{and}\ \partial_y\psi|_{y=0,1}=0.
\]

\subsection{Main Result}
This paper aims to study the long-time asymptotic behavior of perturbation of \eqref{steady states}. We state our main theorem below.

\begin{theorem}\label{main theorem} 
Fix $\kappa \in (0,\frac{1}{10}]$. There exist $\lambda_{b}>0$, $\delta_0>0$, such that for $0<\delta<\delta_0$, if the background density satisfies the following conditions:
\begin{enumerate}[start=1]
    \item (Compact Support) the background density has constant values near boundaries, namely, 
    \begin{align*}
        \textup{supp}\, \underline{\varrho}' \subset [3\kappa,1-3\kappa],
    \end{align*}
    \item (Regularity and Smallness) the following estimates holds
    \begin{align*}
\norm{\underline{\varrho}'}^2_{\mathcal{G}^{\lambda_b;\frac{1}{3}}}
=: \sum\limits_{k}\int_{\mathbb{R}}\Big|\widehat{\underline{\varrho}'}_k(\eta)\Big|^2 e^{2\lambda_b|k,\eta|^{\frac{1}{3}}}\;d\eta\leq \delta^2,
    \end{align*} 
\end{enumerate} 
then the stratified density $\underline{\varrho}$ with the Couette flow is asymptotically stable under suitable perturbations. 

More precisely, there exist $\lambda_b >\lambda_{\textup{in}}>\lambda_{\infty}>0$, and $\epsilon_0=\epsilon_0(\lambda_{\textup{in}},\lambda_{\infty},\lambda_b)\leq \frac{1}{2}$, for any $\epsilon \leq \epsilon_0$, if the initial density perturbation $\rho_{\rm in}$ satisfies 
\begin{align} 
&\textup{supp } \rho_{\rm in} \subset [2\kappa,1-2\kappa],\\
 &\norm{\rho_{\rm in}}^2_{\mathcal{G}^{\lambda_{\textup{in}};\frac{1}{3}}(\mathbb{T}\times [0,1])}\lesssim \epsilon^2,\\
&\int_{\mathbb{T}\times [0,1]} \rho_{\rm in}(x,y)\;dx\;dy=0,
\end{align}
then 
\begin{enumerate}[start=1]
    \item (Compact support) For all $t>0$, \textup{ supp} $\rho(t,x,y) \subseteq \mathbb{T} \times [1.5 \kappa,1-1.5\kappa]$.
    \item (Scaterring) There exists some $\rho_\infty\in \mathcal{G}^{\lambda_{\infty};\frac{1}{3}}$ with \textup{ supp} $\rho_\infty \subseteq \mathbb{T} \times [ \kappa,1-\kappa]$ such that for all  $t> 0$
    \begin{equation}\label{scattering}
    \norm{\rho(t,x+ty,y)-\rho_\infty(x,y)}_{\mathcal{G}^{\lambda_{\infty};\frac{1}{3}}(\mathbb{T}\times [0,1])}\lesssim\dfrac{\epsilon^2+\delta\epsilon}{\JB{t}^3}.
    \end{equation}
    \item  (Damping): The velocity field $U$ satisfies the following decay estimates:
    \begin{align}\label{inviscid damping}
    \JB{t}\norm{U^2(t,x,y)}_{L^2(\mathbb{T}\times [0,1])}+\norm{U^1(t,x,y)}_{L^2(\mathbb{T}\times [0,1])}&\lesssim \dfrac{\epsilon}{\JB{t}^3}.
    \end{align} 
\end{enumerate}
\end{theorem}

We state a few important remarks concerning our theorem.
\begin{remark}
The Gevrey radii $\lambda_b$ and $\lambda_{\mathrm{in}}$ are determined in the proof. See more details in Remark \ref{rem-<1}, \eqref{lm-infty},  \eqref{up-A-remainder}, and \eqref{restriction-lambda_b}. It is reasonable that the solution is less regular than the background density. 

    Our result holds also for smoother background density, namely, the asymptotic stability holds when $\underline{\varrho'}$ is in Gevrey-$m_1$ and $\rho_{\mathrm{in}}$ is in Gevrey-$m_2$ with $1< m_1\leq m_2<3$. Here, the prime notation, $(\cdot)'$,  denotes the derivative in $y$.
    
    If the background density $\underline{\varrho'}$ is in Gevrey-$m_1$ and $\rho_{\mathrm{in}}$ is in Gevrey-$m_2$ with $1<m_1<m_2<3$, then $\lambda_b$ can take any positive value. Meanwhile if the initial density perturbation $\rho_{\mathrm{in}}$ is in Gevrey-$m_2$ with $1<m_2<3$, then $\lambda_{\mathrm{in}}$ can take any positive value as well. 
\end{remark}

\begin{remark}
    The Gevrey-3 of the perturbation appears to be optimal. The forthcoming paper by the same authors will discuss the asymptotic instability and the optimality of this regularity assumption.
\end{remark}

\begin{remark}
    By applying our method, in the whole space setting $\Omega=\mathbb{T}\times \mathbb{R}$, one can prove that if the background density $\underline{\varrho}(y)$ satisfies  $\|\underline{\varrho'}(y)-a_0\|_{\mathcal{G}^{\lambda_b;\frac{1}{3}}}\leq \delta$ for sufficiently small $\delta>0$, then $\underline{\varrho}(y)$ with the Couette flow is still asymptotically stable under Gevrey-3 perturbations. Here $a_0\in \mathbb{R}$ can be any constant. 
\end{remark}

\begin{remark}
    The asymptotic stability of stratified density profiles without shear flows is well-studied (see \cite{dalibard2023long}). For stratified flow, the stable stratification (monotonicity) which asserts that lighter fluid sits on top of denser fluid is known to be a common yet crucial assumption leading to stability phenomena. We refer to several works where such monotonicity assumption plays an important role in stabilizing the system.  We refer to \cite{leblond2023well, KieselevYao, kiselevyaoyaopark2022,Sinambela2022ExistenceAS} for some instability results. Such an instability is closely related to the Rayleigh–Taylor instability, which also occurs when ripples are excited on the interface between a heavy fluid (e.g., water) sitting atop a lighter fluid (e.g., air) in a gravitational field. Additionally, we refer to \cite{castro2019global, elgindi2017asymptotic} for similar stability mechanisms in other fluid models. 

    However, with shear flow, the main stability mechanism is performed via mixing. This is completely different from the mechanism employed in the aforementioned papers. When shear is present, monotonic stratification is redundant in attaining stability. In the present paper, our main result confirms that strong Couette flow stabilizes the system. More precisely, our stability persists in the flow where the denser fluid sits on top of the lighter one i.e., when the density profile is no longer decreasing. This mixing phenomenon is closely related to inviscid damping in ideal fluid. We refer to \cite{BedrossianMasmoudi2015, chen2023nonlinear, IonescuJia2020, IonescuJia2023, masmoudiZhao2020, zhao2023inviscid} for recent nonlinear inviscid damping results of shear flows in ideal fluid.

    It is worth pointing out that, in the three-dimensional case \cite{SZZ3D}, we take advantage of both stability mechanisms. We use the first one to stabilize the streak solution (zero-mode) and prove the asymptotic stability of the Couette flow.
\end{remark}

\begin{remark}
    The compactness assumption of $\underline{\varrho}'$ and the initial perturbation $\rho_{\mathrm{in}}$ is to prevent boundary effects. Under this setting, we can use the Fourier analysis. In fact, this phenomenon can be shown to persist throughout the flow. One key reason why this occurs is due to the decay of the vertical direction of the velocity field, see the proof of Proposition~\ref{bootstrap}.
\end{remark}


\subsection{Notations}\label{sec-notation}
 In this subsection, we introduce some notations used throughout the present work. First, we define the common $l^1$ norm for frequency $(k,\eta)$ which takes the form $|k,\eta|=|k|+|\eta|$. In addition to that, let us explicitly define the underlying function spaces our work is based upon. Namely, the Gevrey-$\frac{1}{s}$ space with Sobolev correction which together with its norm is defined by
\[
\mathcal{G}^{\lambda,\sigma;s}:=\{f\in L^2: \norm{f}_{\mathcal{G}^{\lambda,\sigma;s}}<\infty\}, \text{ where }\norm{f}_{\mathcal{G}^{\lambda,\sigma;s}}^2=\sum_{k}\int_{\eta}|\widehat{f}_k(\eta)|^2\JB{k,\eta}^{2\sigma}e^{2\lambda|k,\eta|^s}\;d\eta.
\]
Notice that, in the statement of the Theorem~\ref{main theorem} we set $\sigma=0$. This applies to many parts of the paper where we drom $\sigma$ from the notation and cling to $\mathcal{G}^{\lambda;s}$.  

For any given scalar $x$, we define the standard Japanese bracket $\JB{x}=\sqrt{1+|x|^2}$. Further, let $f$ and $g$ be any given functions, we define the commonly used notation $f\lesssim g$ to mean that there exists a pure constant (independent of any parameter) $C>0$ such that $f\leq C g$. Throughout the work, we also write $f\approx g$ to mean that there exists $C>0$ (again, independent of any parameter) such that $C^{-1}f\leq g \leq C f$. The lower case $c$ or ${\bf c}$ in this paper is reserved to denote a small positive constant less than 1. We write $f_0:=\frac{1}{2\pi}\int_{\mathbb{T}}f(x,y)\;dx$; this takes the average of the function $f$ in the horizontal direction. The $f_0$ goes by the name ``the zero mode of $f$". In addition, we introduce the other counterpart of $f_0$, namely $f_{\neq}:=f-f_0$. One can think of it as the projection off of the zero mode of $f$ (or nonzero mode of $f$). 
We use $\tilde{f}$ to denote the Fourier transformation of the function $f$ in the $x-$direction and $\widehat{f}$ for the Fourier transformation in both $x$ and $y$ directions. 
For $\eta \geq 0$, we let $E(\eta):=\lfloor \eta \rfloor \in \mathbb{Z}$, which denotes the integer part of $\eta$. For any fixed $\eta \in \mathbb{R}$ and $1\leq |k| \leq E(|\eta|^{\frac{1}{3}})$ with $\eta k\geq 0$, let us denote $t^{\pm}_{k,\eta}:=\frac{|\eta|}{|k|}\pm\frac{|\eta|}{|2k|^3}$. With this in mind, we define the following critical intervals
\[
\textup{I}_{k,\eta}
=\begin{cases} 
[t^{-}_{k,\eta},t^{+}_{k,\eta}],\qquad \text{ for } \eta k \geq 0, \text{ and } 1\leq |k| \leq E(|\eta|^{\frac{1}{3}}), \\
\emptyset \qquad \qquad \qquad \text{ otherwise.}
\end{cases}
\]

For the purpose of capturing mild resonances, additionally, we define a slightly larger time interval denoted by 
\[
\tilde{\textup{I}}_{k,\eta}:=\Big[\frac{2|\eta|}{2|k|+1}, \frac{2|\eta|}{2|k|-1}\Big]\supseteq \textup{I}_{k,\eta}. 
\]
Lastly, we use  $\textup{I}_{k,\eta}^c$ to denote the complement of $\textup{I}_{k,\eta}$, similarly for $\tilde{\textup{I}}_{k,\eta}$.
\subsection{Plan of the paper}
Let us now outline the structure of the paper. In Section~\ref{idea}, we introduce the linear coordinate transformation, toy model, and the time-dependent Fourier multipliers. Section~\ref{Main energy est} is devoted to presenting the main energy, bootstrap hypotheses, and Proposition~\ref{bootstrap}. By assuming the Proposition~\ref{bootstrap}, we then prove Theorem~\ref{main theorem}, particularly, the scattering result. Following that, the remaining sections are subject to proving Proposition~\ref{bootstrap}. More precisely, Section~\ref{properties of multipliers} is designed to provide simplified and modified multiplier estimates recorded in a series of Lemmas. Next, Section~\ref{stokes} presents some important estimates that hinge on the Fourier kernel estimate derived in Appendix~\ref{Appendix A}. In Section~\ref{nonlinear interactions}, we provide the upper bound needed to handle the nonlinear interactions. Finally, in Appendix~\ref{Appendix A} we compute explicitly the upper bound of the Fourier kernel. Appendix~\ref{Aux} records some auxiliary estimates and Appendix~\ref{estimates on Theta and J} provides fundamental bounds pertaining to the weight $\Theta$ and multiplier $\mathcal{J}$ which are heavily used in Section~\ref{properties of multipliers}.

\section{Ideas of the proof}\label{idea}

In the present section, we outline the main ideas used in the present work. 
First, we study the linearized equation and discuss the stability mechanism. By dropping the nonlinear term and the small linear terms in \eqref{governing eq under perturbation} and taking the Fourier transform in $x$, we obtain 
\begin{align}\label{gov eq Fourier x}
    \left\{\begin{aligned}
    &\partial_t \tilde{\rho}_k(t,y) + ik y  \tilde{\rho}_k(t,y) =0 ,\\
    &\Delta_k^2 \tilde{\psi}_k(t,y)= ik \tilde{\rho}_k(t,y) ,\quad
    \tilde{\psi}_k|_{y=0,1}=\partial_y\tilde{\psi}_k|_{y=0,1}=0,\\
    &\tilde{\rho}_k(t,y)|_{t=0}=\tilde{\rho}_{\rm in}(k,y).
    \end{aligned}\right.
\end{align}
A direct calculation yields $\tilde{\rho}_k(t, y)=e^{-ik y t}\tilde{\rho}_{\rm in}(k,y)$. We then have 
\begin{align*}
    \|\tilde{\psi}_k\|_{L^2}
    &=\sup_{\varphi\in C_0^{\infty}(0,1):~\|\Delta_k^2\varphi\|_{L^2}\leq 1}\left|\int_0^1\tilde{\psi}_k(y)\Delta_k^2\varphi dy\right|\\
    &=\sup_{\varphi\in C_0^{\infty}(0,1):~\|\Delta_k^2\varphi\|_{L^2}\leq 1}\left|\int_0^1e^{-ik y t}\tilde{\rho}_{\rm in}(k,y)\varphi(y) dy\right|\\
    &\lesssim \frac{1}{k t}\sup_{\varphi\in C_0^{\infty}(0,1):~\|\Delta_k^2\varphi\|_{L^2}\leq 1}\left|\int_0^1e^{-ik y t}\partial_y(\tilde{\rho}_{\rm in}(k,y)\varphi(y)) dy\right|\\
    &\lesssim \frac{1}{(k t)^4}\sup_{\varphi\in C_0^{\infty}(0,1):~\|\Delta_k^2\varphi\|_{L^2}\leq 1}\left|\int_0^1e^{-ik y t}\partial_y^4(\tilde{\rho}_{\rm in}(k,y)\varphi(y)) dy\right|\\
    &\lesssim \frac{1}{(k t)^4}\|\tilde{\rho}_{\rm in}\|_{H^4}\|\varphi\|_{H^4}\lesssim \frac{1}{(k t)^4}\|\tilde{\rho}_{\rm in}\|_{H^4}. 
\end{align*}
Such a decay estimate is of the same spirit as the inviscid damping for the linearized Euler equation around Couette flow. The above duality argument is not well-adapted, however, to the nonlinear interaction. The damping is due to a mixing effect of the transport term $\partial_t+y\partial_x$. It is natural to introduce the following change of coordinate:
\begin{equation}\label{linear change of coordinate}
(x,y) \mapsto ((z:=x-ty), y).
\end{equation}
In this new coordinate system, we define the following new unknowns
\[
V(t,z,y):=U(t,x,y),\;\theta(t,z,y):=\rho(t,x,y), \; \phi(t,z,y):=\psi(t,x,y) .
\]
Additionally, due to the linear change coordinate, we also obtain
\[
\begin{aligned}
&(\partial_x,\partial_y) \mapsto (\partial_z, \partial_y-t\partial_z) =:\nabla_L,\\&
\partial_{xx}+\partial_{yy}\mapsto \partial_{zz}+(\partial_{y}-t\partial_{z})^2=:\Delta_L.
\end{aligned}
\]
By taking the curl of the second equation \eqref{governing eq under perturbation} and using the incompressibility assumption, we may conclude that the zero mode of the perturbed velocity field is zero for $t\geq 0$, namely
\begin{equation}\label{zero mode of U}
U^{1}_0(t,y)=0,\qquad U^{2}_0(t,y)=0.
\end{equation}
In the new coordinate system, we rewrite \eqref{governing eq under perturbation} as follows
\begin{equation}\label{transport equation in new coord}
\begin{cases}
    \partial_t \theta(t,z,y) + \nabla^{\perp}_{z,y}\phi(t,z,y) \cdot \nabla_{z,y}\theta(t,z,y) =-\partial_z \phi(t,z,y)  \underline{\varrho}' ,\\
   -\Delta_L V +\nabla_L P= -\theta e_2,\quad
   V=\nabla^{\bot}_{L}\phi,
\end{cases}
\end{equation}
where $\nabla_{z,y}=(\partial_z,\partial_y)$. For the remaining portion of the paper, we will work with the system \eqref{transport equation in new coord}, and write $\nabla_{z,y} $ for short as $\nabla $ without causing confusion whenever the argument involves $(z,y)$. 

In light of \eqref{transport equation in new coord}, let us drop the nonlinear term and the small linear term, ignore the boundary effect at this stage, and formally take the Fourier transform in both $z, y$. Then we have $\widehat{\theta}_k(t,\xi)=\widehat{\theta_{\mathrm{in}}}(k,\xi)$ and 
\begin{align}\label{eq: fourier-damping-est}
    \widehat{\phi}_k(t, \xi)=\frac{ik \widehat{\theta}_k(t,\xi)}{((\xi-kt)^2+k^2)^2}=\frac{ik \widehat{\theta_{\mathrm{in}}}(k,\xi)}{((\xi-kt)^2+k^2)^2}. 
\end{align}
We can obtain the $\frac{1}{\JB{t}^4}$ decay rate of $\|\phi\|_{L^2}$ by using the fact that $\JB{t-\xi/k}\JB{\xi/k}\gtrsim \JB{t}$ and a uniform $H^4$ bounds of $\theta(t)$. If $\xi k>0$ and $\xi$ is very large relative to $k$, then the stream-function amplifies by a factor $\frac{\xi^4}{k^4}$ at a critical time given by $t_{\mathrm{c}}=\frac{\xi}{k}$. Such a transient growth is similar to the Orr mechanism \cite{orr1907stability}. 

\subsection{Nonlinear interactions and Growth Mechanisms}

The main potential growth stems from the nonlinear interactions between the velocity field and the density gradient, namely $\nabla^\bot\phi\cdot \nabla\theta $. It is worth noting that one has to pay regularity to get decay in time, say for the velocity field. This leads us to focus on the worst scenario where $\nabla^\bot\phi$ is at high frequency and $\theta$ is at low frequency. Furthermore, ignoring the interaction with good derivative $\partial_z$ on $\phi$,  the evolution of $\theta$ reduces to
\[
\partial_t \theta -\partial_y\phi_{\neq}\partial_z\theta_{lo}=0,
\]
where $\theta_{lo}$ is the low frequency part of $\theta$.
Recalling that $\Delta^2_L \phi= \partial_z \theta$, formally taking the Fourier transform, and hence on the Fourier side the toy model reads
\[
\partial_t \widehat{\theta}(t,k,\eta)=\frac{1}{2\pi} \sum_{l}\int \frac{il\xi (k-l)\widehat{\theta}(t,l,\xi)}{(l^2+(\xi-lt)^2)^2} \widehat{\theta}_{lo}(t,k-l,\eta-\xi)\;d\xi.
\]
Keep in mind that $\widehat{\theta}_{lo}$ is concentrated at low frequencies,  let us now consider the following discrete model where $\eta=\xi$, $l=k\pm1$, and $\widehat{\theta}_{lo}=O(\varsigma)$: 
\begin{equation}\label{discrete toy model}
\partial_t \widehat{\theta}(t,k,\eta)=\varsigma \sum_{l=k\pm1} \frac{|l\eta |}{(l^2+(\eta-lt)^2)^2} \widehat{\theta}(t,l,\eta).
\end{equation}
We have taken the absolute value in the coefficient because we are only interested in an upper bound. For any fixed $k$, it is easy to see that the coefficient $\fr{|k\eta|}{(k^2+(\eta-kt)^2)^2}$ is large when $|\eta|\gtrsim |k|^3$ and $|t-\fr{\eta}{k}|\les \fr{|\eta|}{|k|^3}$. This leads to the notion {\it critical time} $t=\fr{\eta}{k}$ and {\it critical/resonant interval}, i.e., the  neighborhood of the critical time $\fr{\eta}{k}$ with radius $O(\fr{|\eta|}{|k|^3})$. Now let us assume that $t$ lies in such a critical interval, so that for all $l\ne k$, there holds $|t-\fr{\eta}{l}|\gtrsim \fr{|\eta|}{|k|^2}$. Consequently, for $l$ close to $k$, but $l\ne k$, we have
\[
\fr{|l\eta|}{(l^2+(\eta-lt)^2)^2}=\fr{|\eta|}{|l|^3}\fr{1}{(1+|t-\fr{\eta}{l}|^2)^2}\les \fr{|k|^5}{|\eta|^3}\les \fr{|k|^3}{|\eta|}.
\]

Applying these observations to \eqref{discrete toy model}, we derive the following toy model:
\begin{equation}\label{toy}
    \begin{cases}
        \partial_t \widehat{\theta}_{\rm R}\approx \varsigma\fr{|k|^3}{|\eta|}\widehat{\theta}_{\rm NR},\\
        \partial_t \widehat{\theta}_{\rm NR}\approx\varsigma\fr{|\eta|}{|k|^3}\fr{1}{1+(t-\fr{\eta}{k})^2}\widehat{\theta}_{\rm R},
    \end{cases}
\end{equation}
where `R' stands for `resonant', and `NR' stands for `nonresonant'. Fortunately, the toy model that captures the main growth is the same as that in \cite{MasmoudiBelkacemZhao2022Inv} which allows us to use the same Fourier multipliers the authors used. 

\subsection{Stokes equation}
In studying the decay mechanism and capturing the nonlinear growth, the first standard step in the analysis is to recast the problem in terms of the Fourier variables. However, due to the presence of boundary here, employing the Fourier transform directly to the stream function $\phi$ is not possible. Thanks to the fact that $\theta$ is supported away from boundaries, the contribution of $\phi$ only exists in the interior of the domain. In light of that, we can replace $\phi$ in \eqref{transport equation in new coord} by $\phi\chi$ where $\chi$ is a cut-off function defined in \eqref{cutoff function}. It is obvious that $\phi\chi$ has better Fourier properties than $\phi$. In order to understand the relationship between $\widehat{\phi\chi}$ and $\widehat{\theta}$, we study the Stokes equation, obtain the kernel, and study the kernel on the Fourier side. More precisely, in Lemma \ref{rep of psi_k}, by solving the boundary value problem of an ODE equation, we get the Green's function of the Stokes equation in the periodic channel. In Lemma \ref{Estimate fourier psi chi}, we obtain the kernel $\mathfrak{G}$ which satisfies the following equation
\begin{align}\label{eq: phichi=Gtheta}
    \widehat{(\phi_k \chi)}(\eta)=\int_{\mathbb{R}} \mathfrak{G}(t,k,\eta,\zeta) \widehat{\theta}_k(\zeta) \;d\zeta. 
\end{align}
In addition to that, the kernel has the following estimate
\begin{equation*}
    |\mathfrak{G}(t,k,\eta,\zeta)| \lesssim \min \biggl\{ \frac{|k|}{(k^2+(\zeta-kt)^2)^2},  \frac{|k|}{(k^2+(\eta-kt)^2)^2} \biggr\}e^{-\lambda_M|\eta-\zeta|^s}.
\end{equation*}
We present more precise statements of such estimates and other related ones in Section~\ref{stokes}. Moreover, for more detailed calculations, we refer readers to  Appendix~\ref{Appendix A}. 

Now let us point out that the difference between \eqref{eq: fourier-damping-est} and \eqref{eq: phichi=Gtheta} does not essentially change the growth mechanism in the nonlinear interaction. This is because of the following two main observations: (1) the kernel $\mathfrak{G}(t,k,\eta,\zeta)$ does not mix the information
from different frequencies in $z$; (2) the integral in \eqref{eq: phichi=Gtheta} mixes some frequency information of $\theta$ due to the kernel $\mathfrak{G}$, but the essential part is coming from $\eta\approx \zeta$ due to the fast decay of the kernel $e^{-\lambda_M|\eta-\zeta|^s}$. For each fixed frequency in $z$, the critical times are still very sensitive to the frequency in $y$. However, the critical time interval where the growth happens is not sensitive to the frequency in $y$, which allows us to use the same toy model even under the boundary effect. 

\subsection{Key Fourier multipliers}
In this section, we introduce the two key Fourier in \cite{MasmoudiBelkacemZhao2022Inv}, which capture the growth in the nonlinear interactions. 
\subsubsection{Construction of Weight $\Theta$}
The weight $\Theta_m(t,\eta)$ is constructed to capture some growth during the time interval $t\in \Big[ t_{E(|\eta|^{\frac{1}{3}}),|\eta|},2|\eta|\Big]$. For $\eta>0$ and $t\geq 1$, we first introduce two functions $\Theta_{\NR}(t,\eta)$ and $\Theta_{\R}(t,\eta)$ which describe the different growths of solutions with different wave numbers $m$ at the same time interval $\text{I}_{k,\eta}$ with $k=1,2,....,E(|\eta|^{\frac{1}{3}})$. Before proceeding any further, we remark that $\Theta_{\rm NR}(t,\eta)\equiv1$ if $|\eta|\le1$. For $|\eta|>1$,  let $\Theta_{\NR}$ be a non-decreasing function with respect to the time variable for which $\Theta_{\NR}(t,\eta)\equiv 1$ for all $t\geq 2|\eta|$. More precisely, the construction of $\Theta_{\NR}$ is done as follows. Let $t\in \text{I}_{k,\eta}$, the function $\Theta_{\NR}$ satisfies the following equations
\begin{align}\label{def-ThetaNR}
\begin{cases}
\Theta_{\NR}(t,\eta)=\bigg(\frac{k^3}{2\eta}\bigg[1+\beta_{k,\eta}|t-\frac{\eta}{k}|\bigg]\bigg)^{C_1} \Theta_{\NR}(t^+_{k,\eta},\eta),\qquad \text{for all } t\in \bigg[\frac{\eta}{k},t^{+}_{k,\eta}\bigg],\\ \\
\Theta_{\NR}(t,\eta)=\bigg(\bigg[1+\alpha_{k,\eta}|t-\frac{\eta}{k}|\bigg]\bigg)^{-1-C_1} \Theta_{\NR}(\frac{\eta}{k},\eta),\qquad \text{for all } t\in \bigg[t^{-}_{k,\eta},\frac{\eta}{k}\bigg],
\end{cases}
\end{align}
where $C_1$ is a positive constant depending on $\varsigma$ in \eqref{toy} (see more details in \cite{BedrossianMasmoudi2015,MasmoudiBelkacemZhao2022Inv}), $\beta_{k,\eta}$ is chosen such that $\frac{k^3}{2\eta}\bigg[1+\beta_{k,\eta}\frac{\eta}{(2k)^3}\bigg]=1$ and $\alpha_{k,\eta}$ is chosen so that  $\Theta_{\NR}(t^{+}_{k,\eta},\eta)=(\frac{2\eta}{k^3})^{1+2C_1}\Theta_{\NR}(t^{-}_{k,\eta},\eta)$. Consequently, these two yield the expressions for $\beta_{k,\eta}$ and $\alpha_{k,\eta}$, namely
\[
    \alpha_{k,\eta}=\beta_{k,\eta}=16-\dfrac{(2k)^3}{\eta}.
\]
Clearly, $\alpha_{k,\eta}=\beta_{k,\eta}\approx1$ for all $1\le k\le E(\eta^{\frac{1}{3}})$. Now, on the time interval $\text{I}_{k,\eta}$, we define $\Theta_{\R}$ which relies on the behavior of $\Theta_{\NR}$ on  $\text{I}_{k,\eta}$. More precisely, we write it as follows
\begin{equation}\label{Theta R and Theta NR}
\begin{cases}
\Theta_{\R}(t,\eta)=\bigg(\frac{k^3}{2\eta}\bigg[1+\beta_{k,\eta}|t-\frac{\eta}{k}|\bigg]\bigg) \Theta_{\NR}(t,\eta),\qquad \text{for all } t\in \bigg[\frac{\eta}{k},t^{+}_{k,\eta}\bigg],\\ \\
\Theta_{\R}(t,\eta)=\bigg(\frac{k^3}{2\eta}\bigg[1+\alpha_{k,\eta}|t-\frac{\eta}{k}|\bigg]\bigg)\Theta_{\NR}(t,\eta),\qquad   \text{for all } t\in \bigg[t^{-}_{k,\eta},\frac{\eta}{k}\bigg].
\end{cases}
\end{equation}
In addition to that, on the time interval $\tilde{\text{I}}_{k,\eta}\setminus \text{I}_{k,\eta}$, the value of $\Theta_{\R}$ is determined by $\Theta_{\NR}$ in the following manner
\[
\begin{aligned}
    &\Theta_\R(t,\eta)=\Theta_\NR(t,\eta)=\Theta_\NR(\frac{2\eta}{2k-1},\eta), \qquad \text{for all } t\in \bigg[t^{+}_{k,\eta}, \frac{2\eta}{2k-1}\bigg],\\&
    \Theta_\R(t,\eta)=\Theta_\NR(t,\eta)=\Theta_\NR(t^{-}_{k,\eta},\eta), \qquad \quad \; \; \text{for all } t\in \bigg[\frac{2\eta}{2k+1},t^{-}_{k,\eta}\bigg]. 
\end{aligned}
\]
Notice also that via the expressions of $\alpha_{k,\eta}$ and $\beta_{k,\eta}$, we inherently get $\Theta_\R(t^{\pm}_{k,\eta},\eta)=\Theta_\NR(t^{\pm}_{k,\eta},\eta)$ and $\Theta_\R(\frac{\eta}{k},\eta)=\frac{k^3}{2\eta}\Theta_\NR(\frac{\eta}{k},\eta)$. Via \eqref{Theta R and Theta NR}, we may deduce the following approximation relations
\begin{equation}\label{DE for Theta}
\left\{
\begin{aligned}
&\partial_t \Theta_\R \approx \dfrac{k^3}{\eta}\Theta_\NR, \\&
\partial_t \Theta_\NR\approx \dfrac{\eta}{k^3(1+|t-\frac{\eta}{k}|^2)}\Theta_\R.
\end{aligned}\right.
\end{equation}
Component-wise for $t\geq 1$, $\Theta_m(t,\eta)$ is given as follows
\begin{equation}\label{defintion of Theta k}
\Theta_m(t,\zeta)=
\begin{cases}
   \Theta_\R(t,\zeta),  \qquad \qquad  \ \, t\in \text{I}_{m,\zeta},\\
   \Theta_\NR(t,\zeta), \qquad \qquad  {\rm otherwise}.
\end{cases}
\end{equation}
We refer to Appendix~\ref{estimates on Theta and J} for more properties of the multiplier $\Theta_m(t,\zeta)$. 

\subsubsection{Construction of Weight $\Lambda$}
In the construction of $\Theta_k(t,\eta)$ above, we do not assign any growth  in $\Theta_k(t,\eta)$ when $t\in\tilde{\rm I}_{k,\eta}\backslash{\rm I}_{k,\eta}$ or $t<t_{E(|\eta|^{\fr{1}{3}}),\eta}$. In order to balance the potential growth of the perturbations in these time intervals, we construct another weight $\Lambda$. Throughout the process we let $t \in \tilde{\text{I}}_{k,\eta}:=\left[\frac{2|\eta|}{2|k|+1},\frac{2|\eta|}{2|k|-1}\right]$ for $1 \leq |k| \leq E(|\eta|^{\frac{1}{3}}) \leq E(|\eta|^{\frac{1}{3}})+1\leq ...\leq E(|\eta|^{\frac{2}{3}})$. Notice that we have taken the upper bound of $|k|$ to be slightly larger; $ E(|\eta|^{\frac{2}{3}})$ instead of $E(|\eta|^{\frac{1}{3}})$, allowing for longer time interval than that of $\text{I}_{k,\eta}$. 

Let $\Lambda$ be an increasing function such that $\Lambda(t,\eta)\equiv 1$ for $t \geq 2 \eta$ where $|\eta|>1$. Additionally, for all $|k|\geq 1$, we assume that $\Lambda(\frac{2\eta}{2k-1},\eta)$ was given. We demand $\Lambda$ for $1 \leq |k|\leq E(|\eta|^{\frac{1}{3}})$ on the time interval $\left[\frac{2\eta}{2|k|+1},\frac{2\eta}{2|k|-1}\right]$ to satisfy the following equation
\begin{equation}\label{evolution of Lambda 1}
\partial_t \Lambda(t,\eta)=\frac{1}{20}\dfrac{1}{1+|t-\frac{\eta}{k}|^2}\Lambda(t,\eta), \qquad \Lambda|_{t=\frac{2\eta}{2k-1}}=\Lambda(\frac{2\eta}{2k-1},\eta).
\end{equation}
Furthermore, for $k= E(|\eta|^{\frac{1}{3}})+1, ..., E(|\eta|^{\frac{2}{3}})$ the weight $\Lambda$ on the time interval $\left[\frac{2|\eta|}{2|k|+1},\frac{2|\eta|}{2|k|-1}\right]$ should satisfy
\begin{equation}\label{evolution of Lambda 2}
\partial_t \Lambda(t,\eta)=\frac{1}{20}\dfrac{ \frac{\eta}{k^3}}{1+|t-\frac{\eta}{k}|^2}\Lambda(t,\eta), \qquad \Lambda|_{t=\frac{2\eta}{2k-1}}=\Lambda(\frac{2\eta}{2k-1},\eta).
\end{equation}

The next section is devoted to discussing several multipliers used throughout the paper. Via these multipliers, we then define the time-dependent norm as well as the energy estimate.
\section{Main Energy Estimates}\label{Main energy est}
In order to arrive at our main goal, we make  use of the time-dependent norm displayed below
\[
\norm{\A(t,\nabla)\theta}_2^2=\sum_k\int_\eta\A^2_k(t,\eta)|\widehat{\theta}_k(t,\eta)|^2\;d\eta,
\]
where the multiplier in the above integral takes the form
\begin{equation}\label{Multiplier A}
\A_k(t,\eta)=e^{\lambda(t)|k,\eta|^{1/3}}\langle k,\eta \rangle^\sigma \J_k(t,\eta) \M_k(t,\eta).
\end{equation} For readability sake, we, in many places later, drop the time dependency from the multiplier above and simply display $\A_k(\eta)$ instead of $\A_k(t,\eta)$. 

The exponent $\lambda(t)$ is given by
\begin{equation}
\lambda(t):=\lambda_\infty+\fr{\tilde{\delta}}{(1+t)^a}, \qquad\text{with } \tilde{\delta}>0,
\end{equation}
where $a$ and $\lambda_\infty$  are chosen sufficiently small and large. We require that
\begin{equation}\label{lm-infty}
\lambda_\infty \ge100(\mu+C_0),
\end{equation}
with   $\mu$  determined in Lemma \ref{lem-total growth} and $C_0\ge 6\pi$. Therefore, one can choose $\lambda_\infty<\lambda_{\rm in}<\lambda_b$ so that $\lambda(0)<\lambda_{\rm in}$.

The role of the multipliers $\mathcal{J}_k$ and $\mathcal{M}_k$ is to deal with the nonlinear interaction arising from the transport equation of the density. They are equipped with weights that are carefully designed to control terms that are ``growing". Explicitly, we write them below 
\begin{equation}\label{Multiplier components}
\J_k(t,\eta)=\tilde{\J}_k(t,\eta)+e^{\mu|k|^{1/3}},\qquad 
\M_k(t,\eta)=\tilde{\M}_k(t,\eta)+e^{\fr{C_0}{2}|k|^{1/3}},
\end{equation}
where 
\[
\tilde{\J}_k(t,\eta)=\dfrac{e^{\mu|\eta|^{1/3}}}{\Theta_k(t,\eta)},  \text{ and } \tilde{\M}_k(t,\eta)= \dfrac{e^{\fr{C_0}{2}|\eta|^{1/3}}}{\Lambda(t,\eta)}.
\]

For compactness of notation, we define two additional multipliers:
\[
\begin{aligned}
&\mathcal{A}^{\Theta}_k(t,\eta):=e^{\lambda(t)|k,\eta|^{1/3}}\langle k,\eta \rangle^\sigma \dfrac{e^{\mu |\eta|^{1/3}}}{\Theta_k(t,\eta)} \M_k(t,\eta),\\&
\mathcal{A}^{\Lambda}_k(t,\eta):=e^{\lambda(t)|k,\eta|^{1/3}}\langle k,\eta \rangle^\sigma \J_k(t,\eta) \dfrac{e^{\fr{C_0}{2}|\eta|^{1/3}}}{\Lambda(t,\eta)} .
\end{aligned}
\]
Both $\mathcal{A}^{\Theta}_k$
 and $\mathcal{A}^{\Lambda}_k$ satisfy: $\mathcal{A}^{\Theta}_k \leq \mathcal{A}$ and $\mathcal{A}^{\Lambda}_k \leq \mathcal{A}$. And if $|k|\leq |\eta|$ then $\mathcal{A}^{\Theta}_k \gtrsim \mathcal{A}$ and $\mathcal{A}^{\Lambda}_k \gtrsim \mathcal{A}$. In preparation for  later computations and estimates, we define the ``Cauchy-Kovalevskaya" $(\CK)$ terms
\[
\begin{aligned}
&\CK_{\lambda}(t)=-\dot{\lambda}(t)\norm{|\nabla|^{1/6} \A \theta}^2_{L^2}=-\dot{\lambda}(t)\sum_k \norm{|k,\eta|^{1/6}\A_k(t,\eta)\widehat{\theta}_k(t,\eta)}^2_{L^2_\eta},\\&
\CK_{\Theta}(t)=\sum_k \int_\eta \dfrac{\partial_t\Theta_k(t,\eta)}{\Theta_k(t,\eta)} \mathcal{A}^{\Theta}_k(t,\eta)\A_k(t,\eta)|\widehat{\theta}_k(t,\eta)|^2 \;d\eta,\\&
\CK_{\Lambda}(t)=\sum_k \int_\eta \dfrac{\partial_t\Lambda(t,\eta)}{\Lambda(t,\eta)}\mathcal{A}^{\Lambda}_k(t,\eta)\A_k(t,\eta)|\widehat{\theta}_k(t,\eta)|^2 \;d\eta.
\end{aligned}   
\]

We are now in a position to define our main energy.  Due to a simpler coordinate transformation (compared to that of \cite{BedrossianMasmoudi2015, MasmoudiBelkacemZhao2022Inv}) used in the present work, the associated energy appears to be more amenable for analysis. More precisely, our time-dependent energy functional takes the form
\[
\mathcal{E}(t)=\dfrac{1}{2}\norm{\A(t,\nabla)\theta}_2^2.
\]

The local well-posedness of \eqref{governing eq} in the Gevrey-3 class can be proved by a classic iteration process for transport equations. We omit the details. We would like to mention that due to the smallness assumptions, the vertical velocity remains small. Thus, the support of the perturbed density remains away from the boundary. More precisely, we have for $t\in [0,1]$, it holds that
\begin{align}
    &\mathcal{E}(t)\leq \tilde{C}\epsilon^2,\\
    & \textup{supp }\theta(t)\subset [2\kappa-\tilde{C}\epsilon, 1-2\kappa+\tilde{C}\epsilon]
\end{align}
with some $\tilde{C}>0$ independent of $\epsilon$. 

\subsection{Bootstrap hypotheses}
Here, we assume some control over a number of quantities for all time $t>1$ which will be referred to throughout the paper as the bootstrap hypotheses. The size of $\kappa$ will be the same as that mentioned in the statement of Theorem~\ref{main theorem}. To that that end, we list the hypotheses below,

\begin{enumerate}[start=1,label={\bfseries B\arabic*:}]
\item Energy Bound: $\mathcal{E}(t)\leq 4 C\epsilon^2$.
\item Compact Support of $\theta(t)$: $\textup{supp }\theta(t)\subset [1.5\kappa-8C\epsilon, 1-1.5\kappa+8C\epsilon]$. 
\item $\textup{\CK}$ Integral Estimates:
\[
\int_{1}^t \big[\CK_{\lambda}(s)+\CK_{\Lambda}(s)+\CK_{\Theta}(s)\big]\;ds\leq 8C\epsilon^2.
\]
\end{enumerate}

Next is the main proposition of our work. Under the aforementioned bootstrap hypotheses, we prove that during the evolution, the bootstrap estimates can be improved by obtaining better upper bounds than those mentioned above.

\begin{prop}[Bootstrap]\label{bootstrap} There exists $\epsilon_0 \in(0,\frac{1}{2})$ which depends on $\lambda_{\textup{in}},\lambda_\infty,s,$ and $\sigma$ such that provided $\epsilon \in (0,\epsilon_0)$ and the bootstrap hypotheses \textup{\textbf{B1}}-\textup{\textbf{B3} } hold for all time $t\in[1,T^*]$, then for all $ t\in [1,T^*] $
\begin{enumerate}
    \item \label{Energy Bound} Energy Bound: $\mathcal{E}(t)< 2C\epsilon^2$,
    \item \label{Compact support} Compact Support of $\theta$: $\textup{supp }\theta \subset [1.5\kappa-4C\epsilon, 1-1.5\kappa+4C\epsilon]$,

    \item\label{CK integral} $\textup{\CK}$ Integral Estimates:
    \[\int_{1}^t \big[\textup{\CK}_{\lambda}(s)+\textup{\CK}_{\Lambda}(s)+\textup{\CK}_{\Theta}(s)\big]\;ds \leq 6C\epsilon^2.
    \]
\end{enumerate}
As a result, one can take $T^*=\infty$.
\end{prop}
We now introduce the following cutoff function:
\begin{equation}\label{cutoff function}
\text{supp } \chi \subset [\kappa/2,1-\kappa/2], \qquad \chi(y) \equiv 1 \textup{ for all } y\in [\kappa, 1-\kappa],
\end{equation}
where throughout the present work $\kappa$ is chosen to be sufficiently small and fixed as in the statement of Theorem~\ref{main theorem}. Moreover, we require that it belongs to the Gevrey-$\frac{2}{s_0+1}$ class and satisfies
\begin{equation}\label{cutoff regularity}
\sup_{y\in \mathbb{R}}\Big|\dfrac{d^{m}\chi(y)}{dy^m}\Big|\leq M^m(m!)^{\frac{2}{s_0+1}}(m+1)^{-2}.
\end{equation} 
Hence, due to the bootstrap hypotheses that $\theta$ is compactly supported, we may write
\begin{align*}
    \theta(t,z,y)\chi(y)=\theta(t,z,y),\quad \chi'(y)\nabla\theta(t,z,y)=0,\quad 
    \chi(y)\underline{\varrho}'(y)=\underline{\varrho}'(y).
\end{align*}
Next, we apply the cutoff function to \eqref{transport equation in new coord} and rewrite the equation as:
\begin{equation}\label{equation with cut-off}
\begin{cases}
    \partial_t \theta(t,z,y) + \nabla^{\perp}(\phi(t,z,y)\chi(y)) \cdot \nabla\theta(t,z,y) =-\partial_z (\phi(t,z,y)\chi(y))  \underline{\varrho}'(y) ,\\
    \Delta_L^2\phi(t,z,y)=\partial_z\theta,\\
    \phi(t, z, 0)=\phi(t, z, 1)=\partial_y\phi(t, z, 0)=\partial_y\phi(t, z, 1)=0.
\end{cases}
\end{equation}

\begin{proof}
    Here, let us first present the proof of \eqref{Compact support}: compact support of $\theta$. Our argument requires some information on the decay of the vertical velocity component, $U^2$. We know that
    \[
    \Delta \psi(t,x,y)=\partial_x \rho(t,x,y)=\partial_z\theta(t,z,y), \quad \psi(y=1)=\psi(y=0)=\psi'(y=1)=\psi'(y=0)=0.
    \]
    Applying the Fourier transformation in $x$ yields
    \begin{equation}\label{representation of psi k in new coordinate}
    \tilde{\psi}_k(t,y)=\int_{0}^1 i e^{-ik\By t} \mathfrak{K}(t,y,\By) \tilde{\theta}_k(t,\By)\;d\By,
    \end{equation}
    and 
    \begin{equation}\label{representation of psiy k in new coordinate}
    \partial_y\tilde{\psi}_k(t,y)=\int_{0}^1 i e^{-ik\By t} \partial_y\mathfrak{K}(t,y,\By) \tilde{\theta}_k(t,k,\By)\;d\By,
    \end{equation}

    where $\mathfrak{K}(t,y,\By):=K(t,|y-\By|)+K^{g}_{bd}(t,y,\By)$ is the kernel mentioned in Lemma~\ref{rep of psi_k}. By construction, the kernel $\mathfrak{K}$ is 4 times differentiable in the $\By$ direction. Therefore, upon integrating by parts (4 times) equation \eqref{representation of psi k in new coordinate} and (3 times) equation \eqref{representation of psiy k in new coordinate} in the $\By$ variable  along with using the bootstrap hypotheses and the compact support of $\theta_k$ allow us to say that 
    $|\tilde{\psi}_k(t,y)|\lesssim \frac{\epsilon}{\JB{t}^4}$ and $|\partial_y \tilde{\psi}_k(t,y)|\lesssim \frac{\epsilon}{\JB{t}^3}$. As a consequence, we obtain 
    \begin{equation}\label{estimate on U2}
        \norm {U^2(t)}_{L^\infty(\mathbb{T}\times [0,1])}\lesssim \dfrac{\epsilon}{\JB{t}^4}, \qquad \norm{U^1(t)}_{L^\infty(\mathbb{T}\times [0,1])}\lesssim \dfrac{\epsilon}{\JB{t}^3},
    \end{equation}
    which proves the inviscid damping in \eqref{inviscid damping}.
    
    Next, we define $X^1=x+\int_0^ty+ U^2(s,X_1(s),X_2(s))\;ds$ and $X^2=y+\int_0^t U^2(s,X_1(s),X_2(s))\;ds$. Then
    \[
    \left\{
    \begin{aligned}
        &\dfrac{d(X^1(t,x,y),X^2(t,x,y))}{dt}=(y+U^1(t,X_1,X_2),U^2(t,X_1,X_2)),\\&
        (X^1,X^2)_{t=0}=(x,y).
    \end{aligned}\right.
    \]
    Notice that under this coordinate system, $\dfrac{\partial}{\partial t}\theta(t,X^1,X^2)=\underline{\varrho}'(X_2) \partial_z \psi(t,X^1,X^2)$. Additionally, from \eqref{estimate on U2}, $|X^2-y|\lesssim \epsilon$. Thus, we can say that 
    \[
    \dfrac{\partial}{\partial t}\theta(t,X^1,X^2)=0, \qquad \textup{ for all } \quad y\notin [3\kappa-C\epsilon,1-3\kappa+C\epsilon].
    \]
    Since $\theta_{\textup{in}}=0$ for all $y\notin [2\kappa,1-2\kappa]$, then under the evolution, $\theta$ remains compactly supported away from the boundaries so long as $\epsilon$ is chosen small enough. Hence, this completes the proof of \eqref{Compact support} in Proposition~\ref{bootstrap}.
\end{proof}
\begin{remark}
    We would like to remark that in order to prove \eqref{Energy Bound} and \eqref{CK integral}, it suffices to derive an estimate of the form 
    \begin{equation}\label{energy CK integral}
\mathcal{E}(t)+\frac{1}{2} \int_{1}^t \bigg[\textup{\CK}_{\lambda}(s)+\textup{\CK}_{\Lambda}(s)+\textup{\CK}_{\Theta}(s)\bigg]\;ds\leq \mathcal{E}(1)+C \epsilon^2 + C\epsilon^3.
    \end{equation}
    where the constant $C$ is independent of $\epsilon$ and $T^*$. To arrive at \eqref{energy CK integral}, it is natural to investigate the time evolution of the energy $\mathcal{E}(t)$ by the first equation in \eqref{equation with cut-off}, namely
\[
\begin{aligned}
\dfrac{d}{dt}\mathcal{E}(t)&=\dfrac{1}{2} \dfrac{d}{dt} \int_{\mathbb{T}\times[0,1]} |\A(t,\nabla)\theta(t,z,y)|^2\;dz dy\\&=-\textup{\CK}_{\lambda}-\textup{\CK}_{\Lambda}-\textup{\CK}_{\Theta}-\int_{\mathbb{T}\times[0,1]} \A\theta \A(\nabla^{\perp}(\phi\chi)\cdot \nabla \theta)\;dzdy-\int_{\mathbb{T}\times[0,1]} \A\theta \A(\partial_z (\phi\chi)  \underline{\varrho}')\;dz\;dy\\&=
-\textup{\CK}_{\lambda}-\textup{\CK}_{\Lambda}-\textup{\CK}_{\Theta}-\textup{\textbf{NL}}_\theta-\boldsymbol{\Pi}_\theta.
\end{aligned}
\]
\end{remark}

From the above remark, it is clear that the main term to handle here is $\textup{\textbf{NL}}_\theta$. We start by expressing it as follows
\begin{equation}\label{Nonlinear term}
    \textbf{NL}_\theta=\int_{\mathbb{T}\times[0,1]} \A \theta \big[\A(\nabla^{\perp}(\phi\chi)\cdot \nabla \theta)-\nabla^{\perp}(\phi\chi)\cdot \nabla \A \theta \big]\;dzdv,
\end{equation}
where we have used the divergence-free condition of $\nabla^\bot(\phi\chi)$.
In order to gain control over the commutator term $ \textbf{NL}_\theta$, we use a paraproduct decomposition approach. Roughly speaking, one can decompose $ \textbf{NL}_\theta$ into three parts according to the frequency interactions: \textit{transport }(low-high interaction), \textit{reaction }(high-low interaction), and \textit{remainder}:
\begin{equation}\label{paraproduct decomposition}
 \textbf{NL}_\theta=\int_{\mathbb{T}\times[0,1]} \A \theta \big[\A(\nabla^{\perp}(\phi\chi)\cdot \nabla \theta)-\nabla^{\perp}(\phi\chi)\cdot \nabla \A \theta \big]\;dzdv
=\dfrac{1}{2\pi}\sum_{N\geq 8} \textbf{T}_N+\dfrac{1}{2\pi}\sum_{N\geq 8} \textbf{R}_N+\dfrac{1}{2\pi}\mathcal{R},
\end{equation}
where 
\begin{equation}\label{Transport Reaction Remainder}
    \begin{aligned}
        &\textbf{T}_N=2\pi \int_{\mathbb{T}\times[0,1]} \A \theta \big[\A(\nabla^{\perp}(\phi\chi)_{<N/8} \cdot \nabla \theta_N)-\nabla^{\perp}(\phi\chi)_{<N/8} \cdot \nabla \A \theta_N \big]\;dz\;dv,\\&
        \textbf{R}_N=2\pi \int_{\mathbb{T}\times[0,1]} \A \theta \big[\A(\nabla^{\perp}(\phi\chi)_{N}\cdot \nabla \theta_{<N/8})-\nabla^{\perp}(\phi\chi)_{N} \cdot \nabla \A \theta_{<N/8} \big]\;dz\;dv,\\&
        \mathcal{R}=2\pi \sum_{N\in\mathbb{D}} \sum_{\frac{N}{8}\leq N'\leq 8N}\int_{\mathbb{T}\times[0,1]} \A \theta \big[\A(\nabla^{\perp}(\phi\chi)_{N} \cdot \nabla \theta_{N'})-\nabla^{\perp}(\phi\chi)_{N} \cdot \nabla \A \theta_{N'} \big]\;dz\;dv.
    \end{aligned}
\end{equation}
We use the dyadic domain $N \in \mathbb{D}=\{\frac{1}{2}, 1,2,4,...,2^m,...\}$ and denote $f_N$ as the Littlewood--Paley projection onto the $N$-th frequency and $f_{<N}$ as the Littlewood--Paley projection onto the frequencies less than $N$. 

Below, we present estimates for each term: reaction, transport, and remainder.
\begin{prop}[Reaction]\label{Reaction}
    Under the bootstrap hypotheses, it is true that
    \[
     \sum_{N\geq 8} \textup{\textbf{R}}_N \lesssim \dfrac{\epsilon^3}{\JB{t}^{2}}+\epsilon \textup{\CK}_{\lambda}+\epsilon \textup{\CK}_{\Lambda}.
    \]
\end{prop}

\begin{prop}[Transport]\label{Transport}
Under the bootstrap hypotheses, it is true that
\[
\sum_{N\geq 8}\mathbf{T}_N\leq \epsilon \textup{\CK}_{\lambda}.
\]
\end{prop}
\begin{prop}[Remainder]\label{remainder}
    Under the bootstrap hypotheses, it is true that
    \[
    |\mathcal{R}|\lesssim \dfrac{\epsilon^3}{\JB{t}^2}.
    \]
\end{prop}
For the linear term $\boldsymbol{\Pi}_\theta$, we have the following estimate:
\begin{prop}\label{Prop:linear}
Under the bootstrap hypotheses, it is true that
    \begin{align*}
        |\boldsymbol{\Pi}_\theta| \lesssim  \dfrac{\delta \epsilon^2}{\JB{t}^2} +\delta \mathbf{CK}_\lambda+\delta \mathbf{CK}_\Lambda.
    \end{align*}
\end{prop} 
The proof of Propositions~\ref{Reaction}, \ref{Transport}, \ref{remainder}, and \ref{Prop:linear} can be found in subsections~\ref{Reaction term}, \ref{Transport term}, \ref{Remainder}, and \ref{linear} respectively.

\subsection{Conclusion of proof}
 This subsection is devoted to furnishing the proof of Theorem~\ref{main theorem}. In particular, we prove the scattering result displayed in \eqref{scattering},\eqref{inviscid damping}.

Integrating the first equation in \eqref{equation with cut-off} in time yields
\[
\rho_\infty=\theta(1)-\int_{1}^{\infty}\nabla^{\perp} (\phi(\tau,z,y)\chi(y)) \cdot \nabla\theta(\tau,z,y)+\partial_z (\phi(\tau,z,y)\chi(y))  \underline{\varrho}'(y)\;d\tau.
\]
Thus, by the algebra property of the Gevrey space for $\lambda_\infty<\lambda_{\textup{in}}$, bootstrap hypotheses, Minkowsky inequality, and Corollary~\ref{gevrey norm of phi chi}, we infer
\begin{equation}
\begin{aligned}
&\norm{\rho(t,x-ty,y)-\rho_\infty}_{\mathcal{G}^{\lambda_\infty;\frac{1}{3}}}\\
&=\norm{\int_{t}^\infty \nabla^{\perp}(\phi(\tau,z,y)\chi(y)) \cdot \nabla\theta(\tau,z,y)+\partial_z (\phi(\tau,z,y)\chi(y))  \underline{\varrho}'(y)\;d\tau}_{\mathcal{G}^{\lambda_\infty;\frac{1}{3}}}\\&
\lesssim \epsilon \int_{t}^\infty \frac{\epsilon}{\JB{\tau}^4} \;d\tau+ \int_{t}^\infty \frac{\epsilon\delta}{\JB{\tau}^4} \;d\tau
\lesssim\dfrac{\epsilon(\epsilon+\delta)}{\JB{t}^3}.
\end{aligned}
\end{equation}
Now, the proof of the inviscid damping estimates \eqref{inviscid damping} rely on the decay of the stream function. They are recorded in \eqref{estimate on U2}. The proof of Theorem~\ref{main theorem} is therefore complete.  

We present estimates for $\textbf{NL}_\theta$ in Section~\ref{nonlinear interactions}. The linear term encoded in $\boldsymbol{\Pi}_\theta$ will be handled in Section~\ref{linear}. Let us now proceed to estimate the nonlinear interactions. 
\section{Properties of multipliers}\label{properties of multipliers}

In the following section, our goal is to provide estimates of a multiplier term arising from the nonlinear interactions. Similar estimates can be found in \cite{MasmoudiBelkacemZhao2022Inv}, though,  here the proof has not only been significantly simplified, but also  improved to deal with the borderline case that the initial perturbation lies in Gevrey
space of class 3. We present them according to the resonant and non-resonant interactions. To that end, we define the following set
\begin{equation}\label{Set A}
    \mathfrak{A}=\left\{((k,\eta),(l,\xi)): ||l,\xi|-|k,\eta||\leq |(k-l,\eta-\xi)| \leq \frac{3}{16}|l,\xi|\ \text{and}\ |l|\neq 0 \right\}.
\end{equation}
Clearly, any element in $\mathfrak{A}$ satisfies the condition
\begin{equation}\label{comparability}
     \frac{13}{16}|l,\xi| \leq |k,\eta|\leq \frac{19}{16}|l,\xi|.
\end{equation}

To begin with, let us consider the case where resonance or non-resonance does not matter.
\begin{lemma}\label{l xi nonresonant}
For any $t\geq 1$, let us denote
\begin{align*}
\mathfrak{U}_1&=\{l\xi<0\},\quad
\mathfrak{U}_2=\{l \xi>0,\  |\xi|\leq \frac{1}{2}|lt|\,\  \mathrm{or}\ |\xi|\geq 2|lt|\},\\
\mathfrak{U}_3&=\{l \xi>0, \ \frac{1}{2}|lt|<|\xi|<2|lt|, \ |l,\xi|\leq 10^3\},
\end{align*}
and
$\mathfrak{U}=\mathfrak{U}_1\cup\mathfrak{U}_2\cup \mathfrak{U}_3.$
Assume that $((k,\eta),(l,\xi))\in \mathfrak{A}$ and $(l,\xi)\in \mathfrak{U}$, then

\[
\dfrac{|l||l,\xi|}{(l^2+(\xi-lt)^2)^2} \dfrac{\A_k(\eta)}{\A_l(\xi)}\lesssim \dfrac{1}{\JB{t}^2} e^{{\bf c}\lambda(t)|\eta-\xi,k-l|^{1/3}}.
\]

\end{lemma}
\begin{proof}
If $(l,\xi)\in \mathfrak{U}_3$, we infer from \eqref{comparability} that $|k,\eta|\lesssim1$. Consequently,
\begin{align*}
\dfrac{|l||l,\xi|}{(l^2+(\xi-lt)^2)^2} \dfrac{\A_k(\eta)}{\A_l(\xi)}\lesssim \frac{1}{|l|^3\langle \frac{\xi}{l}-t \rangle^4} \lesssim \frac{\langle \frac{\xi}{l}\rangle^4}{\langle t\rangle^4}  \lesssim\frac{1}{\langle t\rangle^4}.
\end{align*}
For any $(l,\xi)\in\mathfrak{U}_1\cup\mathfrak{U}_2$, it is easy to verify that $(\xi-lt)^2\gtrsim \xi^2+l^2 t^2$. Moreover, recalling the definition of ${\rm I}_{l,\xi}$ in Section \ref{sec-notation}, we see that $t\notin{\rm I}_{l,\xi}$ whenever $(l,\xi)\in \mathfrak{U}_1\cup \mathfrak{U}_2$. Combining these with the fact  $|\eta|\leq \frac{19}{16} |l,\xi|$, and using Lemmas  \ref{estimate of Lambda fraction} and \ref{ratio of J},  allow us to infer that
\begin{align*}
&\dfrac{|l||l,\xi|}{(l^2+(\xi-lt)^2)^2} \dfrac{\A_k(\eta)}{\A_l(\xi)}\\
&\lesssim \dfrac{|l||l,\xi|}{(l^2+(\xi-lt)^2)^2}\Big(\dfrac{|\eta|}{|k|^3(1+|t-\frac{|\eta|}{|k|}|)}\mathds{1}_{t\in  \textup{I}_{k,\eta}}+\mathds{1}_{t\notin  \textup{I}_{k,\eta}}\Big) e^{(c\lambda(t)+3\mu+2C_0)|\eta-\xi,k-l|^{1/3}}\\&\lesssim \dfrac{|l||l,\xi|^2}{l^4+\xi^4+l^4t^4}e^{{\bf c}\lambda(t)|\eta-\xi,k-l|^{1/3}}\lesssim \dfrac{1}{\JB{t}^2} e^{{\bf c}\lambda(t)|\eta-\xi,k-l|^{1/3}}.
\end{align*}

This completes the proof of Lemma \ref{l xi nonresonant}.
\end{proof}
Now we turn to the case where the interactions are between the resonant and resonant modes. 
\begin{lemma}[Resonant-Resonant]\label{R-R}
    Let $t\in \textup{I}_{k,\eta} \cap \textup{I}_{l,\xi}$ and $((k,\eta),(l,\xi))\in \mathfrak{A}$. For $k \neq l$, it holds that
    \begin{equation}\label{RR1}
   \dfrac{|l||l,\xi|}{(l^2+(\xi-lt)^2)^2} \dfrac{\A_k(\eta)}{\A_l(\xi)}\lesssim \sqrt{\dfrac{\partial_t\Lambda(t,\eta)}{\Lambda(t,\eta)}} \sqrt{\dfrac{\partial_t\Lambda(t,\xi)}{\Lambda(t,\xi)}}\mathds{1}_{|l|\leq |\xi|}\mathds{1}_{|k|\leq |\eta|}e^{{\bf c} \lambda(t) |k-l,\eta-\xi|^{\frac{1}{3}}}.
    \end{equation}
    Meanwhile, for $k=l,$ the following inequality holds
    \begin{equation}\label{RR2}
   \dfrac{|l|^2}{(l^2+(\xi-lt)^2)^2} \dfrac{\A_k(\eta)}{\A_l(\xi)}\lesssim \sqrt{\dfrac{\partial_t\Lambda(t,\eta)}{\Lambda(t,\eta)}} \sqrt{\dfrac{\partial_t\Lambda(t,\xi)}{\Lambda(t,\xi)}}\mathds{1}_{|l|\leq |\xi|}\mathds{1}_{|k|\leq |\eta|}e^{{\bf c} \lambda(t) |\eta-\xi|^{\frac{1}{3}}}.
    \end{equation}
\end{lemma}
\begin{proof}
 Here, frequencies $(k,\eta)$ and $(l,\xi)$ are both resonant. Thus \eqref{eta comparable to xi} holds, which, together with the fact $t\in{\rm I}_{k,\eta}\cap{\rm I}_{l,\xi}$, in turn implies that $|k|\approx |l|\le |\xi|$.
 
 Recalling \eqref{Multiplier A} and \eqref{Multiplier components}, noting that $((k,\eta),(l,\xi))\in \mathfrak{A}$, and using Lemmas \ref{estimate of Lambda fraction} and \ref{ratio of J},  one deduces that
\begin{align}\label{ratio-A1}
\frac{\mathcal{A}_k(\eta)}{\mathcal{A}_l(\xi)}\nonumber\lesssim e^{c\lambda(t)|k-l,\eta-\xi|^{\frac{1}{3}}}\frac{\mathcal{J}_k(t,\eta)}{\mathcal{J}_l(t,\xi)}\fr{\mathcal{M}_k(t,\eta)}{\mathcal{M}_l(t,\xi)}
\les e^{{(c\lambda(t)+3\mu+2C_0)}|k-l,\eta-\xi|^{\frac{1}{3}}}.
\end{align}

We split our analyses into two parts: $k \neq l$ and $k=l$.
   
Case $k\neq l$. From the above inequality, we obtain
         \[
         \begin{aligned}
             &\dfrac{|l||l,\xi|}{(l^2+(\xi-lt)^2)^2} \dfrac{\A_k(\eta)}{\A_l(\xi)}\\
             \lesssim&\fr{1+|t-\fr{\eta}{k}|}{1+|t-\fr{\xi}{l}|}\fr{|\xi|}{|l|^3(1+|t-\fr{\xi}{l}|)}\sqrt{\dfrac{\partial_t\Lambda(t,\xi)}{\Lambda(t,\xi)}} \sqrt{\dfrac{\partial_t\Lambda(t,\eta)}{\Lambda(t,\eta)}}e^{(c\lambda(t)+3\mu+2C_0) |k-l,\eta-\xi|^{\frac{1}{3}}}.
         \end{aligned}
        \] 
        Recalling the fact that $t\in \textup{I}_{k,\eta} \cap \textup{I}_{l,\xi} \subset \tilde{\textup{I}}_{k,\eta} \cap \tilde{\textup{I}}_{l,\xi}$, we are able to use Lemma~\ref{Scenarios} to absorb the first two factors on the right-hand side of the last inequality. In fact, if (2) in Lemma \ref{Scenarios} holds, noting that $|\eta|\approx|\xi|$ and $|k|\approx|l|$, then
        \begin{align*}
\fr{1+|t-\fr{\eta}{k}|}{1+|t-\fr{\xi}{l}|}\fr{|\xi|}{|l|^3(1+|t-\fr{\xi}{l}|)}\les \fr{1+\fr{|\eta|}{|k|^3}}{1+\fr{|\xi|}{|l|^3}}\les1.
        \end{align*}
        If (3) in Lemma \ref{Scenarios} holds, then
        \begin{align*}
\fr{1+|t-\fr{\eta}{k}|}{1+|t-\fr{\xi}{l}|}\fr{|\xi|}{|l|^3(1+|t-\fr{\xi}{l}|)}\les \fr{|\eta|}{|k|^3}\fr{|\xi|}{|l|^3}\les |\eta-\xi|^2.
        \end{align*}
 Thus, it follows that
        \[
        \begin{aligned}
        \dfrac{|l||l,\xi|}{(l^2+(\xi-lt)^2)^2} \dfrac{\A_k(\eta)}{\A_l(\xi)}
        \lesssim \sqrt{\dfrac{\partial_t\Lambda(t,\xi)}{\Lambda(t,\xi)}} \sqrt{\dfrac{\partial_t\Lambda(t,\eta)}{\Lambda(t,\eta)}} e^{{\bf c}\lambda(t) |k-l,\eta-\xi|^{\frac{1}{3}}},
        \end{aligned}
        \]
which proved \eqref{RR1}.

Case $k=l$. Using the facts that $\frac{1+a}{1+b}\leq 1+|a-b|$ and $|\eta|\approx |\xi|$, we infer that 
    \[
    \begin{aligned}
     \dfrac{|l|^2}{(l^2+(\xi-lt)^2)^2} \dfrac{\A_k(\eta)}{\A_l(\xi)}
     \lesssim &\frac{\JB{\eta-\xi}}{(1+|\frac{\xi}{l}-t|)(1+|\frac{\eta}{k}-t|)}e^{(c\lambda(t)+3\mu+2C_0) |\eta-\xi|^{\frac{1}{3}}}\\
     \lesssim& \sqrt{\dfrac{\partial_t\Lambda(t,\xi)}{\Lambda(t,\xi)}} \sqrt{\dfrac{\partial_t\Lambda(t,\eta)}{\Lambda(t,\eta)}} e^{{\bf c}\lambda(t) |\eta-\xi|^{\frac{1}{3}}},
     \end{aligned}
    \]
    which proves \eqref{RR2}. We complete the proof of Lemma \ref{R-R}.
\end{proof}
We now move on to derive the estimate in the situation when resonant frequency interacts with nonresonant frequency.
\begin{lemma}[Resonant-Nonresonant]\label{R-NR}
    Let $t\in \textup{I}_{k,\eta} \cap \textup{I}^c_{l,\xi}$ and $((k,\eta),(l,\xi))\in \mathfrak{A}$, then
    \[
   \dfrac{|l||l,\xi|}{(l^2+(\xi-lt)^2)^2} \dfrac{\A_k(\eta)}{\A_l(\xi)}\lesssim \Bigg( \dfrac{1}{\JB{t}^2} +\sqrt{\frac{\partial_t\Lambda(t,\eta)}{\Lambda(t,\eta)}}\sqrt{\frac{\partial_t\Lambda(t,\xi)}{\Lambda(t,\xi)}}\mathds{1}_{|l|\leq |\xi|}\mathds{1}_{|k|\leq |\eta|} \Bigg)e^{{\bf c}\lambda(t)|\eta-\xi,k-l|^{\frac{1}{3}}}.
     \]
\end{lemma}
\begin{proof}


    
    The fact $t\in \textup{I}_{k,\eta}$ implies that $\eta k>0$ and $1\leq |k|\leq E(|\eta|^{\frac{1}{3}})$. This, together with \eqref{comparability}, shows that
\[
|\xi-\eta|\leq \frac{3}{16}|l,\xi|\leq \frac{3}{16}\frac{16}{13}|k,\eta|\le \fr{6}{13}|\eta|,
\]
which in turn yields
\begin{equation}\label{eta is like xi}
\frac{7}{13}|\eta|\leq |\xi|\leq \frac{19}{13}|\eta|.
\end{equation}
By Lemma \ref{estimate of Lambda fraction} and Lemma \ref{ratio of J}, we have
\begin{align}\label{ratia-A2}
\dfrac{|l||l,\xi|}{(l^2+(\xi-lt)^2)^2} \dfrac{\A_k(t,\eta)}{\A_l(t,\xi)}\lesssim \dfrac{|l||l,\xi|}{(l^2+(\xi-lt)^2)^2} \dfrac{|\eta|}{|k|^3(1+|t-\frac{|\eta|}{|k|}|)} e^{(c\lambda(t)+3\mu+2C_0)|\eta-\xi,k-l|^{1/3}}.
\end{align}

Our goal now is to bound the product of two fraction symbols appearing in the above inequality. 
To this end, we consider the following cases: 

\textbf{Case 1: $(l,\xi) \in \mathfrak{U}$.} Here, Lemma~\ref{l xi nonresonant} gives us the desired inequality.\\

\textbf{Case 2:  $(l,\xi) \in \mathfrak{U}^c$.} In this case, we focus on the setting where $l\xi>0$ and  $\frac{1}{2}|l|t\leq |\xi|\leq 2|l|t$ with $|l,\xi|>10^3$.  This, together with the fact that $(k,\eta)$ is resonant, yields $\frac{|\eta|}{|k|}\approx t\approx \frac{|\xi|}{|l|}$, $|\xi|\geq 125$, and $ \frac{2|\xi|}{2E(|\xi|^{2/3})+1}\le t\le 2|\xi|$. 
As a consequence, due to \eqref{eta is like xi}, we then have $|k|\approx |l|$, and there exists $n$ such that $1\le|n|\le E(|\xi|^{\fr{2}{3}})$ such that $t\in \tilde{\rm I}_{n,\xi}$. 

If $1 \leq |l|\leq \frac{1}{2}E(|\xi|^{\frac{1}{3}})$, 
combining this restriction  with the fact $\fr{1}{2}\fr{|\xi|}{|l|}\le t\le \fr{2|\xi|}{|l|}$, we have  $\frac{2|\xi|}{2E(|\xi|^{\frac{1}{3}})+1} \leq t\leq 2|\xi|$. Hence, there exists $n$ satisfying $1\le|n|\le E(|\xi|^{\fr{1}{3}})$ such that $t\in \tilde{\textup{I}}_{n,\xi}$. Accordingly, now we have $t\in \tilde{\rm I}_{n,\xi}\cap {\rm I}_{k,\eta}\subset \tilde{\rm I}_{n,\xi}\cap \tilde{\rm I}_{k,\eta} $ with $|\eta|\approx|\xi|$, which in turn implies that $|n|\approx |k|\approx|l|$. If $n\ne l$, noting that $t\in\tilde{\rm I}_{n,\xi}$, then there holds
\begin{align}\label{camparable}
|t-\fr{\xi}{n}|\les \fr{|\xi|}{|n|^2}\approx \fr{|\xi|}{|l|^2} \lesssim|t-\fr{\xi}{l}|
\end{align}
It follows from   \eqref{ratia-A2},  \eqref{up-nonresonant} with $k$ replaced by $l$ and the above inequality that
\begin{align}
    \dfrac{|l||l,\xi|}{(l^2+(\xi-lt)^2)^2} \dfrac{\A_k(\eta)}{\A_l(\xi)}\nn\lesssim &\dfrac{|\xi|}{|l|^3(1+|\frac{\xi}{l}-t|^2)^2}\dfrac{|\eta|}{|k|^3(1+|t-\frac{|\eta|}{|k|}|)}e^{(c\lambda(t)+3\mu+2C_0)|\eta-\xi,k-l|^{1/3}}\\
    \nn\approx&\left(\fr{|\xi|}{|l|^3(1+|t-\fr{\xi}{l}|)}\right)^2\fr{|l|^3}{|\xi|}\fr{|\eta|}{|k|^3}\fr{1}{1+|t-\fr{\xi}{n}|}\fr{1}{1+|t-\fr{\eta}{k}|}\\
    \nn&\times \fr{1+|t-\fr{\xi}{n}|}{1+|t-\fr{\xi}{l}|}\fr{e^{(c\lambda(t)+3\mu+2C_0)|\eta-\xi,k-l|^{1/3}}}{1+|t-\fr{\xi}{l}|}\\
    \nn\lesssim& \sqrt{\frac{\partial_t\Lambda(t,\xi)}{\Lambda(t,\xi)}}\sqrt{\frac{\partial_t\Lambda(t,\eta)}{\Lambda(t,\eta)}}e^{{\bf c}\lambda(t)|\eta-\xi,k-l|^{1/3}}.
    \end{align}

    The remaining sub-case when $ |l|> \frac{1}{2}E(|\xi|^{\frac{1}{3}})$ can be treated similarly, since now it holds that $|l|\approx |k| \approx |\xi|^{\fr13}$. The proof of Lemma \ref{R-NR} is then complete. 
    \end{proof}

We now turn to the estimate for the non-resonant-resonant interaction. It is this scenario where we use the exchange $\Theta_k(t,\eta)$ for $\Theta_l(t,\xi)$ to absorb the large factor $\fr{|\xi|}{|l|^3}$.
\begin{lemma}[Nonresonant-Resonant]\label{NR-R}
    Let $t \in \textup{I}^c_{k,\eta}\cap \textup{I}_{l,\xi} $  and $((k,\eta),(l,\xi))\in \mathfrak{A}$, then
    \begin{equation}\label{NR-RIneq}
    \dfrac{|l||l,\xi|}{(l^2+(\xi-lt)^2)^2} \dfrac{\A_k(\eta)}{\A_l(\xi)}\lesssim \Bigg( \sqrt{\dfrac{\partial_t\Lambda(t,\eta)}{\Lambda(t,\eta)}} \sqrt{\dfrac{\partial_t\Lambda(t,\xi)}{\Lambda(t,\xi)}}\mathds{1}_{|l|\leq |\xi|}\mathds{1}_{|k|\leq |\eta|}+\dfrac{1}{\JB{t}^2} \Bigg) e^{{\bf c}\lambda(t)|k-l,\eta-\xi|^{\frac{1}{3}}}.
    \end{equation}
\end{lemma}
\begin{proof}
    By Lemma \ref{estimate of Lambda fraction} and Lemma \ref{ratio of J}, we have
    \begin{align}\label{ratio-A-NRR}
    &\dfrac{\A_k(\eta)}{\A_l(\xi)}\lesssim \frac{|l|^3\Big(1+|t-\frac{\xi}{l}|\Big)}{|\xi|} e^{(c\lambda(t)+3\mu+2C_0)|k-l,\eta-\xi|^{\frac{1}{3}}}.
    \end{align}
 Next,   similar to the proof of Lemma \ref{R-NR}, we just focus on the case when $(l,\xi)\in \mathfrak{U}^{c}$. In particular, we have $|l,\xi|\ge1000$. Combining this with the fact that $(l,\xi)$ is resonant, we know that $1 \leq |l| \leq E(|\xi|^{\frac{1}{3}})\leq |\xi|$ and $\fr{7}{8}|\xi|^{\frac23}\leq\frac{|\xi|}{E(|\xi|^{\frac{1}{3}})}-\fr{|\xi|}{(2E(|\xi|^{\frac{1}{3}}))^3}\leq t\leq \frac{3|\xi|}{2}$. Moreover, we have $|\eta|\approx |\xi|$ due to  \eqref{eta comparable to xi}.  
Then it is not difficult to verify that $t\geq \frac{2|\eta|}{2E(|\eta|^{2/3})+1}$.  Without loss of generality we assume that $2|\eta|< \fr{3|\xi|}{2}$, then the analysis can be  split into two parts:

(1). Suppose that $2|\eta|<t\leq \frac{3|\xi|}{2}$. Under this condition, we obtain
\[
|\xi-\eta|\geq |\xi|-\frac{3}{4}|\xi|= \frac{|\xi|}{4}\geq \frac{t}{6}.
\]
and thus
\[
\begin{aligned}
\dfrac{|l||l,\xi|}{(l^2+(\xi-lt)^2)^2} \dfrac{\A_k(\eta)}{\A_l(\xi)}&\lesssim\dfrac{|\xi|}{|l|^3(1+(\frac{\xi}{l}-t)^2)^2} \dfrac{|l|^3\Big(1+|t-\frac{\xi}{l}|\Big)}{|\xi|} e^{(c\lambda(t)+3\mu+2C_0)|k-l,\eta-\xi|^{\frac{1}{3}}}\\
&\lesssim \dfrac{|\xi-\eta|^2}{|t|^{2}}e^{(c\lambda(t)+3\mu+2C_0)|k-l,\eta-\xi|^{\frac{1}{3}}}\lesssim \dfrac{1}{\JB{t}^2} e^{{\bf c}\lambda(t)|k-l,\eta-\xi|^{\frac{1}{3}}}.
\end{aligned}
\]

(2). Consider now the case when $\frac{2|\eta|}{2E(|\eta|^{2/3})+1}\le t\leq 2|\eta|$. This, together with the fact that $t\in{\rm I}_{l,\xi}$,  implies that there exists $n$ such that $1\le|n|\le E(|\eta|^{\fr{2}{3}})$ such that $t\in \tilde{\textup{I}}_{n,\eta} \cap \tilde{\textup{I}}_{l,\xi}$. Also since $t\in {\rm I}_{l,\xi}$, it holds that $t\geq \frac{1}{2}|\xi|^{2/3}$. By the fact that $|l,\xi|\geq 1000$, we have $|\xi|\geq 800$, which implies $t\geq 10$ and thus $|\xi|\geq 5|l|$ and $|\eta|>|k|$. By Lemma~\ref{Scenarios}, now we estimate
the left hand side of \eqref{NR-RIneq}  in three scenarios below:\\
 \textbf{Case 1:} $l=n$. 
     In view of \eqref{ratio-A-NRR}, and noting that $|\eta|\approx|\xi|$, we have
    \begin{equation}\label{cs1}
    \begin{aligned}
        &\dfrac{|l||l,\xi|}{(l^2+(\xi-lt)^2)^2} \dfrac{\A_k(\eta)}{\A_l(\xi)}\\
        \lesssim& \dfrac{\min\{1,\fr{|\eta|}{|n|^3}\}}{1+|t-\frac{\eta}{n}|}\fr{1}{1+|t-\frac{\xi}{l}|}\fr{\fr{|\xi|}{|l|^3}}{\min\{1,\fr{|\eta|}{|n|^3}\}}\fr{|l|^3}{|\xi|} \dfrac{1+|t-\fr{\eta}{n}|}{1+|\frac{\xi}{l}-t|}e^{(c\lambda(t)+3\mu+2C_0)|k-l,\eta-\xi|^{\frac{1}{3}}}\\
        \lesssim&\sqrt{\dfrac{\partial_t\Lambda(t,\eta)}{\Lambda(t,\eta)}}\sqrt{\dfrac{\partial_t\Lambda(t,\xi)}{\Lambda(t,\xi)}}\mathds{1}_{|l|\leq |\xi|}\mathds{1}_{|k|\leq |\eta|}e^{{\bf c}\lambda(t)|k-l,\eta-\xi|^{\frac{1}{3}}}.
    \end{aligned}
    \end{equation}
\textbf{Case 2:}   $|t-\frac{\xi}{l}|\gtrsim\frac{|\xi|}{|l|^2}$ and $|t-\frac{\eta}{n}|\gtrsim \frac{|\eta|}{|n|^2}$. It follows from the facts $t\in\tilde{\rm I}_{n,\eta}\cap\tilde{\rm I}_{l,\xi}$ and $|\eta|\approx|\xi|$ that $|n|\approx|l|$. As a result, similar to \eqref{camparable}, there holds 
    \[
    |t-\frac{\eta}{n}|\les \frac{|\eta|}{|n|^2}\approx \frac{|\xi|}{|l|^2}\lesssim |t-\frac{\xi}{l}|.
    \]
  Then the same inequality as in \eqref{cs1} follows for this case. \\
 \textbf{Case 3:}  $|\xi-\eta|\gtrsim \frac{|\xi|}{|l|}$. Recalling that we are dealing with the case when $(l,\xi)\in\mathfrak{U}^c$, which in particular implies that $\fr{|\xi|}{|l|}\approx t$. Then we immediately have $|\xi-\eta|\gtrsim \frac{|\xi|}{|l|}\gtrsim \JB{t}$. Thus,
    \begin{equation}
    \begin{aligned}
        \dfrac{l|l,\xi|}{(l^2+(\xi-lt)^2)^2} \dfrac{\A_k(\eta)}{\A_l(\xi)}&\lesssim \dfrac{|\xi|}{|l|^3(1+(\frac{\xi}{l}-t)^2)^2}\dfrac{|l|^3\Big(1+|t-\frac{\xi}{l}|\Big)}{|\xi|} e^{(c\lambda(t)+3\mu+2C_0)|k-l,\eta-\xi|^{{\frac{1}{3}}}}\\
        &\lesssim \frac{|\xi-\eta|^2}{\JB{t}^2} e^{(c\lambda(t)+3\mu+2C_0)\lambda(t)|k-l,\eta-\xi|^{\frac{1}{3}}}  \lesssim \frac{e^{{\bf c}\lambda(t)|k-l,\eta-\xi|^{\frac{1}{3}}}}{\JB{t}^2}.
    \end{aligned}
    \end{equation}
Gathering together all the estimates above, we complete the proof of Lemma \ref{NR-R}. 
\end{proof}
Next, we consider the nonresonant-nonresonant interaction.
\begin{lemma}[Nonresonant-Nonresonant]\label{NR-NR}
     Let $t\in \textup{I}^c_{k,\eta} \cap \textup{I}^c_{l,\xi}$ and $((k,\eta), (l,\xi))\in\mathfrak{A}$, then
     \begin{equation}\label{NRNRIneq}
     \dfrac{|l||l,\xi|}{(l^2+(\xi-lt)^2)^2} \dfrac{\A_k(\eta)}{\A_l(\xi)}\lesssim \bigg(\frac{|l,\xi|^{\frac{1}{6}} |k,\eta|^{\frac{1}{6}}}{\JB{t}^{\frac{3}{2}}}+ \sqrt{\dfrac{\partial_t\Lambda(t,\eta)}{\Lambda(t,\eta)}}\sqrt{\dfrac{\partial_t\Lambda(t,\xi)}{\Lambda(t,\xi)}}\mathds{1}_{|l|\leq |\xi|}\mathds{1}_{|k|\leq |\eta|}\bigg)e^{{\bf c}\lambda(t)|k-l,\eta-\xi|^{\frac{1}{3}}}.
     \end{equation}
\end{lemma}
\begin{proof} First of all, due to the fact that both $(k,\eta)$ and $(l,\xi)$ are nonresonant, by Lemma \ref{estimate of Lambda fraction} and Lemma~\ref{ratio of J}, we are led to
  \begin{equation}\label{ratio A nrnr}
       \dfrac{\A_k(\eta)}{\A_l(\xi)}\lesssim e^{(c\lambda(t)+3\mu+2C_0)|k-l,\eta-\xi|^{\frac{1}{3}}}.
  \end{equation}
  We are left to estimate  $\frac{|l||l,\xi|}{(l^2+(\xi-lt)^2)^2}$.  Again, we just focus on the case when $(l,\xi)\in\mathfrak{U}^c$. This is equivalent to saying that $l\xi>0$ and $\frac{|lt|}{2}<|\xi|<2|lt|$ and $|l,\xi|>1000$. Under these conditions, we have 
\begin{equation}\label{comparability of eta xi}
    |\xi-\eta|\leq \frac{3}{16}|l,\xi| \leq \frac{9}{16}|\xi|,\quad \text{and}\quad \dfrac{7}{16}|\xi|\leq |\eta|\leq \dfrac{25}{16}|\xi|.
    \end{equation}
    The analysis will be given based on four time intervals:

      (1) $1\leq t<1000 |\xi|^{\frac{1}{2}}$. Inside this interval, we infer from $|\xi|\les |lt|\lesssim |l||\xi|^{\fr{1}{2}}$ that $|\xi|\les|l|^{2}$. Consequently,
        \[
      \dfrac{|l||l,\xi|}{(l^2+|\xi-lt|^2)^2}\lesssim \frac{1}{|\xi|^{\frac{1}{2}}}\lesssim \frac{|\xi|^{\frac{1}{6}}|\eta|^{\frac{1}{6}}}{t^{\frac{5}{3}}}.
       \]
       
       (2) $1000|\xi|^{\frac{1}{2}}\leq t \leq \max \biggl\{ \frac{2|\xi|}{2E(|\xi|^{1/3})+1} , \frac{2|\eta|}{2E(|\eta|^{1/3})+1} \biggr\}$. Then there exist $n,j$ such that $t\in \tilde{\textup{I}}_{j,\eta} \cap  \tilde{\textup{I}}_{n,\xi}$. It is true here that $|\eta|^{1/2} \approx |\xi|^{1/2}\lesssim t \lesssim |\xi|^{2/3}\approx |\eta|^{2/3}$, $|\xi|^{\fr{1}{3}}\les|n|\approx |j|\approx |l| \approx \frac{|\xi|}{t}\les |\xi|^{\fr{1}{2}}$. We also get that $t>1$, which implies $|\xi|\geq 5|l|$ and $|\eta|>|k|$. With this in mind, we consider the following cases:

       \textbf{Case 1:} $n=l$. Via Lemma~\ref{Scenarios}, we are led to consider a number of scenarios. 
       \begin{itemize}
           \item $j=l$. 
           \[
    \dfrac{|l||l,\xi|}{(l^2+|\xi-lt|^2)^2}\lesssim \dfrac{|\xi|^{\frac{1}{2}}|\eta|^{\frac{1}{2}} \JB{\eta-\xi}}{|j|^3(1+|\frac{\xi}{n}-t|) (1+|\frac{\eta}{j}-t|)}\lesssim \sqrt{\frac{\partial_t\Lambda(t,\eta)}{\Lambda(t,\eta)}} \sqrt{\frac{\partial_t\Lambda(t,\xi)}{\Lambda(t,\xi)}}\mathds{1}_{|l|\leq |\xi|}\mathds{1}_{|k|\leq |\eta|}\JB{\eta-\xi}.
    \]
         \item $j\neq l$. We have the following two sub-cases.
        
         (a) $\big| t-\frac{\eta}{j} \big|\geq \frac{1}{10 \alpha}  \frac{|\eta|}{j^2}$ and $|t-\frac{\xi}{n}|\geq \frac{1}{10 \alpha}  \frac{|\xi|}{n^2}$. Combining this with the facts $|l|\les|\xi|^{\fr{1}{2}}$ and $t\les|\xi|^{\fr{2}{3}}$ yield
        \begin{align}\label{j not l}
    \dfrac{|l||l,\xi|}{(l^2+|\xi-lt|^2)^2}\lesssim \dfrac{|\xi|}{|l|^3(1+|\frac{\xi}{l}-t|^2)^2} \lesssim \frac{|l|^5}{|\xi|^{3}}\lesssim\frac{|l|^4}{t|\xi|^2}\les\fr{1}{t} \lesssim \frac{|\xi|^{\frac{1}{6}}|\eta|^{\frac{1}{6}}}{t^{\frac{3}{2}}}.
    \end{align}
          
          (b) $|\xi-\eta|\geq_{\alpha} \frac{|\eta|}{|j|}$. Here, $\frac{|\eta|}{|j|}\approx t$. Then thanks to \eqref{up-nonresonant} with $k$ replaced by $l$, we arrive at
          \begin{align}\label{separate}
          \dfrac{|l||l,\xi|}{(l^2+|\xi-lt|^2)^2}\lesssim 1 \les \frac{\JB{\eta-\xi}^{\frac{3}{2}}}{t^{\frac{3}{2}}}.
          \end{align}
       \end{itemize}
       
    \textbf{Case 2:} $n\neq l$. In this case $|\frac{\xi}{l}-t|\gtrsim \frac{|\xi|}{l^2}$ since $t\in\tilde{\rm I}_{n,\xi}$. Hence, one can get the estimate as in  \eqref{j not l}.

    (3) $\max\biggl\{\frac{2|\eta|}{2E(|\eta|^{1/3})+1},\frac{2|\xi|}{2E(|\xi|^{1/3})+1}\biggr\} \leq t \leq \min\biggl\{2|\eta|,2|\xi|\biggr\}$. This means that $t\in \tilde{\text{I}}_{j,\eta} \cap \tilde{\text{I}}_{n,\xi}$ for some $1\le |j|\le E(|\eta|^{\fr{1}{3}})$, and  $1\le|n|\le E(|\xi|^{\fr{1}{3}})$. We also get that $t>1$, which implies $|\xi|\geq 5|l|$ and $|\eta|>|k|$. Using the lower bound of $t$ in this scenario, recalling that $2|\xi|> |l|t$, we then have $|l| \lesssim |\xi|^{1/3}$. Accordingly, if $l\neq n$, as in {\bf Case 2} above, there holds $|t-\frac{\xi}{l}|\gtrsim \frac{|\xi|}{l^2}$. Then similar to  \eqref{j not l}, we have
    \begin{align}\label{n ne l}
    \dfrac{|l||l,\xi|}{(l^2+|\xi-lt|^2)^2}\lesssim \dfrac{|\xi|}{l^3(1+|\frac{\xi}{l}-t|^2)^2} \lesssim \frac{|l|^5}{|\xi|^{3}}\lesssim\frac{|l|^3}{t^2|\xi|}\lesssim \frac{1}{\JB{t}^2}.
    \end{align}
    Next, we investigate the case $l=n$. By Lemma~\ref{Scenarios}, we consider the following scenarios:
    \begin{itemize}
        \item $j=l.$ Clearly,
        \[
    \dfrac{|l||l,\xi|}{(l^2+|\xi-lt|^2)^2}\lesssim \frac{1}{1+|\frac{\xi}{l}-t|^2}\lesssim \sqrt{\frac{\partial_t\Lambda(t,\eta)}{\Lambda(t,\eta)}} \sqrt{\frac{\partial_t\Lambda(t,\xi)}{\Lambda(t,\xi)}}\mathds{1}_{|l|\leq |\xi|}\mathds{1}_{|k|\leq |\eta|}\JB{\eta-\xi}.
    \]
    \item $j\neq l$. If (d) in Lemma \ref{Scenarios} holds, one can bound $\frac{|l||l,\xi|}{(l^2+(\xi-lt)^2)^2}$ exactly the same way as done in \eqref{n ne l}. If (e) in Lemma~\ref{Scenarios} holds, $\frac{|l||l,\xi|}{(l^2+(\xi-lt)^2)^2}$ can be bounded exactly the same as \eqref{separate}.
  \end{itemize}
    
    (4) $\frac{2|\xi|}{|l|}\geq t>\min\big\{2|\eta|,2|\xi|\big\}$. This, together with \eqref{comparability of eta xi},  implies that $|l|\le2$. Hence, $1\le|l|\le E(|\xi|^{\fr{1}{3}})$ since $|l,\xi|>1000$. Then the restriction $t\in {\rm I}^c_{l,\xi}$ implies that $|t-\frac{\xi}{l}|\gtrsim \frac{|\xi|}{|l|^3}$. Consequently,
     \[
    \dfrac{|l||l,\xi|}{(l^2+|\xi-lt|^2)^2}\lesssim \dfrac{|\xi|}{|l|^3(1+|\frac{\xi}{l}-t|^2)^2} \lesssim \frac{|l|^9}{|\xi|^3}\leq \frac{1}{\JB{t}^3}.
    \]

\end{proof}

We shall use all the estimates derived in Section~\ref{properties of multipliers} to estimate $\mathbf{R}_{N}$ and $\textbf{NL}_\theta$. This is precisely the content of sections \ref{Reaction term} and \ref{linear}.
\section{Stokes Estimate}\label{stokes}
In this section, we study the following system
\begin{equation}\label{eq: Stokes}
\left\{\begin{aligned}
    &\Delta_L^2\phi:=(\partial_z^2+(\partial_y-t\partial_z)^2)^2\phi=\partial_z\theta,\\
    &\phi(t, z,0)=\phi(t, z,1)=\partial_y\phi(t, z,0)=\partial_y\phi(t, z,1)=0
\end{aligned}
\right.    
\end{equation}
where $\phi(t, z, y)=\psi(t, x, y)$ with $z=x-ty$ is the stream function in new coordinates, namely, $u=\nabla^{\bot}\psi$ solves the following Stokes equations
   \begin{equation}\label{stokes equation}
    \begin{cases}
        -\Delta u+\nabla p=\rho e_2,\\
        \nabla \cdot u=0,\quad u|_{y=0,1}=0,
    \end{cases}
   \end{equation}
and the associated stream function $\psi$ in the original coordinate solves
\begin{equation}\label{eq: Stokes2}
\left\{\begin{aligned}
    &\Delta^2\psi:=(\partial_x^2+\partial_y^2)^2\psi=\partial_x\rho,\\
    &\psi(t, x,0)=\psi(t, x,1)=\partial_y\psi(t, x,0)=\partial_y\psi(t,x,1)=0.
\end{aligned}
\right.    
\end{equation}
Let $\tilde{\psi}_k(t,y)=\frac{1}{2\pi} \int_{\mathbb{T}} \psi(t,x,y) e^{-ikx} \;dx$ and $\tilde{\rho}_k(t,y)=\frac{1}{2\pi} \int_{\mathbb{T}} \rho(t,x,y) e^{-ikx} \;dx$ be the Fourier transform of $\psi$ and $\rho$ in $x$ respectively, then $\tilde{\psi}_k(y)$ solves the second equation in \eqref{gov eq Fourier x} together with the boundary conditions. For convenience, we rewrite them again below:
\begin{equation}\label{eq: Stokes3}
\left\{\begin{aligned}
    &(\partial_y^2-k^2)^2\tilde{\psi}_k=ik\tilde{\rho}_k,\\
    &\tilde{\psi}_k(0)=\tilde{\psi}_k(1)=\partial_y\tilde{\psi}_k(0)=\partial_y\tilde{\psi}_k(1)=0
\end{aligned}
\right.    
\end{equation}
\begin{lemma}\label{rep of psi_k}
    Let $\tilde{\psi}_k$ solve \eqref{eq: Stokes3}, then
    \[ \tilde{\psi}_k(t,y)=\int_{0}^{1}iK(k,|y-\By|)\tilde{\rho}_k(\By)\;d\By+\int_{0}^1iK^{g}_{bd}(k,y,\By)\tilde{\rho}_k(\By)\;d\By,
    \]
    where the kernels take the form
    \begin{align*}
    &K(k,w)=\dfrac{1}{D_k}\Bigg[4k^3\Big(w\cosh{(kw)}\Big)-4k\Big(w \cosh{(2k-kw)}\Big)-2\sinh{(2k-kw)}\\&\qquad \qquad +2k\Big(w\cosh{(kw)}\Big)-2 \sinh{(kw)}-4k^2 \sinh{(kw)}\Bigg],
    \\&K^{g}_{bd}(k,y,\By)=\dfrac{1}{D_k}\Bigg(\bigg[8k^3\cosh{(k(y-\By))}+2k^2\Big[2 \sinh{(k(\By+y))}+2\sinh{(2k-k(\By+y))}\Big]\bigg]y\By\\&\qquad \qquad  \qquad
-4k^3\Big[(y+\By) \cosh{(k(y-\By))}\Big]\\& \qquad \qquad  \qquad-4k^2\Big[(\By+y)\sinh{(k(\By+y))}+(y-\By)\sinh{(k(\By-y))}\Big]\\& \qquad \qquad  \qquad
-2k\Big[(y+\By)\cosh{(k(y+\By))}\Big]+2k\Big[(y+\By)\cosh{(2k-k(y+\By))}\Big]+4k^2 \sinh{(k(\By+y))}\\& \qquad \qquad  \qquad+4k\Big[\cosh{(k(\By+y))}-\cosh{(k(\By-y))}\Big]+2\Big[\sinh{(k(\By+y))}+\sinh{(2k-k(y+\By))}\Big]\Bigg),
    \end{align*}
    with $D_k:=4k^2 (4k^2-2\cosh{(2k)}+2)<0$.
   \end{lemma}
The next lemma below presents us with the upper bound of the kernel of $\widehat{\psi_k \chi}$.
\begin{lemma}\label{Estimate fourier psi chi}
    Suppose the smooth cutoff function $\chi$ satisfies  \eqref{cutoff function}, \eqref{cutoff regularity} and $\psi$ satisfies \eqref{eq: Stokes2}. Then there exists a kernel $\mathcal{G}(t,k,\eta,\zeta)$ such that
    \begin{equation}\label{bound on the Fourier of Psi}
    \widehat{(\psi_k \chi)}(\eta)=\int_{\mathbb{R}} \mathcal{G}(k,\eta,\zeta) \widehat{\rho}_k(\zeta) \;d\zeta,
    \end{equation}
    with 
    \begin{equation}\label{upperbound for G}
       |\mathcal{G}(k,\eta,\zeta)| \lesssim \min \biggl\{ \frac{|k|}{(k^2+\eta^2)^2}  ,  \frac{|k|}{(k^2+\zeta^2)^2} \biggr\}e^{-\lambda_M|\eta-\zeta|^s},
    \end{equation}
for some $\lambda_M$ determined by $M$ in \eqref{cutoff regularity} and $s=\frac{s_0+1}{2}$. 
\end{lemma}

We postpone the proofs of the above two lemmas and place them in the Appendix~\ref{Appendix A}. Undergoing the linear change of coordinate \eqref{linear change of coordinate}, we display an equivalent statement of \eqref{bound on the Fourier of Psi} in the corollary below. 
\begin{corollary}\label{bound of fourier phi chi}
    Suppose the smooth cutoff function $\chi$ satisfies \eqref{cutoff regularity} and $\phi$ satisfies \eqref{eq: Stokes}. Then there exists a kernel $\mathfrak{G}(t,k,\eta,\zeta)$ such that
    \begin{equation}\label{fourier phi chi}
    \widehat{(\phi_k \chi)}(\eta)=\int_{\mathbb{R}} \mathfrak{G}(t,k,\eta,\zeta) \widehat{\theta}_k(\zeta) \;d\zeta,
    \end{equation}
    with 
    \begin{equation}\label{kernel estimate in new coordinate}
    |\mathfrak{G}(t,k,\eta,\zeta)| \lesssim \min \biggl\{ \frac{|k|}{(k^2+(\zeta-kt)^2)^2},  \frac{|k|}{(k^2+(\eta-kt)^2)^2} \biggr\}e^{-\lambda_M|\eta-\zeta|^s},
    \end{equation}
    with $\lambda_M$ the same as in Lemma \ref{Estimate fourier psi chi}. 
\end{corollary}
\begin{corollary}\label{gevrey norm of phi chi}
    Under the bootstrap hypotheses, it holds that
    \[
    \norm{\phi \chi}_{\mathcal{G}^{\lambda,\sigma-4;\frac{1}{3}}}\lesssim \frac{\epsilon}{\JB{t}^4}.
    \]
\end{corollary}
\begin{proof}
A direct calculation gives that for $k\neq 0$
\begin{align*}
    &\left|e^{\lambda|k,\eta|^{\frac{1}{3}}}|k,\eta|^{\sigma-4} \widehat{\phi_k \chi}\right|=\left|\int_{\mathbb{R}} e^{\lambda(t)|k,\eta|^{\frac{1}{3}}} |k,\eta|^{\sigma-4}\mathfrak{G}(t,k,\eta,\zeta)\widehat{\theta}_k(\zeta)\;d\zeta\right|\\&
    \leq\int_{\mathbb{R}} \frac{e^{\lambda(t)|k,\eta|^{\frac{1}{3}}}}{e^{\lambda(t)|k,\zeta|^{\frac{1}{3}}}} \frac{|k,\eta|^{\sigma-4} }{|k,\zeta|^{\sigma-4} }\mathfrak{G}(t,k,\eta,\zeta) |k,\zeta|^{\sigma-4}  e^{\lambda(t)|k,\zeta|^{\frac{1}{3}}} 
 |\widehat{\theta}_k(\zeta)|\;d\zeta\\&
 \lesssim \int_{\mathbb{R}} e^{\lambda(t)|\eta-\zeta|^{\frac{1}{3}}} \JB{\eta-\zeta}^{\sigma-4} \frac{\JB{\zeta/k}^4}{\JB{\zeta/k}^4|k|^3\JB{\zeta/k-t}^4} e^{-\lambda_M|\eta-\zeta|^s}  |k,\zeta|^{\sigma-4} e^{\lambda(t)|k,\zeta|^{\frac{1}{3}}} 
 |\widehat{\theta}_k(\zeta)|\;d\zeta\\&
 \lesssim \int_{\mathbb{R}} \frac{1}{\JB{t}^4} e^{-\frac12\lambda_M|\eta-\zeta|^s} |k,\zeta|^\sigma e^{\lambda(t)|k,\zeta|^{\frac{1}{3}}} 
 |\widehat{\theta}_k(\zeta)|\;d\zeta.
\end{align*}
Corollary \ref{gevrey norm of phi chi} follows directly from Young's convolution inequality.
\end{proof}
\begin{prop}\label{bounds on L^2 norms}
    Under the assumptions of Lemma~\ref{Estimate fourier psi chi} and using \eqref{fourier phi chi} and \eqref{kernel estimate in new coordinate} it holds that
    
    \begin{equation}\label{bound given by CKLambda and CKlambda}
    \begin{aligned}
        \norm{\Big(\sqrt{\dfrac{\partial_t \Lambda}{\Lambda}}\mathcal{A}^{\Lambda}\mathbb{P}_{v}+\frac{|\nabla|^{\frac{1}{6}}}{\langle t\rangle^{\frac{3}{4}}}\mathcal{A}\Big)\partial_z^{-1}\Delta_L^2  (\phi\chi)_{\ne}}_{L^2}^2
        \lesssim \mathbf{CK}_{\lambda}+\mathbf{CK}_{\Lambda}.
    \end{aligned}
    \end{equation}
    where $\mathbb{P}_vf=(\mathds{1}_{|l|\leq |\xi|}\hat{f}_l(\xi))^{\vee}$. In particular, we have
    \begin{align}\label{up-Atheta}
\norm{\partial_z^{-1}\Delta_L^2 \mathcal{A}(\phi\chi)_{\ne}}_{L^2}\les \|\A\theta\|_{L^2}.
\end{align}
\end{prop}
\begin{proof}
   Let us first prove the upper-bound of the term involving 
 $\frac{\partial_t \Lambda}{\Lambda}$. In view of \eqref{kernel estimate in new coordinate}, we have
    \begin{align}
    \nn&\left| \sqrt{\frac{\partial_t \Lambda(t,\eta)}{\Lambda(t,\eta)}}\mathcal{A}^{\Lambda}_k(t,\eta) \reallywidehat{\Big(\mathbb{P}_v\partial_z^{-1}\Delta_L^2(\phi\chi)_{\ne}\Big)}_k(t,\eta)\right|\\
    \nn\lesssim&
    \left| \int_{\mathbb{R}} e^{-\lambda_M|\eta-\zeta|^s} \sqrt{\frac{\partial_t \Lambda(t,\eta)}{\Lambda(t,\eta)}}   \frac{\mathcal{A}^{\Lambda}_k(t,\eta)}{\mathcal{A}^{\Lambda}_k(t,\zeta)} \mathds{1}_{|k|\le|\eta|}\mathcal{A}^{\Lambda}_k(t,\zeta)\hat{\theta}_k(\zeta) \;d\zeta  \right|.
    \end{align}
    It  suffices to focus on the time interval  
    \begin{equation}\label{time interval Lambda}
        \frac{2|\eta|}{2E(|\eta|^{\frac{2}{3}})+1} \leq t < 2|\eta|,
    \end{equation}
    since outside that interval $\partial_t\Lambda \equiv 0$. This means that there exists $m$ such that $t\in \tilde{\textup{I}}_{m,\eta}$.
    Our analysis will be   split  into the following two cases:

    {\bf Case 1: $|\zeta-\eta|> \frac{1}{6}|\zeta|$.} Noting that for $k\ne0$, $\fr{\JB{k,\eta}}{\JB{k,\zeta}}\approx\fr{|k,\eta|}{|k,\zeta|}\les\JB{\eta-\zeta}$, then  by Lemmas \ref{lem-total growth} and \ref{estimate of Lambda fraction}, we  arrive at
     \begin{align}\label{ratio-A^Lam}
     \frac{\mathcal{A}^{\Lambda}_k(t,\eta)}{\mathcal{A}^{\Lambda}_k(t,\zeta)}&
     \lesssim e^{\lambda(t)|\eta-\zeta|^{\frac{1}{3}}} \frac{\JB{k,\eta}^{\sigma}}{\JB{k,\zeta}^{\sigma}} e^{3\mu|\eta-\zeta|^{\frac{1}{3}}} e^{2C_0|\eta-\zeta|^{\frac{1}{3}}}
     \lesssim    e^{(\lambda(t)+3\mu+2C_0)|\eta-\zeta|^{\frac{1}{3}}}\JB{\eta-\zeta}^{\sigma}.
     \end{align}

     Moreover,  due to $t\in\tilde{\rm I}_{m,\eta}$ and $|\eta|\leq |\zeta|+|\zeta-\eta| |< \frac{7}{6}|\eta-\zeta|$ in this case, we have
     \[
     \sqrt{\frac{\partial_t \Lambda(t,\eta)}{\Lambda(t,\eta)}}\lesssim 1\approx \frac{|\frac{\eta}{m}|}{t}\les \fr{\JB{\eta-\zeta}}{\JB{t}}.
     \]
It follows  that 
     \begin{align}
    \nn &\norm{ \int_{|\zeta-\eta|> \frac{1}{6}|\zeta|} e^{-\lambda_M|\eta-\zeta|^s} \sqrt{\frac{\partial_t \Lambda(t,\eta)}{\Lambda(t,\eta)}}  \frac{\mathcal{A}^{\Lambda}_k(t,\eta)}{\mathcal{A}^{\Lambda}_k(t,\zeta)} \mathcal{A}^{\Lambda}_k(t,\zeta)\hat{\theta}_k(\zeta) \;d\zeta  }_{L^2_{\eta}}\\
     \lesssim& \norm{\frac{|k,\eta|^{\frac{1}{6}}}{\JB{t}}\mathcal{A}^{\Lambda}_k(t,\eta)\hat{\theta}_k(\eta)}_{L^2_{\eta}}\les\sqrt{\CK_{\lambda}}.
     \end{align}
     
    {\bf Case 2: $|\zeta-\eta|\leq \frac{1}{6}|\zeta|$.}  In this case, we know that $\frac{5}{6}|\zeta|\leq |\eta|<\frac{7}{6}|\zeta|$, thus $|\eta|\approx |\zeta|$ and hence $\JB{k,\eta}\approx \JB{k,\zeta}$. Thanks to the fact $|\eta|\approx|\xi|$, by \eqref{Theta R and Theta NR}, \eqref{defintion of Theta k},  Lemmas \ref{estimate of Theta nonresonant} and \ref{estimate of Lambda fraction}, similar to (but much easier than)  the proof of  \ref{ratio-J},  we have 
     \begin{equation}\label{ratio-ALam'}
     \begin{aligned}
     \frac{\mathcal{A}^{\Lambda}_k(t,\eta)}{\mathcal{A}^{\Lambda}_k(t,\zeta)}
     \lesssim e^{(\lambda(t)+3\mu+2C_0)|\eta-\zeta|^{\frac{1}{3}}}.
     \end{aligned}
     \end{equation}

    Combining  the comparability condition $|\eta|\approx|\zeta|$ and \eqref{time interval Lambda}, we are led to considering the following subcases:
    \begin{enumerate}
        \item $2 |\zeta| \leq t < 2|\eta|$. Now we have $|t-\fr{\zeta}{m}|\ge t-|\zeta|\ge \fr{t}{2}$. Therefore,
        \[
        \sqrt{\frac{\partial_t \Lambda(t,\eta)}{\Lambda(t,\eta)}}\lesssim \frac{1}{1+|t-\frac{\eta}{m}|}=\frac{1}{1+|t-\frac{\zeta}{m}|}\frac{1+|t-\frac{\zeta}{m}|}{1+|t-\frac{\eta}{m}|}\lesssim \JB{\eta-\zeta}\lesssim \frac{1}{\JB{t}}\JB{\eta-\zeta}.
        \] 
        
        \item $ \frac{2|\eta|}{2E(|\eta|^{\frac{2}{3}})+1}\leq t \leq \frac{2|\zeta|}{2E(|\zeta|^{\frac{2}{3}})+1}$. Now we have $t\approx |\eta|^{\frac{1}{3}}$ due to $|\eta|\approx|\zeta|$. This in turn implies that $|m|\approx|\eta|^{\fr{2}{3}}$, since $t\approx |\fr{\eta}{m}|$ whenever $t\in\tilde{\rm I}_{m,\eta}$. Hence, by the definition of  $\Lambda$,
        \[
        \sqrt{\frac{\partial_t \Lambda(t,\eta)}{\Lambda(t,\eta)}}\lesssim \sqrt{\frac{|\eta|}{|m|^3}} \lesssim \frac{1}{|\eta|^{\fr{1}{2}}}\lesssim \frac{1}{t^{\fr{3}{2}}}.
        \]
        
        \item $ \frac{2|\zeta|}{2E(|\zeta|^{\frac{2}{3}})+1}\leq t \leq  2|\zeta|$. There exist $n$ and $m$ such that $t\in \tilde{\textup{I}}_{n,\zeta} \cap \tilde{\textup{I}}_{m,\eta}$. This, together with the fact $|\zeta| \approx |\eta|$, implies that $|m|\approx |n|$. Accordingly, $\fr{{\min\Big\{1,\frac{|\eta|}{|m|^3}\Big\}}}{{\min\Big\{1,\frac{|\zeta|}{|n|^3}\Big\}}}\les 1$.  By Lemma~\ref{Scenarios}, it reduces down to investigate the following three cases:
        \begin{itemize}
            \item $m=n$. Recalling the definition of $\Lambda$, we are led to
            \begin{align}\label{exchange-Lambda}
            \sqrt{\frac{\partial_t \Lambda(t,\eta)}{\Lambda(t,\eta)}}= \fr{1+|t-\fr{\zeta}{n}|}{1+|t-\fr{\eta}{m}|}\fr{\sqrt{\min\Big\{1,\frac{|\eta|}{|m|^3}\Big\}}}{\sqrt{\min\Big\{1,\frac{|\zeta|}{|n|^3}\Big\}}}\frac{\sqrt{\min\Big\{1,\frac{|\zeta|}{|n|^3}\Big\}}}{1+|t-\frac{\zeta}{n}|}
            \lesssim\JB{\zeta-\eta} \sqrt{\frac{\partial_t \Lambda(t,\zeta)}{\Lambda(t,\zeta)}} .
            \end{align}
            \item $m\neq n$, $|t-\frac{\eta}{m}|\gtrsim \frac{|\eta|}{|m|^2}$, $|t-\frac{\zeta}{n}|\gtrsim \frac{|\zeta|}{|n|^2}$. Now there holds $|t-\fr{\zeta}{n}|\les \fr{|\zeta|}{|n|^2}\approx \fr{|\eta|}{|m|^2}\les|t-\fr{\eta}{m}|$. We then find that \eqref{exchange-Lambda} still holds without resorting to $\JB{\eta-\xi}$.

            \item $|\zeta-\eta|\geq \frac{|\zeta|}{|n|}\gtrsim t$. Now it is easy to see that 
            \[
            \sqrt{\frac{\partial_t \Lambda(t,\eta)}{\Lambda(t,\eta)}}\lesssim 1\lesssim \frac{|\zeta-\eta|}{t}.
            \]
        \end{itemize}
    \end{enumerate}
    Combining all the estimates above, we  infer that 
     \begin{equation}
     \begin{aligned}
     &\norm{\int_{|\zeta-\eta|\leq \frac{1}{6}\zeta} e^{-\lambda_M|\eta-\zeta|^s} \sqrt{\frac{\partial_t \Lambda(t,\eta)}{\Lambda(t,\eta)}}  \frac{\mathcal{A}^{\Lambda}_k(t,\eta)}{\mathcal{A}^{\Lambda}_k(t,\zeta)} \mathcal{A}^{\Lambda}_k(t,\zeta)\hat{\theta}_k(\zeta) \;d\zeta  }_{L^2_{\eta}}\\& \qquad \qquad \qquad \lesssim \norm{\bigg(\frac{|k,\eta|^{\frac{1}{6}}}{\JB{t}}+   \sqrt{\frac{\partial_t \Lambda(t,\eta)}{\Lambda(t,\eta)}}\bigg)\mathcal{A}^{\Lambda}_k(t,\eta)\hat{\theta}_k(\eta)}_{L^2_{\eta}}\lesssim \sqrt{\CK_{\lambda}} + \sqrt{\CK_{\Lambda}}.
     \end{aligned}
     \end{equation}

 As for the remaining term in \eqref{bound given by CKLambda and CKlambda}, we infer from \eqref{kernel estimate in new coordinate} that
 \[
 \left|\frac{|\nabla|^{\frac{1}{6}}}{\langle t\rangle^{\frac{3}{4}}}\mathcal{A}_k(t,\eta)\mathcal{F}\Big[\partial_z^{-1}\Delta_L^2  (\phi\chi)_{\ne}\Big]_k(t,\eta)\right|\lesssim \left| \int_{\mathbb{R}} e^{-\lambda_M|\eta-\zeta|^s} \frac{|k,\eta|^{\frac{1}{6}}}{|k,\zeta|^{\frac{1}{6}}}  \frac{\mathcal{A}_k(t,\eta)}{\mathcal{A}_k(t,\zeta)}  \frac{|k,\zeta|^{\frac{1}{6}}}{\langle t\rangle^{\frac{3}{4}}}\mathcal{A}_k(t,\zeta)\hat{\theta}_k(\zeta) \;d\zeta  \right|.
 \]
 Then using again the fact $\fr{|k,\eta|}{|k,\zeta|}\les \JB{\eta-\zeta}$ and noting that \eqref{ratio-A^Lam} still holds with $\A^\Lambda$ replaced by $\A$ regardless of $|\zeta-\eta|>\fr{1}{6}|\zeta|$ or not, we immediately have 
 \begin{equation}\label{up-CKlambda}
\norm{\frac{|\nabla|^{\frac{1}{6}}}{\langle t\rangle^{\frac{3}{4}}}\mathcal{A}\partial_z^{-1}\Delta_L^2  (\phi\chi)_{\ne}}_{L^2}\lesssim \norm{\frac{|k,\eta|^{\frac{1}{6}}}{\JB{t}^{\frac{3}{4}}} \mathcal{A}_k(t,\eta)\hat{\theta}_k(\eta)}_{L^2_{\eta}}\lesssim \sqrt{\CK_\lambda}.
 \end{equation}
Now gluing all estimates yields the desired inequality in \eqref{bound given by CKLambda and CKlambda}. The estimate \eqref{up-Atheta} can be obtained similarly to \eqref{up-CKlambda}. This completes the proof of Proposition \ref{bounds on L^2 norms}.
\end{proof}

\section{Estimate of Nonlinear interactions}\label{nonlinear interactions}
In this section, we employ all the estimates derived in Sections~\ref{properties of multipliers} and ~\ref{stokes} to control the nonlinear interaction in \eqref{Nonlinear term}. For readability sake, we rewrite the nonlinear term of interest below again
\begin{equation}
    \textbf{NL}_\theta=\int \A \theta \big[\A(\nabla^{\perp}(\phi\chi)\cdot \nabla \theta)-\nabla^{\perp}(\phi\chi)\cdot \nabla \A \theta \big]\;dzdv
\end{equation}
where more refined decomposition of $ \textbf{NL}_\theta$ is displayed in \eqref{paraproduct decomposition}.

\subsection{Reaction Term}\label{Reaction term}
Here, we treat the reaction term resulting from the paraproduct decomposition \eqref{paraproduct decomposition}. On the Fourier side via Plancherel, the reaction term reads
\[
 \textbf{R}_N=2\pi\big(\textbf{R}^1_N + \textbf{R}^2_N \big), 
\]
where
\begin{align}
  &\textbf{R}^1_N =-\sum_{k\in\mathbb{Z},l\neq 0} \int_{\eta,\xi} \A_k(\eta) \overline{\widehat{\theta}}_k(\eta)\A_k(\eta)\widehat{\phi_l\chi}(\xi)_N (-\xi,l) \cdot (k-l,\eta-\xi) \widehat{\theta}_{k-l}(\eta-\xi)_{<N/8}\;d\eta\; d\xi, 
  \\&\textbf{R}^2_N =\sum_{k\in\mathbb{Z},l\neq 0} \int_{\eta,\xi} \A_k(\eta) \overline{\widehat{\theta}}_k(\eta)\widehat{\phi_l\chi}(\xi)_N (-\xi,l) \cdot (k-l,\eta-\xi)  \widehat{\A\theta}_{k-l}(\eta-\xi)_{<N/8}\;d\eta\; d\xi.
\end{align}

As we will show below, the term $\textbf{R}^1_N$ plays the most important role in comparison to its counterpart  $\textbf{R}^2_N$ in estimating $\textbf{R}_N$. 
In order to clarify the frequency regime where we perform our analyses, we define the following combination of characteristic functions,
\[
\begin{aligned}
1&=\mathds{1}_{t\notin \textup{I}_{k,\eta},t\notin \textup{I}_{l,\xi}}+\mathds{1}_{t\notin \textup{I}_{k,\eta},t\in \textup{I}_{l,\xi}}+\mathds{1}_{t\in \textup{I}_{k,\eta},t\notin \textup{I}_{l,\xi}}+\mathds{1}_{t\in \textup{I}_{k,\eta},t\in \textup{I}_{l,\xi}}
\\&=: \chi^{\NR,\NR}+ \chi^{\NR,\R}+ \chi^{\R,\NR}+ \chi^{\R,\R}.
\end{aligned}
\]
As a result, 
\[
\begin{aligned}
\textbf{R}^1_N&=-\sum_{k\in\mathbb{Z},l\neq 0} \int_{\eta,\xi} \Big[\chi^{\NR,\NR}+ \chi^{\NR,\R}+ \chi^{\R,\NR}+ \chi^{\R,\R} \Big]\\& \qquad \qquad \qquad \qquad  \times \A_k(\eta) \overline{\widehat{\theta}}_k(\eta)\A_k(\eta)\widehat{\phi_l\chi}(\xi)_N (-\xi,l) \cdot (k-l,\eta-\xi) \widehat{\theta}_{k-l}(\eta-\xi)_{<N/8}\;d\eta d\xi\\&
=\textbf{R}^1_{N;\NR,\NR}+\textbf{R}^1_{N;\NR,\R}+\textbf{R}^1_{N;\R,\NR}+\textbf{R}^1_{N;\R,\R}.
\end{aligned}
\]
Furthermore,  the frequency localization gives
\[
\begin{aligned}
\frac{N}{2}\leq |l,\xi|\leq \frac{3N}{2}\text{ and } 
|k-l,\eta-\xi|\leq \frac{3N}{32}.
\end{aligned}
\]
Noting also that $l\ne0$, we then have $((k,\eta),(l,\xi))\in\mathfrak{A}$ defined in \eqref{Set A} and hence \eqref{comparability} holds.    
We would like to mention that this is the main motivation why we introduce set $\mathfrak{A}$ in \eqref{Set A}.

We are now in a position to treat each term in $\textbf{R}^1_N$. Mainly, we will refer back to estimates we have derived in Section~\ref{properties of multipliers}. We  begin by estimating $\textbf{R}^1_{N;\NR,\NR}$.
\subsubsection{{Treatment of} ${\bf R}^1_{N;\NR,\NR}$} \label{NR NR}
Using merely the expression of $\textbf{R}^1_{N;\NR,\NR}$, we have the following inequality
\begin{equation}\label{Ineq for NR NR}
\begin{aligned}
|\textbf{R}^1_{N;\NR,\NR}|
\leq& \sum_{k\in\mathbb{Z},l\neq0} \int_{\eta,\xi} \chi^{\NR,\NR} |\A_k(\eta) \overline{\widehat{\theta}}_k(\eta)| \dfrac{|l||l,\xi|}{(l^2+(\xi-tl)^2)^2} \dfrac{\A_k(\eta)}{\A_l(\xi)}  \\&
\qquad\qquad   \times |\A_l(\xi)\reallywidehat{\partial_z^{-1}\Delta_L^2 \phi_l \chi}(\xi)_N| |k-l,\eta-\xi| |\widehat{\theta}_{k-l}(\eta-\xi)_{<N/8}|\;d\eta d\xi.
\end{aligned}
\end{equation}
Applying Lemma~\ref{NR-NR}, H\"{o}lder inequality, and Proposition~\ref{bounds on L^2 norms}, we obtain
\[
\begin{aligned}
\sum_{N\geq 8}|\textbf{R}^1_{N;\NR,\NR}|&\lesssim
\sum_{N\geq 8}\dfrac{\epsilon}{\langle t \rangle^{\frac{3}{2}}} \norm{|\nabla|^{\frac{1}{6}}\mathcal{A}\theta_{\sim N}}_{L^2} \norm{|\nabla|^{\frac{1}{6}}\partial_z^{-1}\Delta_L^2 \mathcal{A} \mathbb{P}_{\ne}(\phi\chi)_{ N}}_{L^2}\\&
\qquad + \sum_{N\geq 8}\epsilon \norm{\sqrt{\dfrac{\partial_t \Lambda}{\Lambda}}\mathcal{A}\theta_{\sim N}}_{L^2} \norm{\sqrt{\dfrac{\partial_t \Lambda}{\Lambda}}\mathcal{A}^{\Lambda}\mathbb{P}_v\partial_z^{-1}\Delta_L^2 \mathbb{P}_{\ne} (\phi\chi)_{ N}}_{L^2}\\&
\lesssim \epsilon \big(\CK_\lambda + \CK_\Lambda \big).
\end{aligned}
\]

\subsubsection{{Treatment of }${\bf R}^1_{N;\R,\NR}$} Here, we are considering the case when $t\in \textup{I}_{k,\eta} \cap \textup{I}^{c}_{l,\xi}$. Using a similar inequality as in \eqref{Ineq for NR NR} (except now the term $\chi^{\NR,\NR}$ is replaced by $\chi^{\R,\NR}$) and combining it with the estimate from Lemma~\ref{R-NR}, H\"{o}lder inequality, and Proposition~\ref{bounds on L^2 norms}, we arrive at
\begin{equation}\label{ineq R NR}
\begin{aligned} 
\sum_{N\geq 8}|\textbf{R}^1_{N;\R,\NR}| &
\lesssim \sum_{N\geq 8} \dfrac{\epsilon}{\JB{t}^2} \norm{\mathcal{A}\theta_{\sim N}}_{L^2} \norm{\partial_z^{-1}\Delta_L^2 \mathcal{A}\mathbb{P}_{\ne}(\phi\chi)_{  N}}_{L^2}\\&
\qquad + \sum_{N\geq 8} \epsilon \norm{\sqrt{\dfrac{\partial_t \Lambda}{\Lambda}}\mathcal{A}^{\Lambda}\theta_{\sim N}}_{L^2} \norm{\sqrt{\dfrac{\partial_t \Lambda}{\Lambda}}\mathcal{A}^{\Lambda} \mathbb{P}_v\partial_z^{-1}\Delta_L^2 \mathbb{P}_{\ne}(\phi\chi)_{N}}_{L^2}\\&
\lesssim \frac{\epsilon^3}{\JB{t}^2}+\epsilon (\CK_\lambda+\CK_\Lambda).
\end{aligned}
\end{equation}

\subsubsection{{Treatment of} ${\bf R}^1_{N;\NR,\R}$}
  Here, $t \in \text{I}^{c}_{k,\eta} \cap  \text{I}_{l,\xi}$. 
  Similar to \eqref{ineq R NR}, using the estimate in Lemma~\ref{NR-R}, H\"{o}lder inequality,  and Proposition~\ref{bounds on L^2 norms}, we obtain
 \begin{equation}\label{ineq NR R}
\begin{aligned}
\sum_{N\geq 8}|\textbf{R}^1_{N;\NR,\R}|
\lesssim \frac{\epsilon^3}{\JB{t}^2}+\epsilon (\CK_\lambda+\CK_\Lambda).
\end{aligned}
\end{equation}

\subsubsection{{Treatment of} ${\bf R}^1_{N;\R,\R}$} Now we run into the situation where both $(l,\xi)$ and $(k,\eta)$ are resonant.
Using the estimate in Lemma~\ref{R-R} and Proposition~\ref{bounds on L^2 norms}, we  infer that 
\begin{equation}\label{ineq R R}
\begin{aligned}
\sum_{N\geq 8}\mathbf{R}^1_{N;\R,\R}|&
\lesssim \sum_{N\geq 8}  \epsilon \norm{\sqrt{\dfrac{\partial_t \Lambda}{\Lambda}}\mathcal{A}^{\Lambda}\theta_{\sim N}}_{L^2} \norm{\sqrt{\dfrac{\partial_t \Lambda}{\Lambda}}\mathcal{A}^{\Lambda} \mathbb{P}_{v}\partial_z^{-1}\Delta_L^2 \mathbb{P}_{\ne}(\phi\chi)_{ N}}_{L^2}\\&
\lesssim \epsilon (\CK_\lambda+\CK_\Lambda).
\end{aligned}
\end{equation}

Having established the estimate for $\textbf{R}^1_N$, now we turn to its counterpart term $\textbf{R}^2_N$.
\subsubsection{{Treatment of}  ${\bf R}^2_N$}
First of all, note that on the support of the integrand of $\textbf{R}^2_N$, it is true that $|k-l,\eta-\xi|\lesssim |l,\xi|$. Hence, in view of Lemma \ref{Convolution estimates}, Corollary \ref{gevrey norm of phi chi} and the bootstrap hypotheses, we are led to
\[
\begin{aligned}
\sum_{N\geq 8}|\textbf{R}^2_N|&\lesssim \sum_{N\geq 8}\sum_{k\in\mathbb{Z},l\neq 0} \int_{\eta,\xi} |\A_k(\eta) \overline{\widehat{\theta}}_k(\eta)| |\widehat{\phi_l\chi}(\xi)_N| |(-\xi,l)\cdot(k-l,\eta-\xi)|  |\A\widehat{\theta}_{k-l}(\eta-\xi)_{<N/8}|\;d\eta d\xi,\\&
\lesssim \sum_{N\geq 8} \sum_{k\in\mathbb{Z},l\neq 0} \int_{\eta,\xi} |\A_k(\eta) \overline{\widehat{\theta}}_k(\eta)| |\widehat{\phi_l\chi}(\xi)_N|| (l,\xi)|^2|\A\widehat{\theta}_{k-l}(\eta-\xi)_{<N/8}|\;d\eta d\xi,\\&
\lesssim \sum_{N\geq 8}\norm{\A\theta_{\sim N}}_{L^2}\norm{(\phi \chi)_N}_{\mathcal{G}^{\lambda,\sigma-4;\frac{1}{3}}} \norm{\A \theta}_{L^2}\lesssim \frac{\epsilon^3}{\JB{t}^4}.
\end{aligned}
\]
The estimates for the reaction term are completed.
\subsection{Transport Term}\label{Transport term}
The present section is fully devoted to presenting the estimate of the transport term $\textbf{T}_N$ appearing in \eqref{Transport Reaction Remainder}. On the Fourier side, the transport term reads
\[
\begin{aligned}
\textbf{T}_N&
=i\sum_{k,l \in \mathbb{Z}} \int_{\eta,\xi} \A_k(t,\eta) \overline{\widehat{\theta}}_k(\eta) \Big[\A_k(t,\eta)-\A_l(t,\xi)\Big]\Big(\reallywidehat{\nabla^{\perp}(\phi_{k-l}\chi)}(\eta-\xi)_{<N/8} \cdot(l,\xi) \widehat{\theta}_N(l,\xi)\Big)
\;d\xi\;d\eta.
\end{aligned}
\]
To proceed with the proof, it is easy to verify that
\begin{equation}
\begin{aligned}
    \A_k(t,\eta)-\A_l(t,\xi)&=\A_l(t,\xi))\big[e^{\lambda(t)|k,\eta|^{\frac{1}{3}}-\lambda(t)|l,\xi|^{\frac{1}{3}}}-1\big]\\&
    \quad + \A_l(t,\xi)e^{\lambda(t)|k,\eta|^{\frac{1}{3}}-\lambda(t)|l,\xi|^{\frac{1}{3}}}\bigg[\dfrac{\J_k(\eta)}{\J_l(\xi)}-1\bigg]\dfrac{\M_k(\eta)}{\M_l(\xi)}\dfrac{\JB{k,\eta}^\sigma}{\JB{l,\xi}^\sigma}\\&
    \quad + \A_l(t,\xi)e^{\lambda(t)|k,\eta|^{\frac{1}{3}}-\lambda(t)|l,\xi|^{\frac{1}{3}}}\bigg[\dfrac{\M_k(\eta)}{\M_l(\xi)}-1\bigg]\dfrac{\JB{k,\eta}^\sigma}{\JB{l,\xi}^\sigma}\\&
    \quad + \A_l(t,\xi)e^{\lambda(t)|k,\eta|^{\frac{1}{3}}-\lambda(t)|l,\xi|^{\frac{1}{3}}}\bigg[\dfrac{\JB{k,\eta}^\sigma}{\JB{l,\xi}^\sigma}-1\bigg].
\end{aligned}
\end{equation}
As a consequence of this, we decompose $\textbf{T}_N$ as follows
\[
\textbf{T}_N=\textbf{T}_{N,1}+\textbf{T}_{N,2}+\textbf{T}_{N,3}+\textbf{T}_{N,4},
\]
where

\begin{align*}
    &\textbf{T}_{N,1}
    =i\sum_{k,l \in \mathbb{Z}} \int_{\eta,\xi} \A_k(t,\eta) \overline{\widehat{\theta}}_k(\eta) \Big[\A_l(t,\xi))\big[e^{\lambda(t)|k,\eta|^{\frac{1}{3}}-\lambda(t)|l,\xi|^{\frac{1}{3}}}-1\big]\Big]\\
    &\qquad\qquad\qquad\times
    \Big(\reallywidehat{\nabla^{\perp}(\phi_{k-l}\chi)}(\eta-\xi)_{<N/8}  \cdot(l,\xi) \widehat{\theta}_l(\xi)_N\Big)
\;d\xi\;d\eta,\\&
   \textbf{T}_{N,2}=i\sum_{k,l \in \mathbb{Z}} \int_{\eta,\xi} \A_k(t,\eta) \overline{\widehat{\theta}}_k(\eta) \Bigg[\A_l(t,\xi)e^{\lambda(t)|k,\eta|^{\frac{1}{3}}-\lambda(t)|l,\xi|^{\frac{1}{3}}}\Big[\dfrac{\J_k(\eta)}{\J_l(\xi)}-1\Big]\dfrac{\M_k(\eta)}{\M_l(\xi)}\dfrac{\JB{k,\eta}^\sigma}{\JB{l,\xi}^\sigma}\Bigg]\\&
   \qquad \qquad \qquad \times \Big(\reallywidehat{\nabla^{\perp}(\phi_{k-l}\chi)}(\eta-\xi)_{<N/8} \cdot(l,\xi) \widehat{\theta}_l(\xi)_N\Big)
\;d\xi\;d\eta,\\&
   \textbf{T}_{N,3}=i\sum_{k,l \in \mathbb{Z}} \int_{\eta,\xi} \A_k(t,\eta) \overline{\widehat{\theta}}_k(\eta) \Bigg[\A_l(t,\xi)e^{\lambda(t)|k,\eta|^{\frac{1}{3}}-\lambda(t)|l,\xi|^{\frac{1}{3}}}\big[\dfrac{\M_k(\eta)}{\M_l(\xi)}-1\big]\dfrac{\JB{k,\eta}^\sigma}{\JB{l,\xi}^\sigma}\Bigg]\\&
   \qquad \qquad \qquad \times \Big(\reallywidehat{\nabla^{\perp}(\phi_{k-l}\chi)}(\eta-\xi)_{<N/8}  \cdot(l,\xi) \widehat{\theta}_l(\xi)_N\Big)
\;d\xi\;d\eta,\\&
   \textbf{T}_{N,4}=i\sum_{k,l \in \mathbb{Z}} \int_{\eta,\xi} \A_k(t,\eta) \overline{\widehat{\theta}}_k(\eta) \Bigg[\A_l(t,\xi)e^{\lambda(t)|k,\eta|^{\frac{1}{3}}-\lambda(t)|l,\xi|^{\frac{1}{3}}}\Big[\dfrac{\JB{k,\eta}^\sigma}{\JB{l,\xi}^\sigma}-1\Big]\Bigg]\\&
   \qquad \qquad \qquad \times \Big(\reallywidehat{\nabla^{\perp}(\phi_{k-l}\chi)}(\eta-\xi)_{<N/8}  \cdot(l,\xi) \widehat{\theta}_l(\xi)_N\Big)
\;d\xi\;d\eta.
\end{align*}

\subsubsection{{Treatment of }${\bf T}_{N,1}$} Using the fact that $|e^x-1|\leq |x|e^{|x|}$, $\textbf{T}_{N,1}$ can be estimated in the following manner
\begin{align*}
|\textbf{T}_{N,1}|&\leq \sum_{k,l \in \mathbb{Z}} \int_{\eta,\xi} |\A\widehat{\theta}_k(\eta)| |\reallywidehat{\nabla^{\perp}(\phi_{k-l}\chi)}(\eta-\xi)_{<N/8} |\lambda(t)\Big||k,\eta|^{\frac{1}{3}}-|l,\xi|^{\frac{1}{3}}\Big|\\& \qquad \qquad \qquad \times e^{\lambda(t)|k,\eta|^{\frac{1}{3}}-\lambda(t)|l,\xi|^{\frac{1}{3}}}|l,\xi| \A_l(t,\xi))|\widehat{\theta}_l(\xi)_N|
\;d\xi\;d\eta.
\end{align*}
Notice that on the support of the integrand, the comparability condition in \eqref{comparability} holds. Hence, 
$\Big||k,\eta|^{\frac{1}{3}}-|l,\xi|^{\frac{1}{3}}\Big||l,\xi|\les |k-l,\eta-\xi||k,\eta|^{\frac{1}{6}}|l,\xi|^{\frac{1}{6}}$. Then thanks to Corollary~\ref{gevrey norm of phi chi}, one has
\begin{align*}
|\textbf{T}_{N,1}|
\lesssim& \lambda(t) \sum_{k,l \in \mathbb{Z}} \int_{\eta,\xi} |(k,\eta)|^{\frac{1}{6}}|\A_k(t,\eta) \widehat{\theta}_k(\eta)| \Big|\reallywidehat{\nabla^{\perp}(\phi_{k-l}\chi)}_{<N/8}(\eta-\xi)  e^{\lambda(t)|k-l,\eta-\xi|^{\frac{1}{3}}}\Big|\\
&\qquad \qquad \times|l,\xi|^{\frac{1}{6}} \A_l(t,\xi))|\widehat{\theta}_l(\xi)_N|
\;d\xi\;d\eta\\
\lesssim& \lambda(t)  \norm{|\nabla|^{\frac{1}{6}}\A \theta_N }_{L^2} \norm{\phi\chi}_{{\mathcal {G}}^{\lambda,\sigma-4;\fr{1}{3}}} \norm{|\nabla|^{\frac{1}{6}}\A \theta_{\sim N}}_{L^2}\\
\lesssim& \dfrac{\epsilon}{\JB{t}^{4}}  \norm{|\nabla|^{\frac{1}{6}}\A \theta_N }_{L^2} \norm{|\nabla|^{\frac{1}{6}}\A \theta_{\sim N}}_{L^2} .
\end{align*}

\subsubsection{{Treatment of } ${\bf T}_{N,2}$} The treatment of this portion is rather the most technical one compared to other terms in $\textbf{T}_{N}$. This is mainly due to the presence of the multiplier components $\J$ and $\M$. As the first step towards this analysis, let us define 
\begin{equation}\label{definition of chis chil}
\chi_S:=\mathds{1}_{t\leq \frac{1}{2} \min\{|\xi|^{\frac{2}{3}},|\eta|^{\frac{2}{3}}\}},\qquad \chi_L:=1-\chi_S.
\end{equation}
Thus,
\[
\begin{aligned}
\textbf{T}_{N,2}&=i\sum_{k,l \in \mathbb{Z}} \int_{\eta,\xi} \A_k(t,\eta) \overline{\widehat{\theta}}_k(\eta) \Bigg[e^{\lambda(t)|k,\eta|^{\frac{1}{3}}-\lambda(t)|l,\xi|^{\frac{1}{3}}}\Big[\dfrac{\J_k(\eta)}{\J_l(\xi)}-1\Big]\dfrac{\M_k(\eta)}{\M_l(\xi)}\dfrac{\JB{k,\eta}^\sigma}{\JB{l,\xi}^\sigma}\Bigg]\\&
   \qquad \qquad \qquad \times \Big[\chi_S+\chi_L\Big] \Big(\reallywidehat{\nabla^{\perp}(\phi_{k-l}\chi)}(\eta-\xi)_{<N/8} \cdot(l,\xi) \A_l(t,\xi)\widehat{\theta}_l(\xi)_N\Big)
\;d\xi\;d\eta\\&
:=\textbf{T}_{N,2,S}+\textbf{T}_{N,2,L}.
\end{aligned}
\]
Now, let us estimate $\textbf{T}_{N,2,S}$, that is when $t\leq \fr12\min\{|\xi|^{\frac{2}{3}},|\eta|^{\frac{2}{3}}\}$. In fact, after recalling the definition of $\mathcal{M}_k(\eta)$ in \eqref{Multiplier components}, and  using Lemma \ref{estimate of Lambda fraction}, we have
\begin{equation}\label{ratio-M}
    \fr{\mathcal{M}_k(\eta)}{\mathcal{M}_l(\xi)}\les e^{2C_0|k-l,\eta-\xi|^{\fr{1}{3}}}.
\end{equation}
Combining this with Lemma~\ref{estimate on fraction of J},  the fact that $|l,\xi|\approx |k,\eta|$, and Corollary \ref{gevrey norm of phi chi},  we are led to
\begin{align*}
|\textbf{T}_{N,2,S}|
&\lesssim \sum_{k,l \in\mathbb{Z}} \int_{\eta,\xi} \chi_S |\A\widehat{\theta}_k(\eta)| \Bigg|e^{\lambda(t)(|k,\eta|-|l,\xi|)^{\frac{1}{3}}} \dfrac{\JB{k-l,\xi-\eta}}{(|k|+|l|+|\eta|+|\xi|)^{\frac{2}{3}}} e^{2C_0 |k-l,\eta-\xi|^{\frac{1}{3}}}\Bigg|\\&
   \qquad \qquad \qquad \times |l,\xi||\reallywidehat{\nabla^{\perp}(\phi_{k-l}\chi)}(\eta-\xi)_{<N/8}|  \A\widehat{\theta}_l(\xi)_N|
\;d\xi\;d\eta\\&
\lesssim \sum_{k,l \in\mathbb{Z}} \int_{\eta,\xi}  |\A \widehat{\theta}_k(\eta)| |l,\xi|^{\frac{1}{6}}|k,\eta|^{\frac{1}{6}} e^{{\bf c}\lambda(t)|k-l,\eta-\xi|^{\frac{1}{3}}} \\&
   \qquad \qquad \qquad \times |\reallywidehat{\nabla^{\perp}(\phi_{k-l}\chi)}(\eta-\xi_{<N/8}|  |\A\widehat{\theta}_l(\xi)_N|
\;d\xi\;d\eta\\&
\lesssim \dfrac{\epsilon}{\JB{t}^{4}}  \norm{|\nabla|^{\frac{1}{6}}\A \theta_N }_{L^2} \norm{|\nabla|^{\frac{1}{6}}\A \theta_{\sim N}}_{L^2},
\end{align*}
where ${\bf c}\in(0,1)$.
Now consider  the more difficult term $\textbf{T}_{N,2,L}$, which can be further split as follows:
\[
\begin{aligned}
\textbf{T}_{N,2,L}&=i\sum_{k,l \in\mathbb{Z}} \int_{\eta,\xi} \chi_L \A_k(t,\eta) \overline{\widehat{\theta}}_k(\eta) \Bigg[e^{\lambda(t)(|k,\eta|^{\fr{1}{3}}-|l,\xi|^{\fr{1}{3}})}\Big[\dfrac{\J_k(\eta)}{\J_l(\xi)}-1\Big]\dfrac{\M_k(\eta)}{\M_l(\xi)}\dfrac{\JB{k,\eta}^\sigma}{\JB{l,\xi}^\sigma}\Bigg]\\&
   \qquad \qquad \qquad \times\Big[\chi^{\circled{1}}+\chi^{\circled{2}}\Big]  \Big(\reallywidehat{\nabla^{\perp}(\phi_{k-l}\chi)}(\eta-\xi)_{<N/8}\cdot(l,\xi) \A_l(t,\xi)\widehat{\theta}_l(\xi)_N\Big)
\;d\xi\;d\eta\\&
:=\textbf{T}^{\circled{1}}_{N,2,L}+\textbf{T}^{\circled{2}}_{N,2,L}
\end{aligned}
\]
where $\chi^{\circled{1}}=\mathds{1}_{{|l|\leq 4|\xi|}},\  \chi^{\circled{2}}=\mathds{1}_{{|l|\geq 4|\xi|}}$. 

Let us now focus on $\textbf{T}^{\circled{2}}_{N,2}$. From the frequency restrictions  $|k-l,\eta-\xi|\le \fr{3}{16}|l,\xi|$ and  $|l|\geq 4|\xi|$, we know that 
\begin{equation}\label{l dominates eta}
|\eta|\le|\xi|+|\eta-\xi| \leq \dfrac{31}{64}|l|.
\end{equation}
Recalling the definition of $\mathcal{J}_k(\eta)$  in \eqref{Multiplier components}, Lemma \ref{lem-total growth} and the elementary inequality $|e^x-1|\le|x|e^{|x|}$ again, we have 
\begin{equation}\label{l-dominate}
\begin{aligned}
\bigg|\dfrac{\J_k(\eta)}{\J_l(\xi)}-1\bigg|\leq &\dfrac{|\tilde{\J}_k(\eta)-\tilde{\J}_l(\xi)|}{\tilde{\J}_l(\xi)+e^{\mu |l|^{\fr{1}{3}}}}+\dfrac{|e^{\mu |k|^{\fr{1}{3}}}-e^{\mu|l|^{\fr{1}{3}}}|}{e^{\mu|l|^{\fr{1}{3}}}+\tilde{\J}_l(\xi)}\\
\lesssim&\fr{e^{\fr{21}{20}\mu|\eta|^{\fr13}}+e^{\fr{21}{20}\mu|\xi|^{\fr13}}}{e^{\mu|l|^{\fr13}}}+|e^{\mu(|k|^{\fr{1}{3}}-|l|^{\fr{1}{3}})}-1|\\
\lesssim&\fr{1}{|l|^{\fr{2}{3}}}+\fr{|k-l|}{|k|^{\fr{2}{3}}+|l|^{\fr{2}{3}}}e^{\mu|k-l|^{\fr{1}{3}}}\lesssim \fr{\langle k-l\rangle}{|l|^{\fr{2}{3}}}e^{\mu|k-l|^{\fr{1}{3}}}.
\end{aligned}
\end{equation}
Combining  this with \eqref{ratio-M}, \eqref{triangle1}, and noting that $|k,\eta|\approx|l,\xi|\lesssim|l|$, we find that there exist two constants  $0<c<{\bf c}<1$, such that
\begin{align*}
    |\textbf{T}^{\circled{2}}_{N,2,L}|
&\lesssim\sum_{k,l \in\mathbb{Z}} \int_{\eta,\xi} \chi_L  |\A\widehat{\theta}_k(\eta)| |\A\widehat{\theta}_l(\xi)_N||l|^{\fr{1}{3}}\langle k-l\rangle e^{(\mu+2C_0)|k-l,\eta-\xi|^{\frac{1}{3}}}\\&
   \qquad \qquad \qquad \times e^{c\lambda(t)|k-l,\eta-\xi|^{\fr{1}{3}}}  \Big(|\reallywidehat{\nabla^{\perp}(\phi_{k-l}\chi)}(\eta-\xi)_{<N/8}| \Big)
\;d\xi\;d\eta\\&
\lesssim\sum_{k,l \in\mathbb{Z}} \int_{\eta,\xi} \chi_L |k,\eta|^{\fr{1}{6}} |\A\widehat{\theta}_k(\eta)| \Big[|l,\xi|^{\fr{1}{6}}|\A\widehat{\theta}_l(\xi)_N|\Big]\\&
   \qquad \qquad \qquad \times \Big(e^{{\bf c}\lambda(t)|k-l,\eta-\xi|^{\fr{1}{3}}}  |\reallywidehat{\nabla^{\perp}(\phi_{k-l}\chi)}(\eta-\xi)_{<N/8}| \Big)
\;d\xi\;d\eta\\&
\lesssim \dfrac{\epsilon}{\JB{t}^{4}}  \norm{|\nabla|^{\frac{1}{6}}\A \theta_N }_{L^2} \norm{|\nabla|^{\frac{1}{6}}\A \theta_{\sim N}}_{L^2}.
\end{align*}

Let us now investigate  $\textbf{T}^{\circled{1}}_{N,2,L}$. Noting that $|l|\le 4|\xi|$,  we thus have $|\eta-\xi|\leq \frac{3}{16}|4\xi,\xi|\leq \frac{15}{16}|\xi|$, which gives that $|\xi|\approx |\eta|$. Then we infer from the restriction $t>\frac{1}{2}\min\{|\eta|^{\frac{2}{3}}, |\xi|^{\frac{2}{3}}\}$ and \eqref{comparability} that $|l,\xi|\lesssim|l,\xi|^{\frac{1}{6}}|k,\eta|^{\frac{1}{6}}t$. On the other hand, Lemma \ref{ratio of J} shows that 
we always have $\fr{\mathcal{J}_k(t,\eta)}{\mathcal{J}_l(t,\xi)}\lesssim\langle t\rangle e^{3\mu|k-l,\eta-\xi|^{\fr{1}{3}}}$ since the worst scenario happens when $t\in{\rm I}_{k,\eta}$ which in turn implies that $\fr{|\eta|}{|k|}\approx t$.
Consequently, there exist two constants  $0<c<{\bf c}<1$, such that
\begin{align*}
    |\textbf{T}^{\circled{1}}_{N,2,L}|&\lesssim \langle t\rangle^{2}\sum_{k,l \in\mathbb{Z}} \int_{\eta,\xi} \chi_L |k,\eta|^{\fr{1}{6}} |\A\widehat{\theta}_k(\eta)| |l,\xi|^{\fr{1}{6}}|\A\widehat{\theta}_l(\xi)_N| e^{(3\mu+2C_0)|k-l,\eta-\xi|^{\frac{1}{3}}}\\&
   \qquad \qquad \qquad \times  e^{c\lambda(t)|k-l,\eta-\xi|^{\fr{1}{3}}} |\reallywidehat{\nabla^{\perp}(\phi_{k-l}\chi)}(\eta-\xi)_{<N/8}| 
\;d\xi\;d\eta\\
 &\lesssim \langle t\rangle^{2}\sum_{k,l \in\mathbb{Z}} \int_{\eta,\xi} \chi_L |k,\eta|^{\fr{1}{6}} |\A\widehat{\theta}_k(\eta)| |l,\xi|^{\fr{1}{6}}|\A\widehat{\theta}_l(\xi)_N|\\&
   \qquad \qquad \qquad \times  e^{{\bf c}\lambda(t)|k-l,\eta-\xi|^{1/3}} |\reallywidehat{\nabla^{\perp}(\phi_{k-l}\chi)}(\eta-\xi)_{<N/8}| 
\;d\xi\;d\eta\\&
\lesssim \dfrac{\epsilon}{\JB{t}^{2}}\norm{|\nabla|^{\fr{1}{6}}\A\theta_{\sim N}}_{L^2} \norm{|\nabla|^{\fr{1}{6}}\A\theta_N}_{L^2}.
\end{align*}

\subsubsection{{Treatment of }${\bf T}_{N,3}$}
Let us now proceed to estimate $ \textbf{T}_{N,3}$. Recall the definition of $\chi_S$ and $\chi_L$ in \eqref{definition of chis chil}, in the similar manner we define
\begin{equation}\label{definition of tildes chis chil}
\tilde{\chi}_S:=\mathds{1}_{t\leq \frac{1}{2} \min\{|\xi|^{\frac{1}{3}},|\eta|^{\frac{1}{3}}\}},\qquad \tilde{\chi}_L:=1-\tilde{\chi}_S.
\end{equation}
In light of $\tilde{\chi}_S$ and $\tilde{\chi}_L$, we decompose $\textbf{T}_{N,3}$ as follows
\begin{align*}
 \textbf{T}_{N,3}&=i\sum_{k,l \in\mathbb{Z}} \int_{\eta,\xi} \A_k(t,\eta) \overline{\widehat{\theta}}_k(\eta) \Bigg[e^{\lambda(t)(|k,\eta|^{\fr{1}{3}}-|l,\xi|^{\fr{1}{3}})}\Big[\dfrac{\M_k(\eta)}{\M_l(\xi)}-1\Big]\dfrac{\JB{k,\eta}^\sigma}{\JB{l,\xi}^\sigma}\Bigg]\\&
   \qquad \qquad \qquad \times \Big[\tilde{\chi}_S+\tilde{\chi}_L\Big]\Big(\reallywidehat{\nabla^{\perp}(\phi_{k-l}\chi)}(\eta-\xi)_{<N/8} \cdot(l,\xi) \A_l(t,\xi)\widehat{\theta}_l(\xi)_N\Big)
\;d\xi\;d\eta\\&
:=\textbf{T}_{N,3,S}+\textbf{T}_{N,3,L}.
\end{align*}
Similar to the estimate of $\textbf{T}_{N,2,S}$, using the commutator estimate in Lemma \ref{lem-com-M}, we obtain
\begin{align*}
|\textbf{T}_{N,3,S}|&
\lesssim \sum_{k,l \in\mathbb{Z}} \int_{\eta,\xi} \tilde{\chi}_S  |\A_k\widehat{\theta}_k(\eta)| \Bigg|e^{\lambda(t)(|k,\eta|^{\fr{1}{3}}-|l,\xi|^{\fr{1}{3}})} \dfrac{\JB{k-l,\xi-\eta}}{(|k|+|l|+|\eta|+|\xi|)^{\frac{2}{3}}} e^{C_0 |k-l,\eta-\xi|^{\frac{1}{3}}}\Bigg|\\&
   \qquad \qquad \qquad \times |\reallywidehat{\nabla^{\perp}(\phi_{k-l}\chi)}(\eta-\xi)_{<N/8}| |l,\xi| |\A\widehat{\theta}_l(\xi)_N|
\;d\xi\;d\eta\\&
\lesssim \sum_{k,l \in\mathbb{Z}} \int_{\eta,\xi} \tilde{\chi}_S  |\A\widehat{\theta}_k(\eta)| |l,\xi|^{\frac{1}{6}}|k,\eta|^{\frac{1}{6}} e^{{\bf c}\lambda(t)|k-l,\eta-\xi|^{1/3}} \\&
   \qquad \qquad \qquad \times |\reallywidehat{\nabla^{\perp}(\phi_{k-l}\chi)}(\eta-\xi)_{<N/8}|  |\A\widehat{\theta}_l(\xi)_N|
\;d\xi\;d\eta\\&
\lesssim \dfrac{\epsilon}{\JB{t}^{4}}  \norm{|\nabla|^{\frac{1}{6}}\A \theta_N }_{L^2} \norm{|\nabla|^{\frac{1}{6}}\A \theta_{\sim N}}_{L^2}.
\end{align*}

Now, we move on to  $\textbf{T}_{N,3,L}$. Let us start by applying a more refined decomposition based on the size of $|l|$ relative to $|\xi|$ and vice versa:
\[
\begin{aligned}
 \textbf{T}_{N,3,L}&=i\sum_{k,l \in\mathbb{Z}} \int_{\eta,\xi} \A_k(t,\eta) \overline{\widehat{\theta}}_k(\eta) \Bigg[e^{\lambda(t)(|k,\eta|^{\fr{1}{3}}-|l,\xi|^{\fr{1}{3}})}\Big[\dfrac{\M_k(\eta)}{\M_l(\xi)}-1\Big]\dfrac{\JB{k,\eta}^\sigma}{\JB{l,\xi}^\sigma}\Bigg]\\&
   \qquad \qquad \qquad \times \tilde{\chi}_L(\mathds{1}_{|l|\leq 4|\xi|}+\mathds{1}_{|l|> 4|\xi|})\Big(\reallywidehat{\nabla^{\perp}(\phi_{k-l}\chi)}(\eta-\xi)_{<N/8} \cdot(l,\xi) \A_l(t,\xi)\widehat{\theta}_l(\xi)_N\Big)
\;d\xi\;d\eta\\&
:=\textbf{T}^{\circled{1}}_{N,3,L}+\textbf{T}^{\circled{2}}_{N,3,L}.
\end{aligned}
\]
Firstly, $\textbf{T}^{\circled{1}}_{N,3,L}$ can be treated in a  similar manner as $\textbf{T}^{\circled{1}}_{N,2,L}$. Indeed, now we have $|l,\xi|\les |l,\xi|^{\frac{1}{6}}|k,\eta|^{\frac{1}{6}}|\xi|^{\frac{2}{3}}\lesssim |l,\xi|^{\frac{1}{6}}|k,\eta|^{\frac{1}{6}}t^2$. Combining this with \eqref{ratio-M} yields
\[
\begin{aligned}
    |\textbf{T}^{\circled{1}}_{N,3,L}|
&\lesssim  \sum_{k,l \in\mathbb{Z}} \int_{\eta,\xi} t^{2}\tilde{\chi}_L |k,\eta|^{\fr{1}{6}} |\A\widehat{\theta}_k(\eta)| |l,\xi|^{\fr{1}{6}}|\A\widehat{\theta}_l(\xi)_N|\\&
   \qquad \qquad \qquad \times e^{{\bf c}\lambda(t)|k-l,\eta-\xi|^{\fr{1}{3}}} 
|\reallywidehat{\nabla^{\perp}(\phi_{k-l}\chi)}(\eta-\xi)_{<N/8}| 
\;d\xi\;d\eta\\&
\lesssim \dfrac{\epsilon}{\JB{t}^{2}}\norm{|\nabla|^{\fr{1}{6}}\A\theta_{\sim N}}_{L^2} \norm{|\nabla|^{\fr{1}{6}}\A\theta_N}_{L^2}. 
\end{aligned}
\]
It now remains to estimate $\textbf{T}^{\circled{2}}_{N,3,L}$. Recalling the definition of $\M_k(t,\eta)$ in \eqref{Multiplier components}, Lemma \ref{lem-growth-Lambda} and \eqref{l dominates eta}, we find that
\[
\bigg|\dfrac{\M_k(\eta)}{\M_l(\xi)}-1\bigg|
\lesssim\fr{e^{\fr{21}{20}\fr{C_0}{2}|\eta|^{\fr13}}+e^{\fr{21}{20}\fr{C_0}{2}|\xi|^{\fr13}}}{e^{\fr{C_0}{2}|l|^{\fr13}}}+\Big|e^{\fr{C_0}{2}(|k|^{\fr{1}{3}}-|l|^{\fr{1}{3}})}-1\Big|\lesssim \fr{\langle k-l\rangle}{|l|^{\fr{2}{3}}}e^{\fr{C_0}{2}|k-l|^{\fr{1}{3}}}.
\]
Consequently,
\[
\begin{aligned}
    |\textbf{T}^{\circled{2}}_{N,3,L}|&\lesssim  \sum_{k,l \in\mathbb{Z}} \int_{\eta,\xi}\tilde{\chi}_L |k,\eta|^{\fr{1}{6}} |\A\widehat{\theta}_k(\eta)| |l,\xi|^{\fr{1}{6}}|\A\widehat{\theta}_l(\xi)_N|\\&
   \qquad \qquad \qquad \times e^{{\bf c}\lambda(t)|k-l,\eta-\xi|^{\fr{1}{3}}} 
|\reallywidehat{\nabla^{\perp}(\phi_{k-l}\chi)}(\eta-\xi)_{<N/8}| 
\;d\xi\;d\eta\\&
\lesssim \dfrac{\epsilon}{\JB{t}^{4}}\norm{|\nabla|^{\fr{1}{6}}\A\theta_{\sim N}}_{L^2} \norm{|\nabla|^{\fr{1}{6}}\A\theta_N}_{L^2}. 
\end{aligned}
\]
\subsubsection{{Treatment of }${\bf T}_{N,4}$}
Next, we finally provide an estimate for $\textbf{T}_{N,4}$. By using 
\[
\Big|\dfrac{\JB{k,\eta}^\sigma}{\JB{l,\xi}^\sigma}-1\Big|\lesssim \dfrac{|k-l,\eta-\xi|}{\JB{l,\xi}},
\] and the second estimate in Lemma~\ref{Convolution estimates} in Appendix~\ref{Aux} (along with bootstrap hypotheses), we have 
\[
\begin{aligned}
|\textbf{T}_{N,4}| &\leq \sum_{k,l \in\mathbb{Z}} \int_{\eta,\xi}  |\A\widehat{\theta}_k(\eta)| e^{\lambda(t)(|k,\eta|^{\fr{1}{3}}-|l,\xi|^{\fr{1}{3}})}\\&
   \qquad \qquad  \times \dfrac{|k-l,\eta-\xi|}{\JB{l,\xi}} |\reallywidehat{\nabla^{\perp}(\phi_{k-l}\chi)}_{<N/8}(\eta-\xi)| |l,\xi| |\A\widehat{\theta}_l(\xi)_N|
\;d\xi\;d\eta\\&
\lesssim \frac{\epsilon}{\JB{t}^4}\norm{\A\theta_{\sim N}}_{L^2} \norm{\A\theta_N}_{L^2}.
\end{aligned}
\]
The above inequality therefore concludes the series of estimates of the transport term $\mathbf{T}_N$. 
\subsection{Remainder Term}\label{Remainder}
We are now ready to derive the estimate for the remainder term $\mathcal{R}$. Recall that 
\[
\begin{aligned}
    \mathcal{R}&=2\pi \sum_{N\in\mathbb{D}} \sum_{\frac{N}{8}\leq N'\leq 8N}\int_{\mathbb{T}\times[0,1]} \A \theta \big[\A(\nabla^{\perp}(\phi\chi)_{N} \cdot \nabla \theta_{N'})-\nabla^{\perp}(\phi\chi)_{N} \cdot \nabla \A \theta_{N'} \big]\;dz\;dv\\&
    =\mathcal{R}^1+\mathcal{R}^2.
\end{aligned}
\]
We present only the estimate for $\mathcal{R}^1$. The term $\mathcal{R}^2$ can be treated similarly. 
Observe that on the Fourier side $\mathcal{R}^1$ reads 
\[
\mathcal{R}^1= 2\pi \sum_{N\in\mathbb{D}} \sum_{\frac{N}{8}\leq N'\leq 8N}\sum_{k,l \in\mathbb{Z}}\int_{\eta,\xi} \A \overline{\widehat{\theta}}_k(\eta) \A_k(\eta)\reallywidehat{\nabla^{\perp}(\phi_{l}\chi)}(\xi)_{N} \cdot (\widehat{\nabla \theta})_{k-l}(\eta-\xi)_{N'}\;d\xi\; d\eta.
\]
On the support of the integrand of  $\mathcal{R}^1$, there hold
\[
\fr{N}{2}\le|l,\xi|\le\fr{3}{2} N,\quad{\rm and}\quad \fr{N'}{2}\le|k-l,\eta-\xi|\le\fr{3}{2} N'\quad{\rm with}\quad \fr{N}{8}\le N'\le8N.
\]
This implies that
\[
\fr{1}{24}|k-l,\eta-\xi|\le|l,\xi|\le24|k-l,\eta-\xi|.
\]
Then via \eqref{triangle2} and Remark~\ref{rem-<1}, we find that
\[
|k,\eta|^{\fr{1}{3}}\le 0.98(|l,\xi|^{\fr{1}{3}}+|k-l,\eta-\xi|^{\fr{1}{3}}).
\]
On the other hand, we infer from Lemmas \ref{lem-total growth} and \ref{lem-growth-Lambda} that
\[
\J_k(t,\eta)\M_k(t,\eta)\lesssim e^{\fr{21}{20}(\mu+\fr{C_0}{2})|k,\eta|^{\fr{1}{3}}}\lesssim e^{\fr{21}{20}(\mu+\fr{C_0}{2})(|l,\xi|^{\fr{1}{3}}+|k-l,\eta-\xi|^{\fr{1}{3}})}.
\]
Therefore, on the support of the integrand of  $\mathcal{R}^1$, we have
\begin{equation}\label{up-A-remainder}
\begin{aligned}
\mathcal{A}_k(t,\eta)\les&e^{\big[0.98\lambda(t)+1.05(\mu+\fr{C_0}{2})\big]|l,\xi|^{\fr13}}e^{\big[0.98\lambda(t)+1.05(\mu+\fr{C_0}{2})\big]|k-l,\eta-\xi|^{\fr13}}\langle l,\xi\rangle\langle k-l,\eta-\xi\rangle^{\sigma-1}\\
\les&e^{\lambda(t)|l,\xi|^{\fr13}}e^{\lambda(t)|k-l,\eta-\xi|^{\fr13}}\langle l,\xi\rangle \langle k-l,\eta-\xi\rangle^{\sigma-1}.
\end{aligned}
\end{equation}
As a result, combining this   with Corollary \ref{gevrey norm of phi chi} and the second inequality in Lemma \ref{Convolution estimates} gives us
\begin{align*}
    |\mathcal{R}^1|
    &\lesssim  \sum_{N\in\mathbb{D}} \sum_{\frac{N}{8}\leq N'\leq 8N}\sum_{k,l}\int_{\eta,\xi} |\A\widehat{\theta}_k(\eta)| e^{\lambda(t)|l,\xi|^{\fr13}}\JB{l,\xi}|\reallywidehat{\nabla^{\perp}(\phi_{l}\chi)}(\xi)_{N}|  \\&
    \qquad \qquad \qquad \qquad \qquad \times  e^{\lambda(t)|k-l,\eta-\xi|^{\fr13}}\JB{k-l,\eta-\xi}^{\sigma-1}|\widehat{\nabla \theta}(k-l,\eta-\xi)_{N'}|\;d\xi\; d\eta\\
    &\lesssim  \sum_{N\in\mathbb{D}} \sum_{\frac{N}{8}\leq N'\leq 8N}\|\A\theta\|_{L^2}\|(\phi\chi)_N\|_{\mathcal{G}^{\lambda,\sigma-4;\frac{1}{3}}} \|\theta_{N'}\|_{\mathcal{G}^{\lambda,\sigma;\frac{1}{3}}}
    \lesssim \dfrac{\epsilon^3}{\JB{t}^4}.
\end{align*}

\section{Estimate of Linear Term}\label{linear}
This section is devoted to estimating the term $\boldsymbol{\Pi}_\theta$ and proving Proposition \ref{Prop:linear}. Recall that 
\[
\boldsymbol{\Pi}_\theta=\int \A\theta \A(\partial_z (\phi\chi)\underline{\varrho}'(y))\;dzdy.
\]
On the  Fourier side,   $\boldsymbol{\Pi}_\theta$ can be rewritten as
\begin{equation}\label{Pi theta}
    \begin{aligned}
\boldsymbol{\Pi}_\theta&=\frac{i}{2\pi}\sum_{k\neq 0}\int_{\eta,\xi} \mathcal{A}\overline{\widehat{\theta }}_k(\eta)\mathcal{A}_k(\eta)k\widehat{(\phi_k\chi)}(\xi)\widehat{\underline{\varrho}'}(\eta-\xi)\; d\xi d\eta\\&
=-\frac{1}{2\pi}\sum_{k\neq 0}\int_{\eta, \xi} \mathcal{A}\overline{\widehat{\theta }}_k(\eta)\frac{\mathcal{A}_k(\eta)}{\mathcal{A}_k(\xi)}\dfrac{k^2}{(k^2+(\xi-kt)^2)^2}\mathcal{A}_k(\xi)\reallywidehat{\partial_z^{-1}\Delta_L^2(\phi_k\chi)}(\xi)\widehat{\underline{\varrho}'}(\eta-\xi)\; d\xi d\eta.
\end{aligned}
\end{equation}
Recalling that in Theorem~\ref{main theorem}, we assume that $\norm{\underline{\varrho}'}^2_{\mathcal{G}^{\lambda_b,\fr13}} \leq \delta^2$. By definition, this means that
\begin{equation}\label{size bkgrd dervtv density }
\norm{\underline{\varrho'}}_{\mathcal{G}^{\lambda_b,\fr13}}^2=\int_{\eta}|\widehat{\underline{\varrho}'}(\eta) |^2 e^{2\lambda_b|\eta|^{\fr13}}\;d\eta\leq \delta^2.
\end{equation}

In order to estimate $\boldsymbol{\Pi}_\theta$, the strategy  is  to control the following term
\begin{equation}\label{frac term for linear part}
  \frac{\mathcal{A}_k(\eta)}{\mathcal{A}_k(\xi)}\frac{1}{k^2(1+(\frac{\xi}{k}-t)^2)^2}.
\end{equation}
The idea mimics the proof of Proposition~\ref{bounds on L^2 norms}. For that, we split our analyses into two main cases:
\\
\textbf{Case 1:} $|\eta-\xi|>\frac{1}{6}|\xi|$.
Similar to \eqref{ratio-A^Lam}, we have 
     \[
     \begin{aligned}
     \frac{\mathcal{A}_k(\eta)}{\mathcal{A}_k(\xi)}&
     \lesssim e^{\lambda(t)|\eta-\xi|^{\frac{1}{3}}} \frac{\JB{k,\eta}^{\sigma}}{\JB{k,\xi}^{\sigma}} \bigg(e^{\mu|\eta-\xi|^{\frac{1}{3}}} \frac{\Theta_k(\xi)}{\Theta_k(\eta)}+1\bigg) \bigg(e^{\fr{C_0}{2}|\eta-\xi|^{\frac{1}{3}}} \frac{\Lambda_k(\xi)}{\Lambda_k(\eta)}+1\bigg)
     \\& \lesssim  e^{(\lambda(t)+3\mu+2C_0)|\eta-\zeta|^{\frac{1}{3}}}\JB{\eta-\zeta}^{\sigma}.
     \end{aligned}
     \] 
     In addition to the above estimate, using the fact that $\JB{t-\frac{\xi}{k}} \JB{\frac{\xi}{k}}\gtrsim \JB{t}$, we can directly infer that     
     \[
     \dfrac{1}{k^2\Big(1+(\frac{\xi}{k}-t)^2\Big)^2}\lesssim \frac{1}{\JB{t-\frac{\xi}{k}}^4}\lesssim \frac{\JB{\frac{\xi}{k}}^4}{\JB{t}^4}\lesssim \fr{\JB{\eta-\xi}^4}{\JB{t}^4}.
     \] 
Combining the above two inequalities with \eqref{up-Atheta} and the second inequality of Lemma~\ref{Convolution estimates}, we arrive at
 \begin{equation}
 \begin{aligned}
    &\sum_{k\ne0}\int_{|\eta-\xi|>\fr16|\xi|}|\mathcal{A}\widehat{\theta }_k(\eta)|\Big|\mathcal{A}_k(\eta)k\widehat{(\phi_k\chi)}(\xi)\Big| |\widehat{\underline{\varrho}'}(\eta-\xi)|\; d\xi d\eta\\& \qquad \qquad
    \lesssim \fr{1}{\JB{t}^4}\|\A\theta\|_{L^2}\|\partial_z^{-1}\Delta_L^2\mathcal{A}(\phi\chi)_{\ne}\|_{L^2}\left\|\widehat{\underline{\varrho'}}e^{(\lambda(t)+3\mu+2C_0)|\eta|^{\fr13}}\langle \eta\rangle^{\sigma+5}\right\|_{L^2_\eta}
    \lesssim
    \fr{\epsilon^2}{\JB{t}^4}\|\underline{\varrho'}\|_{\mathcal{G}^{\lambda_b,\fr13}},
 \end{aligned} 
 \end{equation}
 provided 
 \begin{equation}\label{restriction-lambda_b}
     \lambda(0)+3\mu+2C_0+1<\lambda_b.
 \end{equation}

\textbf{Case 2:} $|\eta-\xi|\leq\frac{1}{6}|\xi|$. In this case, it is clear that
 $|\eta| \approx |\xi|$. Similar to \eqref{ratio-ALam'}, and using inequality \eqref{triangle1}, now we  have a refined estimate
$\frac{\mathcal{A}_k(\eta)}{\mathcal{A}_k(\xi)} \lesssim e^{(0.2\lambda(t)+3\mu+2C_0)|\eta-\xi|^{\frac{1}{3}}}.$

Furthermore, in order to estimate $\frac{1}{k^2(1+(\frac{\xi}{k}-t)^2)^2)}$, let us investigate the following two cases. 

\textbf{Subcase 2.1: $(k,\xi)\in\mathfrak{U}$ or $(k,\eta)\in\mathfrak{U}$.} Now Lemma~\ref{l xi nonresonant} leads us to the following inequality
\[
     \frac{1}{k^2(1+(\frac{\xi}{k}-t)^2)^2)}\lesssim \frac{1}{\JB{t}^2},\quad \text{or} \quad \frac{1}{k^2(1+(\frac{\xi}{k}-t)^2)^2)}\lesssim \frac{\langle \xi-\eta\rangle^4}{k^2(1+(\frac{\eta}{k}-t)^2)^2)}\lesssim \frac{\langle \xi-\eta\rangle^4}{\JB{t}^2}.
\]

\textbf{Subcase 2.2: $(k,\xi)\notin\mathfrak{U}$ and $(k,\eta)\notin\mathfrak{U}$.} In other words,
now we have $k \xi>0, \ \frac{|k t|}{2}\leq|\xi|\leq 2|kt|,\  |k,\xi|>1000$, and similarly for $(k,\eta)$. This, together with the fact $|\eta|\approx|\xi|$, implies that  $t\approx \frac{|\xi|}{|k|} \approx \frac{|\eta|}{|k|} $.  

If $|k|\geq \frac{1}{100}|\xi|^{\frac{1}{3}}$, then $t\lesssim |\xi|^{\frac{2}{3}}\lesssim |k|^2$. Thus we have 
     \begin{align*}
          \dfrac{1}{k^2(1+(\frac{\xi}{k}-t)^2)^2}\lesssim \frac{|\xi|^{1/3}}{t^{3/2}} \lesssim \frac{|k,\xi|^{1/6}|k,\eta|^{1/6}}{t^{3/2}}.
     \end{align*}

We then focus on $|k|<\frac{1}{100}|\xi|^{\frac{1}{3}}$. 
     Hence, there exist $m, n$ such that $1\leq |m|\leq E(|\xi|^{\frac{1}{3}})$ and $1\leq |n|\leq E(|\eta|^{\frac{1}{3}})$ with $t\in \tilde{\rm I}_{n,\xi} \cap \tilde{\rm I}_{m,\eta}$. 
      Since $t\approx \frac{|\xi|}{|k|}\approx \frac{|\eta|}{|m|}\approx \frac{|\xi|}{|n|}$, we have $|k|\approx |m|\approx |n|$.
      Let us first consider the scenario when $n\neq k$. Since we know that $t \in \tilde{\textup{I}}_{n,\xi}$, in this subcase  $|t-\frac{\xi}{k}|\gtrsim \frac{|\xi|}{k^2}\gtrsim 1$ must hold. As a consequence, we obtain
      \begin{align}
          \dfrac{1}{k^2(1+(\frac{\xi}{k}-t)^2)^2}\lesssim \frac{1}{\xi^2/k^2}\lesssim \frac{1}{\JB{t}^2},
     \end{align}
     where we have used the fact that $t\approx \frac{|\xi|}{|k|}$.
     
     Let us now consider the case when $m\neq k$. As before, since $t \in \tilde{\textup{I}}_{m,\eta}$, then $|t-\frac{\eta}{k}|\gtrsim \frac{|\eta|}{k^2}\gtrsim 1$ holds. Therefore,
    \begin{align}
          \dfrac{1}{k^2(1+(\frac{\xi}{k}-t)^2)^2}\lesssim \dfrac{\langle \eta-\xi\rangle^2}{k^2(1+(\frac{\eta}{k}-t)^2)}\lesssim  \frac{\langle \eta-\xi\rangle^2}{\eta^2/k^2}\lesssim \frac{\langle \eta-\xi\rangle^2}{\JB{t}^2}.
     \end{align}
     If $m=n=k$, then 
     \[
            \dfrac{\mathds{1}_{|k|<\fr{1}{100}|\xi|^{\fr13}}}{k^2(1+(\frac{\xi}{k}-t)^2)^2}\lesssim \dfrac{\mathds{1}_{|k|<\fr{1}{100}|\xi|^{\fr13}}}{1+|\frac{\xi}{k}-t|}\dfrac{\langle \xi-\eta\rangle}{1+|\frac{\eta}{k}-t|} \lesssim  \sqrt{\frac{\partial_t \Lambda(t,\eta)}{\Lambda(t,\eta)}}\mathds{1}_{|k|<|\eta|}  \sqrt{\frac{\partial_t \Lambda(t,\xi)}{\Lambda(t,\xi)}}\mathds{1}_{|k|<|\xi|} \JB{\eta-\xi}.
            \]
 
Hence, overall, we have
\[
 \frac{\mathcal{A}_k(\eta)}{\mathcal{A}_k(\xi)}\frac{1}{k^2(1+(\frac{\xi}{k}-t)^2)^2} \lesssim \bigg(\frac{1}{\JB{t}^2}+\frac{|k,\xi|^{\frac{1}{6}} |k,\eta|^{\frac{1}{6}}}{t^{\frac{3}{2}}}+ \sqrt{\dfrac{\partial_t\Lambda(t,\eta)}{\Lambda(t,\eta)}}\mathds{1}_{|k|<|\eta|}\sqrt{\dfrac{\partial_t\Lambda(t,\xi)}{\Lambda(t,\xi)}}\mathds{1}_{|k|<|\xi|}\bigg)e^{\lambda(t)|\eta-\xi|^{\frac{1}{3}}}.
\]
Going back to the expression of $\boldsymbol{\Pi}_\theta$ \eqref{Pi theta}, \eqref{size bkgrd dervtv density }, combining all estimates together, and using Proposition~\ref{bounds on L^2 norms}, we obtain
 \[
         |\boldsymbol{\Pi}_\theta| \lesssim  \dfrac{\delta \epsilon^2}{\JB{t}^2} +\delta \CK_\lambda+\delta \CK_\Lambda.
\] This therefore furnishes the desired bound for the linear term.  

\appendix
\section{Stream function Estimate for Stokes System} \label{Appendix A}
This particular section is devoted to deriving some kernel estimates of the stream function as stated in Lemma~\ref{rep of psi_k} and Lemma~\ref{Estimate fourier psi chi} in the original coordinate $(x,y)$. We begin by proving Lemma~\ref{rep of psi_k}.

\begin{proof}[\textbf{Proof of Lemma~\ref{rep of psi_k}}]

     To start with, ignoring the boundary conditions for a moment, one can show that a fundamental set of solutions of \eqref{eq: Stokes3} is given by $\{e^{ky},e^{-ky},ye^{ky},ye^{-ky}\}$. Hence, all solutions of the inhomogeneous differential equation \eqref{eq: Stokes3} can be written in terms of the linear combination
\[
\tilde{\psi}_k^{inh}(y)=a(y)e^{ky}+b(y)e^{-ky}+c(y)ye^{ky}+d(y)ye^{-ky}=:a(y)\phi_1+b(y)\phi_2+c(y)\phi_3+d(y)\phi_4,
\]
where the superscript $``inh"$ has been used to emphasize the solution of the inhomogeneous problem without taking into account the boundary conditions.

Via the standard variation of parameters approach, in order to solve for the coefficients $a(y),b(y),c(y)$ and $d(y)$, we have to assume the following holds, namely
\[
\begin{pmatrix}
     \phi_1& \phi_2 & \phi_3 &\phi_4\\
     \phi'_1& \phi'_2 & \phi'_3 &\phi'_4\\
     \phi''_1& \phi''_2 & \phi''_3 &\phi''_4\\
     \phi'''_1& \phi'''_2 & \phi'''_3 &\phi'''_4
\end{pmatrix}\begin{pmatrix}
         a'(y)\\b'(y)\\c'(y)\\d'(y)
     \end{pmatrix}=\begin{pmatrix}
         0\\0\\0\\ik\tilde{\rho}_k(y)
     \end{pmatrix}.
\]
Hence, undergoing explicit and tedious calculations, we obtain
\begin{equation}\label{a',b',c',d'}
a'(y)=\dfrac{-(1+ky)}{4k^2}e^{-ky}i \tilde{\rho}_k,\;
b'(y)=\dfrac{1-ky}{4k^2}e^{ky} i \tilde{\rho}_k,\;
c'(y)=\dfrac{e^{-ky}i \tilde{\rho}_k}{4k},\;
d'(y)=\dfrac{e^{ky}i\tilde{\rho}_k}{4k}.
\end{equation}

It is now the right time to introduce the boundary conditions back as stated on the second line of \eqref{eq: Stokes3}. With this in mind, we denote the solution to the boundary value problem \eqref{eq: Stokes3} as follows 
\begin{equation}\label{Decomposition of psi}
 \tilde{\psi}_k(y)=\tilde{\psi}_k^{inh}(y)+\tilde{\psi}_k^{bd}(y),   
\end{equation}
where 
\begin{equation}\label{psi bd 1}
\tilde{\psi}_k^{bd}(y)=c_1e^{ky}+c_2e^{-ky}+c_3ye^{ky}+c_4ye^{-ky},
\end{equation} where all the constants $c_i$ for $i=1,2,3,4$ are to be determined. 

Using the expression for $\tilde{\psi}_k(y)$ and the boundary condition  at $y=0$, one obtains $\tilde{\psi}_k(0)=c_1+c_2=0$. Moreover, under the fact that $c_1=-c_2$, the remaining boundary conditions can be written in a matrix form 
\begin{equation}\label{product B and c}
\begin{pmatrix}
    2 \sinh{k}& e^{k} & e^{-k}\\
    2k &1 &1\\
    2k \cosh{k} & e^k(k+1) & e^{-k}(1-k)
     
\end{pmatrix}
\begin{pmatrix}
    c_1 \\c_3 \\c_4
\end{pmatrix}=
\begin{pmatrix}
-\tilde{\psi}_k^{inh}(1) \\0 \\ -(\tilde{\psi}_k^{inh})'(1)
\end{pmatrix}.
\end{equation}
Observe for a moment that using the information in \eqref{a',b',c',d'}, we have

\begin{equation}\label{psi inh}
\begin{aligned}
    \tilde{\psi}_k^{inh}(y)&=\int_0^y\bigg[\dfrac{2\sinh{(k\mathrm{\textbf{y}}-ky)}-2k\By \cosh{(k\By-ky)}+2ky\cosh{(k\By-ky)}}{4k^2}\bigg]i\tilde{\rho}_k(\By)\;d\By
    \\&
    =:\int_0^y \dfrac{\mathcal{R}_{inh}(k,y,\By)}{4k^2}i\tilde{\rho}_k(\By)\;d\By.
\end{aligned}
\end{equation}
Evaluating the above expression at $y=1$ yields
\[
\begin{aligned}
&\tilde{\psi}_k^{inh}(1)=\int_0^1\bigg[\dfrac{2\sinh{(k\By-k)}-2k\By \cosh{(k\By-k)}+2k\cosh{(k\By-k)}}{4k^2}\bigg]i\tilde{\rho}_k(\By)\;d\By,\\
&(\tilde{\psi}_k^{inh})'(1)=\int_0^1\bigg[\dfrac{2k^2\By \sinh{(k\By-k)}-2k^2\sinh{(k\By-k)}}{4k^2}\bigg]i\tilde{\rho}_k(\By)\;d\By.
\end{aligned}
\]
Hence, using the expressions of $\tilde{\psi}_k^{inh}(1)$ and $(\tilde{\psi}_k^{inh})'(1)$ above along with  \eqref{product B and c} allow us to infer


\[
\begin{pmatrix}
    c_1 \\c_3\\c_4
\end{pmatrix}=
\begin{pmatrix}
    \tilde{\psi}_k^{inh}(1)\dfrac{2k \cosh{(k)}+2\sinh{k}}{4k^2-2\cosh{2k}+2}-(\tilde{\psi}_k^{inh})'(1)\dfrac{2 \sinh{(k)}}{4k^2-2\cosh{(2k)}+2}\\
   \tilde{\psi}_k^{inh}(1)\dfrac{-2e^{-k}k^2-2k\sinh{(k)}}{4k^2-2\cosh{(2k)}+2}-(\tilde{\psi}_k^{inh})'(1)\dfrac{2e^{-k}k-2\sinh{(k)}}{4k^2-2\cosh{(2k)}+2}\\
    \tilde{\psi}_k^{inh}(1)\dfrac{-2k^2e^{k}-2k\sinh{(k)}}{4k^2-2\cosh{(2k)}+2}-(\tilde{\psi}_k^{inh})'(1)\dfrac{2\sinh{(k)}-2e^kk}{4k^2-2\cosh{(2k)}+2}
\end{pmatrix}.
\]
With this in mind, we therefore have 
\begin{equation}\label{psi bd}
\begin{aligned}
&\tilde{\psi}_k^{bd}(y)
=\int_{0}^1\dfrac{\mathcal{R}_{bd}(k,y,\By)}{4k^2(4k^2-2\cosh{(2k)}+2)} i\tilde{\rho}_k(\By)\;d\By, 
\end{aligned}
\end{equation}
where 
\[
\begin{aligned}
\mathcal{R}_{bd}(k,y,\By)&=2\sinh{(ky)}I_1
-8k^2y(\By-1)\Big[k \sinh{(ky-k)}-\sinh{(k)}\sinh{(ky)}\Big]\sinh{(k\By-k)}\\ &\quad -8ky\Big[k \cosh{(ky-k)}+\sinh{(k)}\cosh{(ky)}\Big]\Big[(1-\By)k \cosh{(k\By-k)}+\sinh{(k\By-k)}\Big],
\end{aligned}
\]
with $I_1=4\Big(k^2\cosh{(k\By)(1-\By)}-k\By \sinh{(k)}\cosh{(k\By-k)}+k \sinh{(k\By)}+\sinh{(k)}\sinh{(k\By-k)}\Big)$.

Recalling \eqref{psi inh} and \eqref{psi bd}, for notational convenience,  let us introduce the following two terms
\[
    K_{inh}(k,y,\By)=\dfrac{\mathcal{R}_{inh}(k,y,\By)(4k^2-2\cosh{2k}+2)}{(4k^2)(4k^2-2\cosh{2k}+2)},\qquad
    K_{bd}(k,y,\By)=\dfrac{\mathcal{R}_{bd}(k,y,\By)}{(4k^2)(4k^2-2\cosh{(2k)}+2)}.
\]
In light of that, $\tilde{\psi}_k$ can be written in the following way
\begin{equation}\label{decomposition of psi hat}
    \tilde{\psi}_k(y)=\int_{0}^y i K_{inh}(k,y,\By)\;\tilde{\rho}_k(\By)\;d\By+\int_{0}^1 i K_{bd}(k,y,\By)\;\tilde{\rho}_k(\By)\;d\By.
\end{equation}

In order to arrive at the expression of $\tilde{\psi}_k$ in Lemma~\ref{rep of psi_k}, we only need to re-organize the kernel by going through tedious and yet elementary calculations. Indeed we can decompose the kernel into two parts: The `good' kernel $K_{bd}^g$ which decays exponentially in $k$ as long as $y, \By$ are away from the boundary and the `bad' kernel $K$ which does not decay at $y=\By$ but can be rewritten as a convolution. 
This therefore completes the proof of Lemma~\ref{rep of psi_k}. \qedhere
\end{proof}

Having established the proof of Lemma~\ref{rep of psi_k}, we proceed to present the proof of Lemma~\ref{Estimate fourier psi chi}. To start with, we recall the cutoff function $\chi$ \eqref{cutoff function}. Here, the main goal is to obtain the following type of estimate
\[
\begin{aligned}
    &\mathcal{F}({\psi}_k \chi)(\eta)=\int_{\mathbb{R}}\mathcal{G}(k,\eta,\zeta)\widehat{\rho}_k(\zeta)\;d\zeta,
\end{aligned} 
\]
where again $\mathcal{F}$ represents the Fourier transform in $(x,y)$. In preparation for the presentation of the proof below, we define a cutoff function $\chi_1$ satisfying the conditions
\[
\begin{aligned}
&\chi_1(y)=1,\quad \text{for}\quad 0\leq y\leq 1-\frac{\kappa}{10},\\
&\textup{supp}\chi_1 \subset (-1,1),\qquad  \sup_{y\in\mathbb{R}}\bigg|\frac{\partial^m\chi_1}{\partial_y^m}\bigg|\leq M^m(m!)^{\frac{2}{s_0+1}}(m+1)^{-2}.
\end{aligned}
\]
Due to the compact support of $\tilde{\rho}_k(\By)$ and $\chi^{(n)}(y)$ we have $\chi^{(n)}(y)\chi_1(|y-\By|)\tilde{\rho}_k(\By)\equiv \chi^{(n)}(y)\tilde{\rho}_k(\By)$ for $n=1,...,4$. 
\begin{proof}[\textbf{Proof of Lemma~\ref{Estimate fourier psi chi}}]
As the first step towards such estimation, 
let us apply the operator $\Delta^2_k$ to $\tilde{\psi}_k \chi$ and use Lemma~\ref{rep of psi_k}. We have the following equality
\begin{equation}\label{laplacian square with cutoff}
\begin{aligned}
    &\Delta^2_k (\tilde{\psi}_k \chi)
    =ik\tilde{\rho}_k(y)\chi(y)+\Big(4\tilde{\psi}_k^{'''}(y)-4k^2\tilde{\psi}_k^{'}(y)\Big)\chi'(y)\\
    &\quad\qquad\qquad+\Big(6\tilde{\psi}_k^{''}(y)-2k^2\tilde{\psi}_k(y)\Big)\chi''(y)+4\tilde{\psi}_k^{'}(y)\chi'''(y)+\tilde{\psi}_k(y) \chi''''(y)\\&
    =ik\tilde{\rho}_k(y)\chi(y)+\sum_{n=1}^4i\chi^{(n)}(y)\int_{\mathbb{R}}  \Mu_n(k,\By-y)\tilde{\rho}_k(\By)\;d\By
    +\sum_{n=1}^4 i\int_{\mathbb{R}} N_n(y,\By)\tilde{\rho}_k(\By)\;d\By,
\end{aligned}
\end{equation}
where
\begin{equation}\label{K1 and K2}
\begin{aligned}
&M_1(k,w)=\Big(4\partial^3_yK(k,|w|)-4k^2\partial_yK(k,|w|)\Big)\chi_1(|w|),\\
&M_2(k,w)=\Big(6\partial^2_yK(k,|w|)-2k^2K(k,|w|)\Big)\chi_1(|w|),\\
&M_3(k,w)=4\partial_yK(k,|w|)\chi_1(|w|),\quad 
M_4(k,w)=K(k,|w|)\chi_1(|w|), \quad \textup{for }w=y-\By,\\
&N_1(k,y,\By)=\Big(4(\partial^3_yK^{g}_{bd})(k,y,\By) -4k^2 (\partial_yK^{g}_{bd})(k,y,\By)\Big)\chi'(y)\chi(\By), \\
&N_2(k,y,\By)=\Big(6(\partial^2_yK^{g}_{bd})(k,y,\By)-2k^2 K^{g}_{bd}(k,y,\By)\Big)\chi''(y)\chi(\By), \\&  N_3(k,y,\By)=4K^{g}_{bd}(k,y,\By)\chi'''(y)\chi(\By),\quad
N_4(k,y,\By)=K^{g}_{bd}(k,y,\By)\chi''''(y)\chi(\By).
\end{aligned}
\end{equation}
Now, by taking the Fourier transform in the $y$ variable, we have 
\begin{align*}
    (\eta^2+k^2)^2\widehat{\psi_k\chi}(\eta)
    &=ik\hat{\chi}\ast \hat{\rho}+i\sum_{n=1}^4\widehat{\chi^{(n)}}\ast (\widehat{M_n}\hat{\rho})+i\sum_{n=1}^4\int_{\mathbb{R}} \mathcal{N}_n(\eta, \zeta)\hat{\rho}_k(\zeta)d\zeta\\
    &=:\int_{\mathbb{R}} \mathcal{G}_1(k,\eta,\zeta)\hat{\rho}_k(\zeta)d\zeta
\end{align*}
where 
$\mathcal{N}_n(\eta, \zeta)= \frac{\widehat{N_n}(\eta,-\zeta)}{4\pi^2}.$
Thus we obtain the Fourier kernel $\mathcal{G}(k,\eta, \zeta)=\frac{\mathcal{G}_1(k,\eta,\zeta)}{(k^2+\eta^2)^2}$, so that 
\begin{align*}
    \widehat{\psi_k\chi}(\eta)=\int_{\mathbb{R}} \mathcal{G}(k,\eta,\zeta)\hat{\rho}_k(\zeta)d\zeta.
\end{align*}

Now let us estimate the kernel $\mathcal{G}$. By Lemma~\ref{Fourier bound given by exp} and the regularity in \eqref{cutoff regularity}, there exists $\lambda_1$ such that
\begin{align*}
    |ik\hat{\chi}(\eta-\zeta)|\lesssim |k|e^{-\lambda_1|\eta-\zeta|^s}.
\end{align*}
A direct calculation gives $D_k\gtrsim k^2\cosh 2k$, which gives that $|M_n|\lesssim |k|$. Thus we have $|\widehat{M_n}(\eta)|\lesssim |k|$ and again by applying Lemma~\ref{Fourier bound given by exp}, there exists $\lambda_2$ 
\begin{align*}
    |i\widehat{\chi^{(n)}}(\eta-\zeta)\widehat{M_n}(\eta)|\lesssim |k|e^{-c\lambda_2|\eta-\zeta|^s}. 
\end{align*}
It therefore remains to estimate $\widehat{N_n}(\eta,-\zeta)$. But we know that via the expressions of $N_n$ \eqref{K1 and K2}, $K^g_{bd}$ \eqref{rep of psi_k}, Lemma~\ref{estimates for K good 1}, Remark~\ref{stirling}, and Lemma~\ref{Fourier bound given by exp} that there exists $\lambda_3$ such that $|\widehat{N_n}(\eta,-\zeta)|\lesssim e^{-\kappa|k|/8}e^{-\lambda_3|\eta,\zeta|^s}\lesssim e^{-\kappa|k|/8}e^{-\lambda_3|\eta-\zeta|^s}$.
Thus, combining all estimates gives us that 
\[
 \big|\mathcal{G}_1(k, \eta,\zeta)\big|\lesssim |k|e^{-2\lambda_M|\eta-\zeta|^s}. 
\]
holds for $\lambda_M:=\frac{1}{2}\min\{\lambda_1,\lambda_2,\lambda_3\}$. 
which leads us to the kernel estimate in Lemma~\ref{Estimate fourier psi chi}. 
\qedhere
\end{proof}
\section{Auxiliary Estimates}\label{Aux}
Here, we record some useful estimates that we use in our analysis. The cutoff function $\chi$ used in the series of lemmas below is the one introduced in \eqref{cutoff function} with regularity \eqref{cutoff regularity}.

\begin{lemma}[\cite{Yamanaka}]\label{Fourier bound given by exp}
    Let $d=1,2$ and $0<s<1,$ $K>1$ and $g\in C^\infty(\mathbb{R}^d)$ with  $\textup{supp } g \subset [a,b]^d $ and satisfies the bound 
    \[
    |D^\alpha g(x)|\leq K^m(m+1)^{m/s},
    \] for $x\in \mathbb{R}^d$, all integers $m\geq0$ and multi-indices $\alpha$ with $|\alpha|$=m. Then it follows that
    \[
    |\widehat{g}(\xi)|\lesssim_{K,s} Le^{-\lambda_K|\xi|^s}.
    \]
\end{lemma}
\begin{remark}\label{stirling}
    We would like remark that via the Stirling's approximation $N! \sim \sqrt{2\pi N} \big(\frac{N}{e}\big)^N$, there exist constants $K_1,K_2$ such that
    \[
    K_1^m(m+1)^{m/s}\lesssim \Gamma_s(m) \lesssim K_2^m (m+1)^{m/s},
    \]
    where $\Gamma_s(m)=\frac{2^{-5}(m!)^{\frac{1}{s}}}{(m+1)^{2}}$.
\end{remark}

\begin{lemma}\label{estimates for K good 1}
Let $m_1,m_2\in \mathbb{Z}_{\geq 0}$. Consider the functions 
\[
f(y)\in\{\chi^{(n)}(y),y \chi^{(n)}(y)\},\quad h(\By)\in\{\chi(\By),\By \chi(\By)\} \text{ for } n=0,1,2,3,
\]
and  $g(z)\in\{ \sinh{(kz)},\cosh{(kz)}, \sinh{(2k-kz)},\cosh{(2k-kz)} \}$. Then  the following estimate holds
\[
\Big|\dfrac{\partial_{\By}^{m_1}\partial_y^{m_2}\big(f(y)g(y\pm\By)h(\By)\big)}{\cosh{2k}}\Big|\lesssim  e^{-\frac{\kappa |k|}{4}} \Gamma_s(m_1+m_2) \mathfrak{M}^{m_1+m_2}, 
\]
for some constant $\mathfrak{M}$.
\end{lemma}
\begin{proof}
In order to prove the lemma, we are going to use the following inequalities, for $N \in \mathbb{Z}_{\geq 0}$, $k>0.$ 
\[ \frac{k^N}{e^k}\leq \Big( \frac{N}{e} \Big)^N \lesssim \Gamma_s(N),\qquad 
\sum_{j=0}^{N}\frac{N!}{j!(N-j)!)} \Gamma_s(j)\Gamma_s(N-j)<\Gamma_s(N),
\]
where $\Gamma_s$ is as in Remark~\ref{stirling}. In practice, we will take $N=m_1+m_2$. Before delving deeper into the proof, we would like to state the inequalities below as a consequence of the regularity \eqref{cutoff regularity}:
\[
\sup_{y\in [0,1]}\Bigg|\frac{d^jf(y)}{dy}\Bigg|\lesssim \frac{\mathfrak{M}_1^j(j!)^{\frac{1}{s}}}{(j+1)^2}, \qquad \sup_{\By\in [0,1]}\Bigg|\frac{d^jh(\By)}{d\By}\Bigg|\lesssim \frac{\mathfrak{M}_2^j(j!)^{\frac{1}{s}}}{(j+1)^2}
\]

 The proof is done when the argument of $g$ takes the form $y+\By$. The proof works the same when the argument is $y-\By$.
\begin{align*}
&\Bigg|\dfrac{\partial_{\By}^{m_1}\partial_y^{m_2}\big(f(y)g(y+\By)h(\By)\big)}{\cosh{2k}}\Bigg|=\Bigg|\partial_{\By}^{m_1} \Bigg(\sum_{j=0}^{m_2}\dfrac{m_2!}{j!(m_2-j)!}\frac{d^jf(y)}{dy} k^{m_2-j}\dfrac{\partial^{m_2-j}_y g((y+\By))h( \By)}{\cosh{2k}}\Bigg)\Bigg|
\\& \; \lesssim \sum_{r=0}^{m_1}\sum_{j=0}^{m_2} \Bigg|\dfrac{m_1!}{r!(m_1-r)!}\dfrac{m_2!}{j!(m_2-j)!} \frac{\mathfrak{M}_1^{j}(j!)^{1/s}}{(j+1)^2} k^{m_2+m_1-r-j}e^{-\frac{\kappa k}{2}}\frac{\mathfrak{M}_2^{r}(r!)^{1/s}}{(r+1)^2}\Bigg|\\&\;
\lesssim \sum_{r=0}^{m_1} \sum_{j=0}^{m_2} \Bigg|\dfrac{m_1!}{r!(m_1-r)!} \dfrac{m_2!}{j!(m_2-j)!} \mathfrak{M}_1^{j}\Gamma_{s}(j)\mathfrak{M}_2^{r}\Gamma_s(r) e^{-\frac{\kappa k}{4}}(\frac{m_2+m_1-r-j}{e}(\frac{4}{\kappa}))^{m_2+m_1-r-j}\Bigg|\\&\;
\lesssim \mathfrak{M}^{m_1+m_2} e^{-\frac{\kappa k}{4}}\sum_{r=0}^{m_1} \sum_{j=0}^{m_2} \Bigg|\dfrac{m_1!}{r!(m_1-r)!} \dfrac{m_2!}{j!(m_2-j)!} \Gamma_{s}(j)\Gamma_s(r) \Gamma_{s}(m_2+m_1-r-j)\Bigg|\\&
\;
\lesssim  \mathfrak{M}^{m_1+m_2} e^{-\frac{\kappa k}{4}} \Gamma_s(m_1+m_2),
\end{align*}
where $\mathfrak{M}=\max\{\mathfrak{M}_1,\mathfrak{M}_2,\frac{4}{\kappa}\}$.
\end{proof}

\begin{lemma}[\cite{BedrossianMasmoudi2015}]\label{exponent ineq}
    Suppose that $0<s<1$ and $x,y \geq 0$.\\
        (1) If $x+y>0$, then 
           $ |x^s-y^s|\lesssim_s \dfrac{1}{x^{1-s}-y^{1-s}}|x-y|.$\\
        (2) If $|x-y|\leq \frac{x}{C}$ for some $C>1$, then 
        \begin{equation}\label{triangle1}
            |x^s-y^s|\leq \frac{s}{(C-1)^{1-s}}|x-y|^s.
        \end{equation}
        (3) More generally, it holds that
           $ |x+y|^s \leq \Big(\dfrac{x}{x+y}\Big)^{1-s} (x^s+y^s).$
            In particular, if $y\le x\le {\rm K}y$ for some $\rm K>0$, then
        \begin{align}\label{triangle2}
|x+y|^s\le\Big(\fr{\rm K}{1+\rm K}\Big)^{1-s}(x^s+y^s).
        \end{align}
\end{lemma}
\begin{remark}\label{rem-<1}
    Taking $s=\fr{1}{3}$, $C=\fr{16}{3}$ in \eqref{triangle1} and ${\rm  K=24}$ in \eqref{triangle2}, it is easy to see that
    \[
    \max\left\{\fr{\fr{1}{3}}{(\fr{13}{3})^{\fr{2}{3}}}, (\fr{24}{25})^{\fr{2}{3}} \right\}=(\fr{24}{25})^{\fr{2}{3}}\approx0.97315\cdots.
    \]
\end{remark}

\begin{lemma}[\cite{BedrossianMasmoudi2015}]\label{Convolution estimates}
    Let $f(\xi),g(\xi) \in L^2_\xi(\mathbb{R}^d)$, and $\langle \xi \rangle^\sigma h(\xi)\in L^2_\xi(\mathbb{R}^d)$. For any $\sigma > d/2$ we have,
\begin{align*}
    &\norm{f \ast h}_{L^2_{\xi}} \lesssim \norm{f}_{L^2_{\xi}} \norm{\langle \cdot \rangle^\sigma h }_{L^2_{\xi}},\\
    &\int |f(\xi)(g\ast h)(\xi)| \;d\xi \lesssim \norm{f}_{L^2_{\xi}} \norm{g}_{L^2_{\xi}} \norm{\langle \cdot \rangle^\sigma h}_{L^2_{\xi}}.
\end{align*}
\end{lemma}

\section{Properties of $\Theta$ and $\mathcal{J}$}\label{estimates on Theta and J}

For the sake of self-containment of the paper, we state the following lemma which can be found in \cite[Lemma 3.3]{MasmoudiBelkacemZhao2022Inv}.
More precisely, it emphasizes the well-separation of critical times (including the mild resonant time interval).
\begin{lemma}[\cite{MasmoudiBelkacemZhao2022Inv}] \label{Scenarios}
    Let $\eta$ and $\xi$ be such that there exists a number $\alpha\geq 1$ such that $\alpha^{-1}|\xi|\leq |\eta|\leq \alpha |\xi|$ and let $k$ and $n$ be such that $t\in \tilde{\textup{I}}_{k,\eta}\cap \tilde{\textup{I}}_{n,\xi}$, where $k\approx n$. Then at least one of the following scenarios occurs:
    \begin{enumerate}[(a)]
        \item $k=n$ and $t \in \textup{I}_{k,\eta} \cap \textup{I}_{k,\xi}$(almost the same interval),
        \item $k=n$ and $\big|t-\frac{\eta}{k}\big|\geq \frac{1}{10\alpha}\frac{|\eta|}{k^3}$ and $\big|t-\frac{\xi}{k}\big|\geq \frac{1}{10\alpha}\frac{|\xi|}{k^3}$(away from resonance),
        \item $k=n$  and $|\xi-\eta| \gtrsim_{\alpha} \frac{|\eta|}{k^2}$(well-separated),
        \item $\big| t-\frac{\eta}{k} \big|\geq \frac{1}{10 \alpha}  \frac{|\eta|}{k^2}$ and $|t-\frac{\xi}{n}|\geq \frac{1}{10 \alpha}  \frac{|\xi|}{n^2}$(away from resonance),
        \item $|\xi-\eta|\gtrsim_{\alpha} \frac{|\eta|}{|n|}$(well-separated).
\end{enumerate}
 Additionally, if $t\in \textup{I}_{k,\eta} \cap \textup{I}_{n,\xi}$, then at least of the following things holds:
 \begin{enumerate}[(1)]
  \item $k=n$(almost the same interval),
  \item $\big|t-\frac{\eta}{k}\big|\gtrsim_{\alpha} \frac{|\eta|}{k^3}$ and $\big|t-\frac{\xi}{n}\big|\gtrsim_{\alpha}\frac{|\xi|}{n^3}$(away from resonance),
  \item $\big|\xi-\eta \big|\gtrsim_{\alpha} \frac{|\eta|}{n^2}$(well-separated).
 \end{enumerate}
\end{lemma}

The following two lemmas below ultimately predict the growth of high frequencies which signal the loss of Gevrey-3 regularity. 
\begin{lemma}[\cite{MasmoudiBelkacemZhao2022Inv}]\label{lem-total growth}
    Suppose that $|\eta|>1$. Then there exists $\mu=60(1+2C_1)$ such that 
    \begin{equation}\label{growth of Theta}
          \dfrac{\Theta_k(2\eta,\eta)}{\Theta_k(0,\eta)}=\dfrac{1}{\Theta_k(0,\eta)}=\dfrac{1}{\Theta_k(t_{E(\eta^{\frac{1}{3}}),\eta},\eta)}\approx \dfrac{e^{\frac{\mu}{20}|\eta|^{\frac{1}{3}}}}{\eta^{\frac{\mu}{120}}},
    \end{equation}
    where $C_1$ is the constant the same as in \eqref{def-ThetaNR}.
\end{lemma}

\begin{lemma}[\cite{MasmoudiBelkacemZhao2022Inv}]\label{lem-growth-Lambda}
For $|\eta|>1$, it holds that
\[
1\le\fr{1}{\Lambda(t,\eta)}\le e^{\fr{3\pi}{20}|\eta|^{\fr13}}.
\]
\end{lemma}

By the definitions of $\Theta_{\rm NR}(t,\eta)$ and $\Theta_{\rm R}(t,\eta)$, we immediately have the following lemma.
\begin{lemma} \label{estimate on ratio of theta}
    For any $t\in\textup{I}_{k,\eta}$ and $t\geq E(|\eta|^{\frac{1}{3}})$, it is true that
    \[ \dfrac{\partial_t\Theta_{\NR}(t,\eta)}{\Theta_{\NR}(t,\eta)} \approx \dfrac{1}{1+|t-\frac{\eta}{k}|}\approx \dfrac{\partial_t\Theta_{\R}(t,\eta)}{\Theta_{\R}(t,\eta)}.
    \]
    If $t\in\textup{I}_{k,\eta}$, we have
    \[
     \dfrac{\Theta_{\NR}(t,\eta)}{\Theta_{\R}(t,\eta)} \approx \dfrac{\frac{|\eta|}{|k|^3}}{1+|t-\frac{\eta}{k}|}.
    \]
\end{lemma}
\begin{lemma}[\cite{MasmoudiBelkacemZhao2022Inv}]\label{estimate of Theta nonresonant}
    For any $\eta,\xi \in \mathbb{R}$ and $t\geq 1$, we have

    \[ \dfrac{\Theta_{\NR}(t,\xi)}{\Theta_{\NR}(t,\eta)} \lesssim e^{\mu|\eta-\xi|^{\frac{1}{3}}}. 
    \]
\end{lemma}

\begin{lemma}[\cite{MasmoudiBelkacemZhao2022Inv}]\label{estimate of Lambda fraction}
    For any $\eta,\xi \in \mathbb{R}$ and $t\geq 1$, the following inequality holds
    \[ \dfrac{\Lambda(t,\xi)}{\Lambda(t,\eta)} +\dfrac{\Lambda(t,\eta)}{\Lambda(t,\xi)}\lesssim e^{C_0|\eta-\xi|^{\frac{1}{3}}}.
    \]
\end{lemma}

\begin{lemma}\label{ratio of J}
Suppose that $((k,\eta), (l,\xi)) \in \mathfrak{A}$, then the following inequality holds for $t\geq 1$
\begin{align}\label{ratio-J}
\dfrac{\mathcal{J}_k(t,\eta)}{\mathcal{J}_l(t,\xi)} &\lesssim \bigg(\mathds{1}_{\textup{A}}\dfrac{|\eta|}{|k|^3(1+|t-\frac{\eta}{k}|)}
+\mathds{1}_{\textup{B}}\frac{|l|^3\Big(1+|t-\frac{\xi}{l}|\Big)}{|\xi|} 
+\mathds{1}_{({\textup{A}}\cup {\textup{B}})^c}\bigg)e^{3\mu|k-l,\eta-\xi|^{\frac{1}{3}}},
\end{align}
where $\textup{A}=\{t \in \textup{I}_{k,\eta}\cap \textup{I}^c_{l,\xi}, k \neq l\}\cap \mathfrak{A}$,  $\textup{B}=\{t \in \textup{I}^c_{k,\eta}\cap  \textup{I}_{l,\xi}, \}\cap \mathfrak{A}$ and $(\textup{A}\cup {\textup{B}})^c=\mathfrak{A}\setminus (\textup{A}\cup \textup{B})$.
\end{lemma}
\begin{proof}
Before proceeding any further, it is worth pointing out that the facts $((k,\eta),(l,\xi))\in \mathfrak{A}$ and $t\in{\rm I}_{l,\xi}$ imply that
\begin{align}\label{eta comparable to xi}
 \fr{5}{8}|\xi|\le|\eta|\le\fr{11}{8}|\xi|.
\end{align}
    Recalling \eqref{Multiplier components}, and using the elementary inequality $\left||x|^{\fr{1}{3}}-|y|^{\fr{1}{3}}\right|\le |x-y|^{\fr{1}{3}}$, we have
    \begin{align}\label{ratio-J-1}
\fr{\mathcal{J}_k(t,\eta)}{\mathcal{J}_l(t,\xi)}\les& e^{\mu |\eta-\xi|^{\fr{1}{3}}}\fr{\Theta_l(t,\xi)}{\Theta_k(t,\eta)}+e^{\mu|k-l|^{\fr{1}{3}}}.
    \end{align}
If $t\in \textup{I}_{k,\eta}\cap \textup{I}^c_{l,\xi}$, by \eqref{Theta R and Theta NR}, \eqref{DE for Theta} and Lemma \ref{estimate of Theta nonresonant}, one deduces that
\begin{align}\label{ratio-Theta}
\fr{\Theta_l(t,\xi)}{\Theta_k(t,\eta)}\mathds {1}_{\rm A}= \fr{\Theta_{\rm NR}(t,\xi)}{\Theta_{\rm NR}(t,\eta)}\fr{\Theta_{\rm NR}(t,\eta)}{\Theta_{\rm R}(t,\eta)}\mathds {1}_{\rm A}\les e^{\mu|\eta-\xi|^{\frac{1}{3}}}\dfrac{\mathds {1}_{\rm A}|\eta|}{|k|^3(1+|t-\frac{\eta}{k}|)}.
\end{align}
Then \eqref{ratio-J} holds since $1\les \fr{\mathds{1}_{\rm A}|\eta|}{|k|^3(1+|t-\fr{\eta}{k}|)}$ and the last term on the right hand side of \eqref{ratio-J-1} can be bounded by $e^{\mu|k-l,\eta-\xi|^{\frac{1}{3}}}\dfrac{\mathds {1}_{\rm A}|\eta|}{|k|^3(1+|t-\frac{\eta}{k}|)}$.

Similar to \eqref{ratio-Theta}, if $t\in \textup{I}^c_{k,\eta}\cap  \textup{I}_{l,\xi}$, we have
\begin{align}
\fr{\Theta_l(t,\xi)}{\Theta_k(t,\eta)}\mathds {1}_{\rm B}\nn= \fr{\Theta_{\rm NR}(t,\xi)}{\Theta_{\rm NR}(t,\eta)}\fr{\Theta_{\rm R}(t,\xi)}{\Theta_{\rm NR}(t,\xi)}\mathds {1}_{\rm B}\les e^{\mu|\eta-\xi|^{\frac{1}{3}}} \mathds{1}_{\textup{B}}\frac{|l|^3\Big(1+|t-\frac{\xi}{l}|\Big)}{|\xi|}.
\end{align}
If $|k|\ge\fr{1}{8}|\xi|$, then the last term on the right hand side of \eqref{ratio-J-1} can be bounded as follows:
\[
e^{\mu |k-l|^{\fr{1}{3}}}\les e^{\mu |k-l|^{\fr{1}{3}}}\fr{|k|}{|l|}\fr{|l|}{|\xi|}\les \langle k-l\rangle e^{\mu |k-l|^{\fr{1}{3}}}\frac{|l|^3\Big(1+|t-\frac{\xi}{l}|\Big)}{|\xi|}.
\]
Then \eqref{ratio-J} holds immediately. On the other hand, if $|k|<\fr{1}{8}|\xi|$, one cannot use the upper bound on the right-hand side of \eqref{ratio-J-1} anymore. Indeed, it is the place where the factor $e^{\mu |\xi|^{\fr{1}{3}}}$ in the definition of $\mathcal{J}_l(t,\xi)$ plays its role. More precisely, now \eqref{ratio-J-1} can be replaced by
\begin{align}
\fr{\mathcal{J}_k(t,\eta)}{\mathcal{J}_l(t,\xi)}\nn\les& e^{\mu |\eta-\xi|^{\fr{1}{3}}}\fr{\Theta_l(t,\xi)}{\Theta_k(t,\eta)}+e^{\mu (|k|^{\fr{1}{3}}-\fr{1}{2}|\xi|^{\fr{1}{3}})}e^{-\fr12|\xi|^{\fr{1}{3}}},
\end{align}
and hence \eqref{ratio-J} holds.

For the rest cases, we write $({\rm A\cup B})^c={\rm D}_1\cup {\rm D}_2\cup {\rm D}_3\cup{\rm D}_4$, where
\begin{align*}
{\rm D}_1=&\{t\in {\rm I}_{k,\eta}^c\cap {\rm I}_{l,\xi}^c \}\cap\frak{A},\quad
{\rm D}_2=\{ t\in {\rm I}_{k,\eta}\cap {\rm I}_{l,\xi}, k=l \}\cap\frak{A},\\
{\rm D}_3=&\{ t\in {\rm I}_{k,\eta}\cap {\rm I}_{l,\xi}, k\ne l \}\cap\frak{A},\quad
{\rm D}_4=\{ t\in {\rm I}_{k,\eta}\cap {\rm I}_{l,\xi}^c, k=l  \}\cap\frak{A}.
\end{align*}
If $t\in {\rm I}_{k,\eta}^c\cap {\rm I}_{l,\xi}^c$, from \eqref{defintion of Theta k}, \eqref{ratio-J-1} and Lemma \ref{estimate of Theta nonresonant}, we find that
\begin{align}
\fr{\mathcal{J}_k(t,\eta)}{\mathcal{J}_l(t,\xi)}\mathds{1}_{{\rm D}_1}\nn\les& e^{\mu |\eta-\xi|^{\fr{1}{3}}}\fr{\Theta_{\rm NR}(t,\xi)}{\Theta_{\rm NR}(t,\eta)}+e^{\mu |k-l|^{\fr{1}{3}}}\les e^{2\mu |k-l,\eta-\xi|^{\fr{1}{3}}}.
\end{align}
If $t\in {\rm I}_{k,\eta}\cap {\rm I}_{l,\xi}$ with $k=l$, thanks to \eqref{eta comparable to xi}, we have
\begin{align}
\fr{\mathcal{J}_k(t,\eta)}{\mathcal{J}_l(t,\xi)}\mathds{1}_{{\rm D}_2}\nn\les& e^{\mu |\eta-\xi|^{\fr{1}{3}}}\fr{\Theta_{\rm R}(t,\xi)}{\Theta_{\rm NR}(t,\xi)}\fr{\Theta_{\rm NR}(t,\xi)}{\Theta_{\rm NR}(t,\eta)}\fr{\Theta_{\rm NR}(t,\eta)}{\Theta_{\rm R}(t,\eta)}{\mathds 1}_{t\in{\rm I}_{k,\eta}\cap{\rm I}_{k,\xi}}+e^{\mu |k-l|^{\fr{1}{3}}}\\
\nn\les&e^{2\mu|\eta-\xi|^{\fr{1}{3}}}\fr{|\eta|}{|\xi|}\fr{1+|t-\fr{\xi}{k}|}{1+|t-\fr{\eta}{k}|}+e^{\mu |k-l|^{\fr{1}{3}}}\les e^{3\mu|k-l,\eta-\xi|^{\fr{1}{3}}}.
\end{align}
If $t\in {\rm I}_{k,\eta}\cap {\rm I}_{l,\xi}$ with $k\neq l$, arguing as in \eqref{ratio-Theta}, we are led to
\begin{align}
\fr{\mathcal{J}_k(t,\eta)}{\mathcal{J}_l(t,\xi)}\mathds{1}_{{\rm D}_3}\nn\les& e^{\mu |\eta-\xi|^{\fr{1}{3}}}\fr{\Theta_{\rm R}(t,\xi)}{\Theta_{\rm NR}(t,\xi)}\fr{\Theta_{\rm NR}(t,\xi)}{\Theta_{\rm NR}(t,\eta)}\fr{\Theta_{\rm NR}(t,\eta)}{\Theta_{\rm R}(t,\eta)}{\mathds 1}_{t\in{\rm I}_{k,\eta}\cap{\rm I}_{l,\xi}}+e^{\mu |k-l|^{\fr{1}{3}}}\\
\nn\les&e^{2\mu|\eta-\xi|^{\fr{1}{3}}}\fr{|\eta|}{|k|^3}\fr{1}{1+|t-\fr{\eta}{k}|}+e^{\mu |k-l|^{\fr{1}{3}}}.
\end{align}
Thanks to  \eqref{eta comparable to xi}, now (2) or (3) holds in Lemma \ref{Scenarios}. Accordingly,
\[
\fr{|\eta|}{|k|^3}\fr{{\mathds 1}_{\rm D_3}}{1+|t-\fr{\eta}{k}|}\les\langle \eta-\xi\rangle,
\]
then \eqref{ratio-J} follows immediately. If $t\in {\rm I}_{k,\eta}\cap {\rm I}_{l,\xi}^c$ with $k=l$, we still have
\begin{align}
\fr{\mathcal{J}_k(t,\eta)}{\mathcal{J}_l(t,\xi)}\mathds{1}_{{\rm D}_4}
\nn\les&e^{2\mu|\eta-\xi|^{\fr{1}{3}}}\fr{|\eta|}{|k|^3}\fr{\mathds{1}_{{\rm D}_4}}{1+|t-\fr{\eta}{k}|}+e^{\mu |k-l|^{\fr{1}{3}}}.
\end{align}
On the other hand, it is easy to verify that if $t\in{\rm I}_{k,\xi}^c$ and $k\ne0$, then
\begin{align}\label{up-nonresonant}
\fr{|\xi|\mathds{1}_{t\in{\rm I}_{k,\xi}^c}}{|k|^3(1+|t-\fr{\xi}{k}|)}\les1.
\end{align}
Therefore,
\[
\fr{\mathcal{J}_k(t,\eta)}{\mathcal{J}_l(t,\xi)}\mathds{1}_{{\rm D}_4}
\nn\les e^{2\mu|\eta-\xi|^{\fr{1}{3}}}\fr{|\eta|-|\xi|+|\xi|}{|k|^3}\fr{\mathds{1}_{t\in{\rm I}_{k,\xi}^c}}{(1+|t-\fr{\xi}{k}|)}\fr{1+|t-\fr{\xi}{k}|}{1+|t-\fr{\eta}{k}|}+e^{\mu |k-l|^{\fr{1}{3}}}\les e^{3\mu |k-l,\eta-\xi|^{\fr{1}{3}}}.
\]
This completes the proof of Lemma \ref{ratio of J}.
\end{proof}

\begin{lemma}[\cite{MasmoudiBelkacemZhao2022Inv}] \label{estimate on fraction of J}
    For all $t\leq \frac{1}{2}\min\{|\xi|^{\frac{2}{3}},|\eta|^{\frac{2}{3}}\}$, the following inequality holds
    \[
    \bigg|\dfrac{\J_k(t,\eta)}{\J_l(t,\xi)}-1\bigg|\lesssim \dfrac{\JB{k-l,\xi-\eta}}{(|k|+|l|+|\eta|+|\xi|)^{\frac{2}{3}}} e^{3\mu |k-l,\eta-\xi|^{\frac{1}{3}}}.
    \]
\end{lemma}

\begin{lemma}[\cite{MasmoudiBelkacemZhao2022Inv}] \label{lem-com-M}
    For all $t\leq \frac{1}{2}\min\{|\xi|^{\frac{1}{3}},|\eta|^{\frac{1}{3}}\}$, the following inequality holds
    \[
    \bigg|\dfrac{\M_k(t,\eta)}{\M_l(t,\xi)}-1\bigg|\lesssim \dfrac{\JB{k-l,\xi-\eta}}{(|k|+|l|+|\eta|+|\xi|)^{\frac{2}{3}}} e^{C_0 |k-l,\eta-\xi|^{\frac{1}{3}}}.
    \]
\end{lemma}

\bigbreak
\noindent{\bf Acknowledgments: }
WZ extends his gratitude to Professor Yao Yao for the insightful discussion pertaining to the instability mechanism associated with the Stokes transport equation. 
RZ is partially supported by NSF of China under  Grants 12222105. Part of this work was done when RZ was visiting NYU Abu Dhabi. He appreciates the hospitality of NYUAD.

\noindent{\bf Conflict of interest:}  We confirm that we do not have any conflict of interest. 

\noindent{\bf Data availibility:} The manuscript has no associated data.

\begin{center}
\bibliographystyle{siam}
\bibliography{Ref.bib}
\end{center}
\end{document}